\documentclass[11pt]{amsart}

\usepackage[margin=1.27in]{geometry}                
\usepackage{fancyhdr}
\usepackage[fixamsmath]{mathtools}
\mathtoolsset{centercolon}

\usepackage{tabularx,booktabs}
\newcolumntype{Y}{>{\centering\arraybackslash}X} 

\usepackage{enumitem}

\usepackage{caption}

\usepackage{etoolbox}

\usepackage{tikz}
\usetikzlibrary{calc,patterns,intersections}
\usepackage{tikz-cd}

\usepackage{graphicx}
\usepackage{amssymb}
\usepackage{amsthm}
\usepackage{epstopdf}
\usepackage{mathrsfs}
\usepackage{amsthm}
\usepackage{multirow}
\usepackage{empheq}
\usepackage{color}
\usepackage[normalem]{ulem}
\newtheorem{theorem}{Theorem}
\newtheorem{lemma}[theorem]{Lemma}
\newtheorem{proposition}[theorem]{Proposition}

\theoremstyle{remark}
\newtheorem*{remark}{Remark}

\theoremstyle{definition}

\theoremstyle{definition}
\newtheorem*{example}{Example}

\newcommand{\I}{{\scriptscriptstyle I}} 
\newcommand{\J}{{\scriptscriptstyle J}}
\newcommand{\K}{{\scriptscriptstyle K}} 
\renewcommand{\L}{{\scriptscriptstyle L}} 
\newcommand{\M}{{\scriptscriptstyle M}} 
\newcommand{\N}{{\scriptscriptstyle N}} 

\newcommand{\A}{{\scriptscriptstyle A}} 
\newcommand{\B}{{\scriptscriptstyle B}}
\newcommand{\C}{{\scriptscriptstyle C}}
\newcommand{\D}{{\scriptscriptstyle D}}

\newcommand{\pt}{\hspace{1pt}} 
\newcommand{\hp}{\hspace{0.5pt}} 

\newcommand{\npt}{\hspace{-1pt}} 
\newcommand{\nhp}{\hspace{-0.5pt}} 

\renewcommand{\Re}{\pt\mbox{Re}\pt} 
\renewcommand{\Im }{\pt\mbox{Im}\pt}

\newcommand{\sprod}{\npt\cdot\nhp}   
\newcommand{\vprod}{\npt\times\npt} 

\newcommand{\verteq}{\rotatebox{90}{$\,=$}}

\newcolumntype{L}[1]{>{\raggedright\let\newline\\\arraybackslash\hspace{0pt}}m{#1}}
\newcolumntype{C}[1]{>{\centering\let\newline\\\arraybackslash\hspace{0pt}}m{#1}}
\newcolumntype{R}[1]{>{\raggedleft\let\newline\\\arraybackslash\hspace{0pt}}m{#1}}

\makeatletter
\pretocmd{\chapter}{\addtocontents{toc}{\protect\addvspace{15\p@}}}{}{}
\pretocmd{\section}{\addtocontents{toc}{\protect\addvspace{5\p@}}}{}{}
\makeatother

\makeatletter
\renewcommand{\tocsection}[3]{%
  \indentlabel{\@ifnotempty{#2}{\bfseries\ignorespaces#1 #2\quad}}\bfseries#3}
\renewcommand{\tocsubsection}[3]{%
  \indentlabel{\@ifnotempty{#2}{\ignorespaces#1 #2\quad}}#3}

\newcommand\@dotsep{4.5}
\def\@tocline#1#2#3#4#5#6#7{\relax
  \ifnum #1>\c@tocdepth 
  \else
    \par \addpenalty\@secpenalty\addvspace{#2}%
    \begingroup \hyphenpenalty\@M
    \@ifempty{#4}{%
      \@tempdima\csname r@tocindent\number#1\endcsname\relax
    }{%
      \@tempdima#4\relax
    }%
    \parindent\z@ \leftskip#3\relax \advance\leftskip\@tempdima\relax
    \rightskip\@pnumwidth plus1em \parfillskip-\@pnumwidth
    #5\leavevmode\hskip-\@tempdima{#6}\nobreak
    \leaders\hbox{$\m@th\mkern \@dotsep mu\hbox{.}\mkern \@dotsep mu$}\hfill
    \nobreak
    \hbox to\@pnumwidth{\@tocpagenum{\ifnum#1=1\bfseries\fi#7}}\par
    \nobreak
    \endgroup
  \fi}
\AtBeginDocument{%
\expandafter\renewcommand\csname r@tocindent0\endcsname{0pt}
}
\def\l@subsection{\@tocline{2}{0pt}{2.5pc}{5pc}{}}
\makeatother

\title{An analogue of the Gibbons--Hawking Ansatz for quaternionic K\"ahler spaces}
\author{Radu A. Iona\c{s}}
\date{}                                           

\newcommand{\Address}{{
  \bigskip\bigskip 
  \sffamily
            Radu A. Iona\c{s} \\[3pt]
\indent C.\,N. Yang Institute for Theoretical Physics \\
\indent Stony Brook University, Stony Brook, NY 11794, U.S.A.

}}

\begin{document}

\begin{abstract}
We show that the geometry of $4n$-dimensional quaternionic K\"ahler spaces with a locally free $\smash{ \mathbb{R}^{n+1} }$-action admits a Gibbons--Hawking-like description based on the Galicki--Lawson notion of quaternionic K\"ahler moment map. This generalizes to higher dimensions a four-dimensional construction, due to Calderbank and Pedersen, of self-dual Einstein manifolds with two linearly independent commuting Killing vector fields. As an application, we use this new Ansatz to give an explicit equivariant completion of the twistor space construction of the local c-map proposed by Ro\v{c}ek, Vafa and Vandoren. 
\end{abstract}

\maketitle

\thispagestyle{fancy}\rhead{YITP-\hp SB\pt-18\hp-40}


\tableofcontents

\newpage

\section{Introduction}

The Gibbons--Hawking Ansatz \cite{Gibbons:1979zt} is a method of constructing four-dimensional hyperk\"ahler spaces with a tri-Hamiltonian locally free $\mathbb{R}$-action in terms of the solutions of a Bogomolny-type differential equation whose consistency condition is a Laplace equation. It is one of the fundamental constructions in hyperk\"ahler geometry, with vast applications ranging from the construction of gravitational instantons to serving as an asymptotic model or some other type of limit for various geometries and moduli spaces. The original method of Gibbons and Hawking was generalized in \cite{MR953820} by Pedersen and Poon to $4n$-dimensional hyperk\"ahler spaces with a tri-Hamiltonian locally free $\mathbb{R}^n$-action; in what follows we will refer to such spaces as \textit{extended} Gibbons--Hawking spaces. A key feature of these constructions is the fact that the rank of the action is such that the $n$ orbit parameters of the action together with the $3n$ components of its associated hyperk\"ahler moment map provide a complete set of coordinates for the hyperk\"ahler space. Geometrically, the hyperk\"ahler structure is defined on the total space of an $\mathbb{R}^n$-bundle over an open subset of the space  $\smash{ \mathbb{R}^n \npt\otimes \mathbb{R}^3 }$, where it is given explicitly in terms of a Higgs field and $\mathbb{R}^n$-connection 1-form satisfying a generalized form of the Bogomolny equation. This construction is general: any $4n$-dimensional hyperk\"ahler space with a locally free tri-Hamiltonian $\mathbb{R}^n$-action can be proven to arise locally in this way.

In this paper we show that a close analogue of this Ansatz exists as well in quaternionic K\"ahler setting. Recall that hyperk\"ahler and quaternionic K\"ahler manifolds are $4n$-dimensional Riemannian manifolds with the holonomy groups of their Levi-Civita connections given by $Sp(n)$ and $\smash{ Sp(n) \times_{\mathbb{Z}_2} \! Sp(1) }$, respectively, or subgroups thereof (here, we extend these definitions to include the pseudo-Riemannian variants of these geometries resulting from replacing the unitary quaternionic group $Sp(n)$ with any one of its non-compact versions; also, for \mbox{$n=1$}, the holonomy definition is rather too unrestrictive in the quaternionic K\"ahler case, and one usually discards it: the natural analogues of quaternionic K\"ahler manifolds in four dimensions are self-dual Einstein manifolds). Both types of geometries are Einstein. Hyperk\"ahler geometry, however, is essentially distinguished from more general quaternionic K\"ahler one by having zero scalar curvature. This distinction notwithstanding, in specific circumstances it is nonetheless useful to think of the former as the zero scalar curvature limit of the latter. Another, less obvious way in which the two geometries are related was discovered by Swann, who showed in \cite{MR1096180} that over any quaternionic K\"ahler manifold one can construct a quaternionic bundle with fiber $\smash{ \mathbb{H}^{\times} \npt /\mathbb{Z}_2 }$, where $\smash{ \mathbb{H}^{\times} }$ denotes the multiplicative group of non-zero quaternions, the total space of which carries a natural hyperk\"ahler structure as well as a homothetic vector field: that is to say, a hyperk\"ahler cone structure. The power of Swann's result resides in its transformative potential: it gives a canonical way to radically reformulate a quaternionic K\"ahler geometry problem as a hyperk\"ahler geometry one. 

And this is precisely the path that we take in our search for a quaternionic K\"ahler analogue of the extended Gibbons--Hawking Ansatz: rather than ask the question directly about quaternionic K\"ahler spaces, we ask instead the corresponding question about hyperk\"ahler cones. Namely, we ask under what conditions an extended Gibbons--Hawking space possesses a hyperk\"ahler cone structure. For any extended Gibbons--Hawking space of real dimension \mbox{$4n+4$} (this dimension is chosen for later convenience), the base of its \mbox{$\mathbb{R}^{n+1}$-fibration} is, as we have stated above, an open subset of the space $\smash{ \mathbb{R}^{n+1} \npt\otimes \mathbb{R}^3 }$, which can be viewed as the space of configurations of \mbox{$n+1$} distinguishable points in $\mathbb{R}^3$. On this space, one can always define a natural quaternionic action by considering the rigid rotations and simultaneous scalings of such configurations which leave a given point, say, the origin, fixed. These transformations form indeed a quaternionic group since \mbox{$SO(3) \otimes \mathbb{R}^+ \npt  \cong \mathbb{H}^{\times} \npt /\mathbb{Z}_2$}. So then, a more specific interrogation would be when can this \mbox{$\mathbb{H}^{\times} \npt /\mathbb{Z}_2$}-action on the base of the \mbox{$\mathbb{R}^{n+1}$-fibration} be lifted to a genuine hyperk\"ahler cone structure on the total space, that is, on the extended Gibbons--Hawking space? This question is answered in Proposition~\ref{GH-HKC}, which states that a necessary and sufficient condition for that to happen is for the Higgs field of the extended Gibbons--Hawking space to be invariant at rigid rotations and scale with a certain weight under simultaneous scalings of the configurations. What remains now is to match this description of the hyperk\"ahler cone, in which the \mbox{$\mathbb{R}^{n+1}$}-bundle structure is manifest, to the Swann description, centered instead on the \mbox{$\smash{ \mathbb{H}^{\times} \npt /\mathbb{Z}_2 }$}-bundle structure. This is facilitated by a certain condensed quaternionic reformulation of the extended Gibbons--Hawking formulas that we introduce in \mbox{\S\,\ref{GT-quatern}}. It ultimately yields an explicit description of the geometry of the base of the Swann bundle, which is in this case a \mbox{$4n$-dimensional} quaternionic K\"ahler geometry with an inherited locally free isometric \mbox{$\mathbb{R}^{n+1}$-action}. 

The emerging picture is summarized in intrinsic terms in subsection~\ref{QK-Ansatz}. Its main features are as follows: The quaternionic K\"ahler structure is defined on the total space of an $\mathbb{R}^{n+1}$-bundle over the space of inequivalent (with respect to rigid rotations and simultaneous scalings) non-degenerate configurations of distinguishable $n+1$ points in $\mathbb{R}^3$, which we denote by $\smash{ \textrm{Im}\mathbb{HP}^{\pt n} }$ (and define precisely in subsection~\ref{ssec:ImHP^n}). A natural set of coordinates is given by the orbit parameters of the $\mathbb{R}^{n+1}$-action together with this action's associated Galicki--Lawson-type moment maps, multiplied by a scale. The latter can also be understood as inhomogeneous local coordinates on $\smash{ \textrm{Im}\mathbb{HP}^{\pt n} }$. The quaternionic K\"ahler metric, 2-forms and $SO(3)$ connection 1-forms are given explicitly in these coordinates in terms of a \textit{reduced} Higgs field and \mbox{$\mathbb{R}^{n+1}$-connection} 1-form satisfying a set of partial differential equations on $\smash{ \textrm{Im}\mathbb{HP}^{\pt n} }$ similar to the Bogomolny equations from the hyperk\"ahler case. These field equations admit also a dual formulation, a remarkable consequence of which is that the local quaternionic K\"ahler structure turns out to be  completely determined by a single real-valued function on $\smash{ \textrm{Im}\mathbb{HP}^{\pt n} }$ satisfying a set of linear partial differential constraints. This function is closely related to the hyperk\"ahler potential of the Swann bundle. (Swann bundles have the distinctive feature among hyperk\"ahler spaces that all three elements of a defining triplet of 2-forms\,---\,see \textit{e.g.}~Theorem~\ref{HK-criterion}\,---\,can be derived, as K\"ahler forms with respect to their corresponding complex structures, from the \textit{same} K\"ahler potential.) Very importantly, like in the hyperk\"ahler case, this picture is general: \textit{any} $4n$-dimensional quaternionic K\"ahler space with a locally free isometric $\mathbb{R}^{n+1}$-action arises locally in this way.

The main body of the paper is organized into five sections. In \mbox{\textbf{section~\ref{sec:HK_QK}}} we give a \textit{non-Riemannian} characterization of both hyperk\"ahler and quaternionic K\"ahler manifolds, in a spirit similar to that of \textit{exterior differential systems} \cite{MR1083148}. That is, in each case we formulate a criterion in which the metric is not among the fundamental objects defining the geometry but, rather, is a derived, composite object. The building blocks of either geometry may be considered instead to be triplets of non-degenerate and pointwise linearly independent 2-forms: globally defined in the hyperk\"ahler case, and locally defined in the quaternionic K\"ahler one. Such a triplet is required to satisfy two conditions: an algebraic condition, and an exterior differential one (Theorems~\ref{HK-criterion} and \ref{QK-criterion-1}). In the hyperk\"ahler case the latter condition is due to Hitchin \cite{MR887284} and requires that the three \mbox{2-forms} be symplectic. In the quaternionic K\"ahler case a corresponding condition was found by Alekseevsky, Bonan and Marchiafava \cite{MR1376149}. These conditions do not constrain the triplets uniquely. In the hyperk\"ahler case one has in fact a half 3-sphere's worth of equivalent triplets to choose from (the triplets can be rotated by $SO(3)$ transformations, whose group manifold is topologically $S^3/\mathbb{Z}_2$). In the quaternionic K\"ahler case, on the other hand, one regards the triplets as local frames for an $SO(3)$-bundle. The proofs we give to these criterions are new; we base them in both cases on the successive  application of two key identities, which we collect together in Lemma~\ref{ab}. We show also that \mbox{Alekseevsky \textit{et al.}'s} condition can be equivalently replaced by a Hermitian--Einstein-type condition for the curvature 2-form of the associated principal $SO(3)$-bundle, which takes as well the form of an exterior differential condition. This gives us the very useful equivalent quaternionic K\"ahler criterion of Theorem~\ref{QK-criterion-2}. In the remainder of the section we describe how certain symmetry conditions\,---\,namely, the Killing condition and what we call the homothetic Euler condition, which is the basic symmetry of hyperk\"ahler cones\,---\,fit into this framework. In spite of the perhaps unusual point of view, the material discussed in this preliminary section is for the most part known, and readers accustomed to the basic concepts of hyperk\"ahler and quaternionic K\"ahler geometry and interested mainly in the headline subject of the paper can safely skip ahead, to return maybe only for  occasional references.  

In \mbox{\textbf{section~\ref{sec:Sw_bdls}}} we review in detail Swann's quaternionic bundle construction. We clarify among other things the role of the hypercomplex structures on the quaternionic fiber, and review how Killing symmetries fit into this picture as well.

In \mbox{\textbf{section~\ref{sec:QK_analogue}}}, after many preparations, we finally address in full force the issue of finding a quaternionic K\"ahler analogue of the extended Gibbons--Hawking Ansatz, following the strategy exposed above. In the last part of the section we then specialize to four dimensions (\textit{i.e.}~\mbox{$n=1$}). The space $\smash{ \textrm{Im}\mathbb{HP}^{\pt 1} }$ is isomorphic to the complex upper half-plane, and we find that our formulas yield in this case precisely the description given by Calderbank and Pedersen in \cite{MR1950174} to the four-dimensional avatars of the quaternionic K\"ahler spaces we consider, which are self-dual Einstein spaces with two linearly independent commuting Killing vector fields. In this sense, therefore, our results can be thought of as a natural higher-dimensional generalization of those of Calderbank and Pedersen. 

In \mbox{\textbf{section~\ref{sec:HK-LT}}} we translate the description of hyperk\"ahler cones from the Gibbons--Hawking-type framework used up until now to the language of the Legendre transform approach of Lindstr\"om and Ro\v{c}ek \cite{Lindstrom:1983rt, MR877637}. This is useful particularly when one wants to solve the field equations by means of twistor methods. 

In \mbox{\textbf{section~\ref{sec:c-map}}}, as an application of our newfound Ansatz, we give a thorough and explicit account of the twistor construction of a quaternionic K\"ahler metric found in \cite{Ferrara:1989ik} by Ferrara and Sabharwal, in connection with the so-called \textit{local c-map} construction from string theory and supergravity. The twistor approach to this metric was initiated in \cite{Rocek:2005ij, Rocek:2006xb} by Ro\v{c}ek, Vafa and Vandoren, and developed further in \cite{Neitzke:2007ke}. The idea is that, if one goes from the Ferrara--Sabharwal space six dimensions higher to the twistor space (in the sense of \cite{MR877637}) of its Swann bundle, there the quaternionic K\"ahler metric information is encoded in a single\,---\,and simple\,---\,holomorphic symplectic gluing function. The challenge becomes then to show that the quaternionic K\"ahler metric can be retrieved from this holomorphic data. In a first stage, by making use of the general methods associated to the Legendre transform approach, one can come down two dimensions from the twistor space and determine from this function the geometry of the Swann bundle. This step has been generally already understood in the literature. The new results of section~\ref{sec:QK_analogue} can then be used to descend a further four dimensions to retrieve explicitly and in an equivariant manner the Ferrara--Sabharwal metric. This exercise clarifies a number of issues and cements the basis of the twistor-theoretic understanding of this important construction.

\section{Hyperk\"ahler vs.~quaternionic K\"ahler geometry} \label{sec:HK_QK}

\subsection{General aspects} \label{ssec:HK_QK_gen}

\subsubsection{}

In Berger's list of special holonomy manifolds, hyperk\"ahler and quaternionic K\"ahler manifolds are $4n$-dimen\-sio\-nal Riemannian manifolds with the holonomy group of their Levi-Civita connection a subgroup of the unitary quaternionic group $Sp(n)$, respectively the group $\smash{ Sp(n) \times_{\mathbb{Z}_2} \! Sp(1) }$.\footnote{\,An explicit embedding of $\smash{ Sp(n) \times_{\mathbb{Z}_2} \! Sp(1) }$ as a Lie subgroup of $SO(4n)$, the holonomy group of a general orientable $4n$-dimensional Riemannian manifold, is given as follows: regard $\mathbb{R}^{4n}$ as $\mathbb{H}^n$ and consider the linear endomorphisms on the latter given by the simultaneous left and right actions \mbox{$x_{\scriptscriptstyle I} \mapsto A_{\scriptscriptstyle IJ}\hp x_{\scriptscriptstyle J} \hp u^{-1}$} for any \mbox{$(x_{\scriptscriptstyle I})_{\scriptscriptstyle I = 1, \dots, n} \in \mathbb{H}^n$}, quaternionic unitary matrix $(A_{\scriptscriptstyle IJ})_{\scriptscriptstyle I, J = 1,\dots, n}$, and unit quaternion $u$; the summation convention over repeated indices is understood. This preserves the quaternionic norm on $\mathbb{H}^n$, and so the corresponding Euclidean norm on $\mathbb{R}^{4n}$, which thus exhibits it as an element of $O(4n)$, and in fact of $SO(4n)$. The $\mathbb{Z}_2$-quotient comes from the invariance under a simultaneous change of sign of both $A_{\scriptscriptstyle IJ}$ and $u$. Note also that 
for \mbox{$n > 1$}, \mbox{$Sp(n) \times_{\mathbb{Z}_2} \! Sp(1)$} is a maximal subgroup of $SO(4n)$ \cite{MR0244913}.}~Here we will distinguish between the two notions and assume, moreover, that hyperk\"ahler manifolds are \textit{not} subsummed to the quaternionic K\"ahler ones. There exist also pseudo-Riemannian versions of these manifolds which are obtained by replacing the group $Sp(n)$ with one of its non-compact versions $Sp(p,q)$ with $p+q=n$. In these notes we will assume that the notions of hyperk\"ahler and quaternionic K\"ahler manifolds include these variations as well. In four dimensions (that is, for $n=1$), due to the isomorphism $\smash{ Sp(1) \times_{\mathbb{Z}_2} \npt Sp(1) \cong SO(4) }$ the above definition of quaternionic K\"ahler manifolds is rather unrestrictive as it encompasses all orientable four-manifolds. In this case the holonomy condition is usually replaced by a curvature condition, and the natural four-dimensional analogues of quaternionic K\"ahler manifolds are considered to be the self-dual Einstein manifolds. In fact, one can show that quaternionic K\"ahler manifolds are always Einstein \cite{MR0231313}, while hyperk\"ahler ones are Einstein with vanishing scalar curvature\,---\,that is, Ricci-flat. 

The holonomy definition can be shown to be equivalent in the quaternionic K\"ahler case to the existence of a three-dimensional subbundle of the bundle of endomorphisms of the tangent bundle $\text{End}(TM)$ which is locally spanned by three almost complex structures \mbox{$I_1$, $I_2$, $I_3$}, forming the algebra of imaginary quaternions ($\smash{ I_1^2=I_2^2=I_3^2 = I_1I_2I_3 = -1 }$), 
and is preserved by the Levi-Civita connection $\nabla$. That is to say, there exist locally defined 1-forms $\theta_1$, $\theta_2$, $\theta_3 \in T^*\npt M$ such that
\begin{equation} \label{cd_Ii-QK}
\nabla_{\! X} I_i  = - \pt 2 \pt \varepsilon_{ijk} \pt \theta_j(X) \hp I_k
\end{equation}
for any vector field \mbox{$X \in TM$}, where the indices $i,j,k$ run over the values $1,2,3$, $\varepsilon_{ijk}$ denotes the anti-symmetric Levi-Civita symbol and the numerical factor is conventional. In four dimensions, the $\theta_i$ can be viewed as the self-dual part of the spin connection. 

In the hyperk\"ahler case, on the other hand, the holonomy definition is equivalent to the existence of a globally defined triplet of almost complex structures $I_1$, $I_2$, $I_3$ forming the algebra of imaginary quaternions, each element of which is individually preserved by the Levi-Civita connection:
\begin{equation} \label{cd_Ii-HK}
\nabla_{\! X} I_i  = 0
\end{equation}
for any vector field \mbox{$X \in TM$}. In contrast to the quaternionic K\"ahler setting, one can now show that the almost complex structures are in fact integrable and thus give rise to genuine complex structures. 

In either case, the metric $g$ can always be chosen to be simultaneously Hermitian with respect to each $I_i$, and by combining it with them one can form three 2-forms \mbox{$\omega_i(X,Y) = g(X,I_iY)$}, for any \mbox{$X,Y \in TM$}, which are defined locally in the quaternionic K\"ahler setting and globally in the hyperk\"ahler one. In four dimensions, these form a frame of the bundle of self-dual 2-forms. The above covariant differentiation properties readily translate for the 2-forms into 
\begin{equation} \label{cd_oi-QK}
\nabla_{\! X} \omega_i  = - \pt 2 \pt \varepsilon_{ijk} \pt \theta_j(X) \pt \omega_k
\end{equation}
in the quaternionic K\"ahler case and
\begin{equation}
\nabla_{\! X} \omega_i  = 0
\end{equation}
in the hyperk\"ahler one, for an arbitrary vector field \mbox{$X \in TM$}. Since the Levi-Civita connection is symmetric, these formulas imply further that
\begin{equation}
d \omega_i = - \pt 2 \pt \varepsilon_{ijk} \pt \theta_j \nhp \wedge \omega_k
\end{equation}
in the first case and
\begin{equation}
d\omega_i = 0
\end{equation}
in the second, where we now view $\omega_i$ and $\theta_i$ as differential forms. Note, however, that the reverse implication does not trivially hold. 

Assuming the dimension $n$ to be finite and the metric non-degenerate, since each $I_i$ is invertible (with inverse $-I_i$), it follows that each $\omega_i$, viewed as an element of $\smash{ \text{Hom}(T^*\npt M,TM) }$, is also invertible, with inverse $\smash{ \omega_i^{-1} \! \in \text{Hom}(TM,T^*\npt M) }$. A simple argument based on the quaternionic algebra shows then that we can write
\begin{equation}
I^{\phantom{.}}_1 = \omega_3^{\smash{-1}}\omega^{\phantom{1}}_2 \qquad\quad
I^{\phantom{.}}_2 = \omega_1^{\smash{-1}}\omega^{\phantom{1}}_3 \qquad\quad
I^{\phantom{.}}_3 = \omega_2^{\smash{-1}}\omega^{\phantom{1}}_1 \mathrlap{.}
\end{equation}

\subsubsection{}

In this approach, the three 2-forms are composite objects constructed from the metric and the almost complex structures. In what follows we will present an alternative approach in which this relationship is reversed, and \textit{they} rather than the latter are regarded as the fundamental building objects of the geometry. Our discussion will emphasize in particular the fact that the two geometries inform each other and can be treated in parallel, and that, although we have made a special point of distinguishing between them, it is often convenient in practice to think of hyperk\"ahler geometry as a limit case of quaternionic K\"ahler geometry, when the hallmark three-dimensional subbundle of the latter becomes trivial and the scalar curvature vanishes. 

The proofs of the subsequent two theorems rely crucially on the following technical lemma: 

\begin{lemma} \label{ab}
Let $M$ be a manifold with an almost complex structure $I$\,---\,that is, an endomorphism of the tangent bundle $TM$ such that \mbox{$I^2=-1$}\,---\,and let 
$$N_I(X,Y) = - \pt I^2[X,Y] + I([IX,Y] + [X,IY]) - [IX,IY]$$ 
for arbitrary vector fields \mbox{$X,Y \in TM$} be the associated Nijenhuis tensor. A real-valued 2-form $\omega$ on $M$ is
\begin{itemize}
\setlength\itemsep{0.4em}

\item[a)] of $(1,1)$ type with respect to $I$ if and only if 
$$\omega(X,IY) = \omega(Y,IX) \pt \stackrel{\smash[t]{\textup{def}}}{=} - \pt g(X,Y)$$
for any $X,Y \in TM$. In this case the following identity holds:
\begin{equation}
\omega(N_I(X,Y),Z) = d\omega(IX,IY,Z) - d\omega(X,Y,Z) + 2 \hp \nabla_{\! Z} \hp \omega(X,Y)
\end{equation}
where \smash{$\nabla$} is the Levi-Civita connection corresponding to the compatible metric $g$. 
\item[b)] of mixed $(2,0)+(0,2)$ type with respect to $I$ if and only if 
$$\omega(X,IY) = - \pt \omega(Y,IX) \pt \stackrel{\smash[t]{\textup{def}}}{=} - \pt \tilde{\omega}(X,Y)$$
for any $X,Y \in TM$. (The complex-valued 2-forms \mbox{$\omega + i\hp \tilde{\omega}$} and \mbox{$\omega - i\hp \tilde{\omega}$} are then of pure $(2,0)$ and $(0,2)$ type with respect to $I$, respectively.) In this case we have the identity
\begin{equation}
\omega(N_I(X,Y),Z) = d\omega(IX,IY,Z) - d\omega(X,Y,Z) + d\tilde{\omega}(IX,Y,Z) + d\tilde{\omega}(X,\mathrlap{IY,Z) .}
\end{equation}
\end{itemize}
\end{lemma} 

\begin{proof}
Part $(a)$ is essentially Proposition 4.6 from ch.\,IX, sec.\,4\vphantom{, p.\,148} of reference \cite{MR0238225} (the numerical factors differ due to differing normalization conventions for differential forms). Let us reproduce here the proof briefly. On one hand, we transform the last term on the right-hand side of the identity we want to prove as follows
{\allowdisplaybreaks
\begin{align}
\nabla_{\! Z} \hp \omega(X,Y) & = - \, g((\nabla_{\! Z} I)X,Y) \\[2pt]
& = - \, g(\nabla_{\! Z}(IX),Y) + g(I\hp \nabla_{\! Z} X,Y) \nonumber \\[2pt]
& = - \, g(\nabla_{\! Z}(IX),Y) - g(\nabla_{\! Z} X,IY) \nonumber 
\end{align}
}%
and then apply for each resulting term the Koszul formula for the Levi-Civita connection
\begin{align}
2 \pt g(\nabla_{\! X} Y,Z) & = X(g(Y,Z)) + Y(g(X,Z)) - Z(g(X,Y)) \\
& \ - \pt g([Y,Z],X) - g([X,Z],Y) + g([X,Y],Z) \mathrlap{.} \nonumber
\end{align}
On the other hand, for the remaining two terms on the right-hand side we use the identity
\begin{align} \label{d-2-form}
d\omega(X,Y,Z) & = X(\omega(Y,Z)) - Y(\omega(X,Z)) + Z(\omega(X,Y)) \\[1pt]
& \ - \pt \omega([Y,Z],X) + \omega([X,Z],Y) - \omega([X,Y],Z) \mathrlap{.} \nonumber
\end{align}
The result follows then by isolating the expression on the left-hand side and then cancelling the remaining terms by making use of the relationship between $\omega$ and $g$, and the properties which follow from it. 

The formula at part $(b)$ can be verified in a relatively more straightforward manner by applying the identity \eqref{d-2-form} to both the $d\omega$ and the $d\tilde{\omega}$ terms, and then systematically replacing in the resulting expression $\tilde{\omega}$ with $\omega$ by using some version of the defining relation for the former. 
\end{proof}

With this lemma in hand we can now return to our initial discussion. On the hyperk\"ahler side we have the following criterion, which is in essence a version of Lemma 6.8 of reference \cite{MR887284}:

\begin{theorem} \label{HK-criterion}
A finite-dimensional manifold $M$ is hyperk\"ahler if and only if it possesses a triplet of non-degenerate and pointwise linearly independent 2-forms \mbox{$\omega_1$, $\omega_2$, $\omega_3$} satisfying simultaneously:
\begin{itemize}
\setlength\itemsep{0.4em}

\item \underline{an algebraic condition}: if we define 
\begin{equation} 
I^{\phantom{.}}_1 = \omega_3^{\smash{-1}}\omega^{\phantom{1}}_2, \
I^{\phantom{.}}_2 = \omega_1^{\smash{-1}}\omega^{\phantom{1}}_3, \ 
I^{\phantom{.}}_3 = \omega_2^{\smash{-1}}\omega^{\phantom{1}}_1 \
\in \textup{End}(TM)
\end{equation}
then we have $I_1^2 = I_2^2 = I_3^2 = -1$.
\item \underline{an exterior differential condition}: the 2-forms are symplectic, \textit{i.e.},
\begin{equation}
d\omega_1 = d\omega_2 = d\omega_3 = 0 \rlap{.}
\end{equation}
\end{itemize}
\end{theorem}

\begin{proof}

Observe that although the algebraic condition formally requires only that \mbox{$I_1$, $I_2$, $I_3$} be almost complex structures, it fully implies in fact that they satisfy the entire imaginary quaternionic algebra. Indeed, since they square to minus the identity we have
{\allowdisplaybreaks
\begin{align} \label{I_i-two_exprs}
I^{\phantom{.}}_1 = \omega_3^{\smash{-1}}\omega^{\phantom{1}}_2 = - \pt \omega_2^{\smash{-1}}\omega^{\phantom{1}}_3 \\
I^{\phantom{.}}_2 = \omega_1^{\smash{-1}}\omega^{\phantom{1}}_3 = - \pt \omega_3^{\smash{-1}}\omega^{\phantom{1}}_1 \nonumber \\
I^{\phantom{.}}_3 = \omega_2^{\smash{-1}}\omega^{\phantom{1}}_1 = - \pt \omega_1^{\smash{-1}}\omega^{\phantom{1}}_2 \nonumber
\end{align}
}%
from which we get immediately that
\mbox{$
I^{\phantom{.}}_1 I^{\phantom{.}}_2 I^{\phantom{.}}_3 = 
( - \pt \omega_2^{\smash{-1}}\omega^{\phantom{1}}_3)
(- \pt \omega_3^{\smash{-1}}\omega^{\phantom{1}}_1)
(- \pt \omega_1^{\smash{-1}}\omega^{\phantom{1}}_2)
= - 1
$}.

Let us define now the following sections of $\text{Hom}(T^*\nhp M, TM)$ or, equivalently, rank-2 contravariant tensors: 
{\allowdisplaybreaks
\begin{align} \label{g1g2g3}
g_1 = - \pt \omega^{\phantom{1}}_1I^{\phantom{1}}_1 = - \pt \omega^{\phantom{1}}_1 \omega_3^{\smash{-1}}\omega^{\phantom{1}}_2 = \omega^{\phantom{1}}_1\omega_2^{\smash{-1}}\omega^{\phantom{1}}_3 \\
g_2 = - \pt \omega^{\phantom{1}}_2I^{\phantom{1}}_2 = - \pt \omega^{\phantom{1}}_2\omega_1^{\smash{-1}}\omega^{\phantom{1}}_3 = \omega^{\phantom{1}}_2\omega_3^{\smash{-1}}\omega^{\phantom{1}}_1 \nonumber \\
g_3 = - \pt \omega^{\phantom{1}}_3I^{\phantom{1}}_3 = - \pt \omega^{\phantom{1}}_3\omega_2^{\smash{-1}}\omega^{\phantom{1}}_1 = \omega^{\phantom{1}}_3\omega_1^{\smash{-1}}\omega^{\phantom{1}}_2 \mathrlap{.} \nonumber
\end{align}
}%
Since the $\omega_i$'s are all antisymmetric, by comparing for example the third expression for $g_1$ with the fourth one for $g_2$ we see that \mbox{$g_1^T = g^{\phantom{'}}_2$}, where the superscript $T$ denotes transposition. Similarly, $\smash{ g_2^T = g^{\phantom{'}}_3 }$ and $\smash{ g_3^T = g^{\phantom{'}}_1 }$, from which it follows immediately that \mbox{$g_1 = g_2 = g_3 =$} a symmetric rank-2 contravariant tensor. In other words, the algebraic condition also implies that we can define a \textit{metric} on $M$ by
\begin{equation} \label{g-def}
g(X,Y) = - \pt \omega_1(X,I_1Y) = - \pt \omega_2(X,I_2Y) = - \pt \omega_3(X,I_3Y)
\end{equation}
for any $X,Y \in TM$. The minus signs have been chosen in order to ensure compatibility with our original definition of the 2-forms $\omega_i$.  Such a metric is automatically Hermitian with respect to each $I_i$ (this is possible only when \mbox{$\dim M = 4n$}), and its signature is in general of type $\smash{ (4n_+,4n_-) }$, with $\smash{n_+ \nhp + n_- = n }$, so the manifold $M$ is in general pseudo-Riemannian.

Remark now that in order to prove that the manifold $M$ carries a hyperk\"ahler structure it suffices to prove that each $\omega_i$ is covariantly constant with respect to the Levi-Civita connection of the metric $g$. The covariant constancy of the $I_i$ follows then right away.

This can be shown in two steps. First, note that based on the quaternionic algebra we can write successively 
\mbox{$\omega_2(X,I_1Y) = g(X,I_2I_1Y) = - \pt g(X,I_3Y) = - \pt \omega_3(X,Y)$},
and since the last expression is antisymmetric at the exchange of $X$ and $Y$ it follows that $\omega_2$ is of mixed $(2,0)+(0,2)$ type with respect to $I_1$. By part $(b)$ of Lemma \ref{ab} we then have
\begin{align} \label{cl-b}
\omega_2(N_{I_1}(X,Y),Z) & = d\omega_2(I_1X,I_1Y,Z) - d\omega_2(X,Y,Z) \\
& \hspace{1pt} + d\omega_3(I_1X,Y,Z) + d\omega_3(X,I_1Y,Z) \mathrlap{.} \nonumber
\end{align}
As $\omega_2$ and $\omega_3$ are by assumption closed, the right-hand side of this equation vanishes, and since moreover $\omega_2$ is non-degenerate it follows that the Nijenhuis tensor of $I_1$ vanishes. By the Newlander--Nirenberg theorem, $I_1$ is then an integrable complex structure. 

Second, as $\omega_1(X,I_1Y) = - \pt g(X,Y)$ is symmetric at the exchange of $X$ and $Y$, $\omega_1$ is of type $(1,1)$ with respect to $I_1$.  By part $(a)$ of Lemma \ref{ab} we can then write 
\begin{equation} \label{cl-a}
2 \hp \nabla_{\! Z} \hp \omega_1(X,Y) = \omega_1(N_{I_1}(X,Y),Z) - d\omega_1(I_1X,I_1Y,Z) + d\omega_1(X,Y,Z) \mathrlap{.}
\end{equation}
The vanishing of the Nijenhuis tensor for $I_1$ together with the closure of $\omega_1$ imply then that $\omega_1$ is covariantly constant with respect to the Levi-Civita connection. The covariant constancy of $\omega_2$ and $\omega_3$, as well as the integrability of $I_2$ and $I_3$, follow by cyclically permuting the indices in the argument. 
\end{proof}

\begin{remark}
The triplets of 2-forms defining a given hyperk\"ahler metric are not unique. To be more precise, the conditions of the theorem are invariant under the following two types of constant linear transformations:
\begin{itemize}
\setlength\itemsep{0.2em}

\item[1)] \mbox{$\omega_i \mapsto \lambda \pt \omega_i$} for any non-vanishing constant $\smash{\lambda \in \mathbb{R}^{\times}}$. The corresponding hyperk\"ahler metric rescales then from $g$ to \mbox{$\lambda \hp g$}. 
\item[2)] \mbox{$\omega_i \mapsto R_{ij} \hp \omega_j$} for any constant rotation matrix \mbox{$R_{ij} \in SO(3)$}. These transformations leave the metric unchanged. Thus, for any triplet of 2-forms $\omega_i$ satisfying the conditions of the theorem one automatically has a half 3-sphere's worth of such triplets giving the same hyperk\"ahler metric (recall that, topologically, \mbox{$SO(3) \cong S^3/\mathbb{Z}_2$}).
\end{itemize}


\noindent Out of these, the only invariance which is perhaps not immediately obvious is that of the algebraic condition at $SO(3)$ rotations. The considerations in the first part of the proof make it clear that the algebraic condition is equivalent to the existence of a triplet \mbox{$I_i \in \text{End}(TM)$} and a metric $g$ such that \mbox{$\omega_i \hp I_j = \varepsilon_{ijk} \pt \omega_k - \delta_{ij} \hp g$}. Then, since $SO(3)$ transformations preserve both the $\varepsilon$-symbol and the $\delta$-symbol  (the $S$ and $O$ in $SO(3)$, respectively), this relation is invariant under simultaneous rotations \mbox{$\omega_i \mapsto R_{ij} \hp \omega_j$} and \mbox{$I_i \mapsto R_{ij} I_j$} while $g$ stays unchanged. That is to say, for each rotated set of $\omega_i$'s there exists a rotated set of $I_i$'s satisfying the same algebraic condition and determining the same metric.
\end{remark}

On the quaternionic K\"ahler side, on the other hand, we have the following version of Proposition 3 from reference \cite{MR1376149}:

\begin{theorem} \label{QK-criterion-1}
A finite-dimensional manifold $M$ is quaternionic K\"ahler if and only if it possesses a rank-3 $SO(3)$-subbundle $Q$ of the bundle of non-degenerate 2-forms on $M$ and an open covering $\pt\mathcal{C}$ of $M$ on each element $U$ of which one can pick a local frame \mbox{$\omega_1$, $\omega_2$, $\omega_3$} for $Q|_{U}$ satisfying simultaneously 
\begin{itemize}
\setlength\itemsep{0.4em}

\item \underline{an algebraic condition}: if we define 
\begin{equation}
I^{\phantom{.}}_1 = \omega_3^{\smash{-1}}\omega^{\phantom{1}}_2, \
I^{\phantom{.}}_2 = \omega_1^{\smash{-1}}\omega^{\phantom{1}}_3, \ 
I^{\phantom{.}}_3 = \omega_2^{\smash{-1}}\omega^{\phantom{1}}_1 \
\in \textup{End}(TM)
\end{equation}
then we have $I_1^2 = I_2^2 = I_3^2 = -1$.
\item \underline{an exterior differential condition}: 
there exist locally defined 1-forms \mbox{$\theta_1$, $\theta_2$, $\theta_3$} such that 
\begin{equation}
d\omega_i = - \hp 2 \pt \varepsilon_{ijk\,} \theta_j \wedge \omega_k \rlap{.}
\end{equation}
\end{itemize}
\end{theorem}

\begin{proof}

The structure of the proof mirrors closely that of the previous theorem. In particular, by virtually the same arguments as before the algebraic condition implies that the almost complex structures \mbox{$I_1$, $I_2$, $I_3$} satisfy the imaginary quaternionic algebra, and that, moreover, a tri-Hermitian metric $g$ can be defined by means of the same equation \eqref{g-def}. 

The essential difference with respect to the hyperk\"ahler case is that these quantities are now defined \textit{locally}, and we need to check in addition that they give rise to the requisite global structures. For that, let us consider two arbitrary overlapping open sets \mbox{$U, V \nhp \in \pt\mathcal{C}$}. The two corresponding local frames $\omega_{iU}$ and $\omega_{iV}$ of the vector bundle $Q$ are related on the intersection $U \cap V$ by means of a transition function $R_{ij}$ which is a local $SO(3)$ rotation: \mbox{$\omega_{iV} = R_{ij} \pt \omega_{jU}$}. Adjoining to the symbol of each locally defined quantity its domain of definition we have then by essentially the same argument as in the Remark following the previous theorem \mbox{$(R_{im} \pt \omega_{mU}) (R_{jn} \hp I_{nU}) = \varepsilon_{ijk} \hp (R_{kl} \pt \omega_{lU}) - \delta_{ij} \pt g_U$}. By considering the off-diagonal terms in this relation we can see immediately that we must have \mbox{$I_{iV} = R_{ij} \hp I_{jU}$}, that is, the almost complex structures are local sections\,---\,and in fact, frames\,---\,of an associated $SO(3)$-subbundle of $\text{End}(TM)$. On the other had, from the diagonal terms we get that \mbox{$g_V = g_U$}, that is, the metric is globally defined. 

To prove that the manifold $M$ carries a quaternionic K\"aher structure it suffices to show that the action of the Levi-Civita covariant derivative with respect to the metric $g$ on the 2-forms $\omega_i$ is of the form \eqref{cd_oi-QK}. The corresponding covariant differentiation property \eqref{cd_Ii-QK} of the almost complex structures $I_i$ follows then immediately.

By the same reasoning as in the hyperk\"ahler case, $\omega_2$ and $\omega_1$ are of mixed $(2,0)+(0,2)$ respectively pure $(1,1)$ type with respect to $I_1$. So then by parts $(b)$ and $(a)$ of Lemma~\ref{ab} the same identities \eqref{cl-b} and \eqref{cl-a} hold now, too. From the first one, by inserting the expressions for $d\omega_2$ and $d\omega_3$ given by the current exterior differential condition we retrieve, after a purely algebraic computation, the following form for the Nijenhuis tensor of $I_1$:
\begin{align}
N_{I_1}(X,Y) & = 2 \hp (\theta_2(X) - \theta_3(I_1X)) I_2Y + 2 \hp (\theta_2(I_1X) + \theta_3(X)) I_3Y \\
& - \pt 2 \hp (\theta_2(Y) \hp - \theta_3(I_1Y)) \hp I_2X \hp - 2 \hp (\theta_2(I_1Y) \hp + \theta_3(Y)) \hp I_3X \mathrlap{.} \nonumber
\end{align}
Plugging then this into the second identity together with the remaining expression for $d\omega_1$ gives us after another series of algebraic manipulations that
\begin{equation}
\nabla_{\! Z} \hp \omega_1(X,Y)  = - \pt 2 \hp\theta_2(Z) \pt \omega_3(X,Y) + 2\hp \theta_3(Z) \pt \omega_2(X,Y) \mathrlap{.}
\end{equation}
Permuting the indices cyclically in this argument yields the rest of the components of the desired property \eqref{cd_oi-QK}. 
\end{proof}

\begin{remark}
Quaternionic K\"aher manifolds are particular examples of \textit{quaternionic manifolds} (for a detailed review see \textit{e.g.}~\cite{MR1441871}). These are manifolds $M$ which carry a three-dimensional subbundle of $\text{End}(TM)$ with fibers spanned at each point by a basis $\{I_1,I_2,I_3\}$ whose elements satisfy the algebra of imaginary quaternions, and which, in addition, admit a torsion-free connection preserving this structure. If it exists, such a connection, called an Oproiu connection, is not unique. Their so-called structure tensor takes the form
\begin{equation}
\frac{1}{6} \sum_{i=1}^3 N_{I_i}(X,Y) = \sum_{i=1}^3 (\hp \theta_{\hp \mathrlap{i}}{}^{\text{Op}}(X) \hp I_iY - \theta_{\hp \mathrlap{i}}{}^{\text{Op}}(Y) \hp I_iX )
\end{equation}
where the triplet of 1-forms $\smash{ \theta^{\text{Op}}_{\smash{i}} }$ are the Oproiu connection 1-forms. 

From the above formula for the Nijenuis tensor for $I_1$ and its cyclic permutations one can verify that in the quaternionic K\"ahler case the structure tensor is indeed of this form, with
\begin{equation}
\theta_{\hp \mathrlap{i}}{}^{\text{Op}} = \frac{2}{3} \hp \theta_i - \frac{1}{3} \hp \varepsilon_{ijk} \hp \theta_j I_k \mathrlap{.}
\end{equation}
(Here we use the convention that the $I_i$ act on vector fields from the left and dually, \pagebreak on 1-forms, from the right.) Note that these can be rewritten as
\begin{equation}
\theta_{\hp \mathrlap{i}}{}^{\text{Op}} = \theta_i + \xi I_i
\qquad\text{with}\qquad
\xi = \frac{1}{3} (\theta_1I_1 + \theta_2I_2 + \theta_3I_3)
\end{equation}
which essentially means (see \textit{e.g.}~\cite{MR2200267}) that in the quaternionic K\"ahler case the Oproiu connection is equivalent to the Levi-Civita one.
\end{remark}

The following alternative version of the criterion enunciated in Theorem \ref{QK-criterion-1} also holds:

\begin{theorem} \label{QK-criterion-2}
The exterior differential condition of Theorem \ref{QK-criterion-1} can be replaced with the following condition:

\begin{itemize}
\item The principal $SO(3)$-bundle $\mathcal{Q}$ associated to the bundle $Q$ admits a connection with curvature 2-form equal, up to a non-vanishing constant, to $\omega_i$. That is, there exist locally defined connection 1-forms $\theta_i$ and a constant $s \neq 0$ such that
\begin{equation} \label{Einst_cond}
d\theta_i + \varepsilon_{ijk\,} \theta_j \wedge \theta_k  = s\, \omega_i \mathrlap{.}
\end{equation} 
\end{itemize}
\end{theorem}

\begin{proof}

The fact that this condition implies the second condition of Theorem \ref{QK-criterion-1} is straightforward: the latter is a Bianchi-type consistency condition for the former. The converse implication follows, for $n > 1$, from a result first shown by Alekseevsky, see \textit{e.g.} part $(iii)$ of Theorem 5.7 in \cite{MR1441871}. For $n=1$, the condition above is equivalent to the Einstein equation. In either case, one can identify
\begin{equation}
s = \frac{R_g}{8 n(n+2)}
\end{equation}
where $R_g$ is the scalar curvature of the metric $g$. ($s$ is called \textit{reduced scalar curvature}, and the extra $1/2$ factor with respect to the definition in \cite{MR1441871} is a matter of convention stemming from our unusual normalization of the $SO(3)$ connection 1-forms.)
\end{proof}

\subsection{Killing vector fields} \label{ssec:Killing}

Next, we want to explore how symmetry conditions, and in particular the Killing condition, translate in the framework in which triplets of 2-forms rather than the metric and the (almost) complex structures play the fundamental role. In order to address this question, we need to recall first some background material culminating with a general theorem due to B.~Kostant concerning Killing vector fields on Riemannian manifolds with special holonomy. 

\subsubsection{}

Let $M$ be a Riemannian manifold with metric $g$. By definition, we call an endomorphism \mbox{$\Omega \in \text{End\hp}(TM)$} skew-symmetric if 
\begin{equation}
g(X,\Omega \hp Y) = - \pt g(Y,\Omega X) \pt \stackrel{\smash[t]{\textup{def}}}{=} \hp \omega(X,Y) \mathrlap{.}
\end{equation}
for any pair \mbox{$X,Y \in TM$}. Skew-symmetric endomorphisms from $\text{End}(TM)$ are in one-to-one correspondence with 2-forms on $M$ by way of the metric. On the space $\text{SkewEnd}(TM)$ of such endomorphisms one can define a natural inner product by
\begin{equation} \label{inpr}
\langle \hp \Omega_1, \Omega_2 \hp \rangle = - \pt \frac{1}{\dim M} \pt \text{trace}\hp(\hp \Omega_1 \Omega_2) \mathrlap{.}
\end{equation}
This is in general non-degenerate and, in Riemannian signature, positive definite. 
Moreover, a choice of orthonormal basis on the tangent space $T_pM$ at a point \mbox{$p \in M$} gives a vector space isometry $\smash{ T_pM \longrightarrow  \mathbb{R}^{\hp \dim M} }$, with the latter space endowed with the standard Euclidean metric. This determines in turn an isomorphism between $\text{SkewEnd}(T_pM)$ and $\mathfrak{so}(\dim M)$, the Lie algebra of skew-symmetric endomorphisms of $\smash{ \mathbb{R}^{\hp \dim M} }$.

\subsubsection{}

For any vector field \mbox{$X \in TM$} let us define the operator 
\mbox{$A_X= \mathcal{L}_X - \nabla_{\! X}$}
measuring the difference between the Lie derivative and the Levi-Civita covariant derivative along~$X$. Also, for any two vector fields \mbox{$X,Y \in TM$}, let
\mbox{$R_{X,Y} = [\hp \nabla_{\! X},\nabla_{\nhp Y}] - \nabla_{\nhp\smash{[X,Y]}}$}
be the usual Riemannian curvature operator of the Levi-Civita connection. As derivations acting linearly on any tensor field on $M$ and vanishing on functions, both of these operators are linear algebraic rather than differential in nature and are represented by tensors. The tensors are in either case endomorphisms of the tangent bundle $TM$. For $\smash{ A_X }$, this is \mbox{$\nabla \nhp X$}, the covariant derivative of $X$, which one may view as the endomorphism of $TM$ defined by \mbox{$\nabla \nhp X Z = \nabla_{\nhp Z}\nhp X$} for any \mbox{$Z \in TM$}. For $\smash{ R_{X,Y} }$, it is the Riemann curvature endomorphism $R(X,Y)$. So for example, on functions, vector fields and endomorphisms of the tangent bundle we have, respectively, 
{\allowdisplaybreaks
\begin{alignat}{6} \label{ops_on_tens}
& A_X f && = 0                                     &\qquad& R_{X,Y}f && = 0 &\qquad& \forall \  f && \in C^{\infty}(M) \\[-1pt]
& A_X Z && = - \pt \nabla \nhp X Z &\qquad& R_{X,Y}Z && = R(X,Y) Z  &\qquad& \forall \ Z && \in TM  \nonumber \\[-1pt]
& A_X \Omega && = - [\hp \nabla \nhp X, \Omega \hp] &\qquad& R_{X,Y}\Omega && = [R(X,Y),\Omega \hp] &\qquad& \forall \ \Omega && \in \text{End}(TM) \nonumber 
\end{alignat}
}%
where the square brackets denote commutators. What is more, \mbox{$R(X,Y) \in \text{SkewEnd}(TM)$} for all \mbox{$X,Y \in TM$}, whereas \mbox{$\nabla \nhp X \in \text{SkewEnd}(TM)$} if and only if $X$ is a Killing vector field. 

In fact, in regard to Killing vector fields we have, moreover, the following properties (see \textit{e.g.}~Proposition 2.6 in ch.\,VI, sec.\,2 of \cite{MR0152974}):

\begin{lemma}
Let $M$ be a Riemannian manifold. 
\begin{itemize}
\setlength\itemsep{0.2em}

\item[1)] If $X$ is a Killing vector field on $M$ then
\begin{equation} \label{AX_Kill_1}
[\hp \nabla_{\nhp Y},A_X] = R_{X,Y}
\end{equation}
for any vector field \mbox{$Y \in TM$}.

\item[2)] If both $X$ and $Y$ are Killing vector fields on $M$ then we have
\begin{equation} \label{R-A-rep}
R_{X,Y} = [A_X,A_Y] - A_{\smash{[X,Y]}} \mathrlap{.}
\end{equation}
\end{itemize}
\end{lemma}

\begin{proof}
The relations follow immediately from the definitions upon using the fact that the Lie derivative with respect to a Killing vector field commutes with the Levi-Civita connection. Formally, this can be expressed as follows: if $X$ is a Killing vector field, then \mbox{$[\mathcal{L}_X,\nabla_{\nhp Y}] = \nabla_{\nhp\smash{[X,Y]}}$} for any vector field $Y$.
\end{proof}

\noindent Note that in terms of the associated skew-symmetric endomorphisms the operatorial relation \eqref{AX_Kill_1} can be equivalently written as
\begin{equation} \label{Kill_R-prop}
\nabla_Y (\nabla \nhp X) = R(Y,\nhp X) \mathrlap{.}
\end{equation}
This is sometimes known as \textit{Kostant's formula}.

\subsubsection{}

The isometries of a Riemannian manifold form a Lie group (Myers--Steenrod). On another hand, we have the notion of \textit{holonomy group}. Given a connected Riemannian manifold $M$, its holonomy group at a point \mbox{$p \in M$} is by definition the set of linear endomorphisms of the tangent space $T_pM$ obtained by parallel transporting tangent vectors with respect to the Levi-Civita connection around piecewise smooth loops based at $p$. This set has naturally the structure of a group. The \textit{restricted holonomy group} at $p$ is the subgroup coming from contractible loops. The latter is a connected Lie subgroup of the group of orthogonal transformations of $T_pM$. By choosing an orthonormal basis for $T_pM$ we may identify it with a Lie subgroup of $SO(\dim M)$. This subgroup depends on both the choice of basis and of base point $p$, but its conjugacy class depends on neither. The possible holonomy groups of simply-connected irreducible Riemannian manifolds have been classified by M.~Berger. 

The relation between the holonomy group and the (continuous part of the) isometry group of a Riemannian manifold is described by the following result: 

\begin{theorem}[Kostant \cite{MR0084825}] \label{Kostant_th}
Let $M$ be a Riemannian manifold and $\mathfrak{h}$ be the Lie algebra of the restricted holonomy group at a point \mbox{$p \in M$}. If $X$ is a Killing vector field on $M$, then the skew-symmetric endomorphism of the tangent space $T_pM$ determined by $\nabla \nhp X$ belongs to the normalizer of $\mathfrak{h}$ in $\mathfrak{so}(\dim M)$. Moreover, if $M$ is orientable and compact then this belongs simply to $\mathfrak{h}$.
\end{theorem}

\begin{proof}

As a skew-symmetric endomorphism of the tangent bundle, $\nabla \nhp X$ can be decomposed uniquely into two orthogonal components with respect to the inner product \eqref{inpr}, one belonging to $\mathfrak{h}$ and the other to its orthogonal complement in $\mathfrak{so}(\dim M)$: 
\begin{equation}
\nabla \nhp X = \nabla \nhp X^{\hp\mathfrak{h}} + \nabla \nhp X^{\hp\mathfrak{h}^{\smash{\perp}}} \! \mathrlap{.}
\end{equation}
Suppose then we act with $\nabla_Y$ for some vector field $Y$ simultaneously on both sides of this equation. One the left-hand side we obtain, by way of the property \eqref{Kill_R-prop}, $R(Y,\nhp X)$, which, by the Ambrose-Singer theorem, belongs to $\mathfrak{h}$. On the other hand, since the inner product is preserved by parallel transport, one can argue that $\smash{ \nabla_Y(\nabla \nhp X^{\hp\mathfrak{h}}) \in \mathfrak{h} }$ and $\smash{ \nabla_Y (\nabla \nhp X^{\hp\mathfrak{h}^{\smash{\perp}}} \nhp) \in \mathfrak{h}^{\perp}}$ for all $Y$. The inner product being non-degenerate, the intersection $\smash{ \mathfrak{h} \cap \mathfrak{h}^{\perp} }$ is trivial, and so it follows that one must have
\mbox{$\nabla_Y (\nabla \nhp X^{\hp\mathfrak{h}^{\smash{\perp}}} \nhp) = 0$}.
This means in particular that $\smash{\nabla \nhp X^{\hp\mathfrak{h}^{\smash{\perp}}}}$ is invariant at parallel transport around a closed loop and thus belongs to the centralizer of $\mathfrak{h}$ in $\smash{ \mathfrak{so}(\dim M) }$. The first part of the theorem follows then immediately.

For further details as well as the proof of the fact that when $M$ is orientable and compact $\smash{\nabla \nhp X^{\hp\mathfrak{h}^{\smash{\perp}}}}$ vanishes we direct the interested reader to the original paper \cite{MR0084825}, or the review in ch.\,VI of \cite{MR0152974}.
\end{proof}

\subsubsection{}

Let us return now to the particular case of quaternionic K\"ahler and hyperk\"ahler ma\-ni\-folds, for which \mbox{$\dim M = 4n$}. Remark that we have now at our disposal three more skew-symmetric endomorphisms of the tangent bundle, namely \mbox{$I_1$, $I_2$, $I_3$}. Since $\textup{SkewEnd}(T_pM)$ is isomorphic in either case, for any point \mbox{$p \in M$}, to the fundamental representation of the Lie algebra $\mathfrak{so}(4n)$, it inherits this one's structure. Thus, it naturally splits into three subspaces isomorphic to the $\mathfrak{sp}(n)$ and $\mathfrak{sp}(1)$ Lie subalgebras of $\mathfrak{so}(4n)$, and their complement in $\mathfrak{so}(4n)$, respectively. The second subspace is spanned by \mbox{$I_1$, $I_2$, $I_3$} and, moreover, any element from the first subspace commutes with any element from the second. What is more, the three subspaces are mutually orthogonal with respect to the inner product \eqref{inpr}. 

These algebraic properties can be seen explicitly for instance within the so-called $E$--$H$ formalism of \cite{MR664330}. Let us recall briefly its main aspects. For quaternionic K\"ahler manifolds one can associate in principle a locally defined vector bundle to each representation of the holonomy group $\smash{ Sp(n) \times_{\mathbb{Z}_2} \! Sp(1) }$. In particular, the complexified cotangent bundle $\smash{ T^*_{\mathbb{C}}M }$ splits locally into a tensor product \mbox{$E \otimes H$} of two such vector bundles corresponding to the standard complex representations of rank $2n$ and $2$ of $Sp(n)$ and $Sp(1)$, respectively. For \mbox{$n=1$}, in four dimensions, where the holonomy definition is side-stepped, one takes instead $E$ and $H$ to be the two spinor bundles of $M$. The bundles $E$ and $H$ come equipped with natural symplectic forms\,---\,constant anti-symmetric sections \mbox{$\varepsilon_E \in \Lambda^2E$} and \mbox{$\varepsilon_H \in \Lambda^2H$}. In the hyperk\"ahler case, which in this context is most profitably  viewed as a quaternionic K\"ahler limit case, the same split structure exists, except that the bundle $H$ is now trivial. 

Any $SO(4n)$-module decomposes into direct sums of irreducible $\smash{ Sp(n) \times_{\mathbb{Z}_2} \! Sp(1) }$-mo\-dules, which are  tensor products of even total number of $Sp(n)$ and $Sp(1)$-modules. In particular, this is the case for the bundle \mbox{$\Lambda^2T^*\npt M$}. Any 2-form $\omega$ thus decomposes as
\begin{equation}
\omega = \omega^{\pt\mathfrak{sp}(n)} + \omega^{\pt\mathfrak{sp}(1)} + \omega^{\hp\mathfrak{cpl}}
\end{equation}
where, in a local $E$--$H$ frame, we have
{\allowdisplaybreaks
\begin{alignat}{2}
& \omega^{\pt\mathfrak{sp}(n)} \! && = \omega_E \otimes \varepsilon_H,\, \text{with $\omega_E \pt \in \pt S^2E \hp \cong \mathfrak{sp}(n)$} \\
& \omega^{\pt\mathfrak{sp}(1)} \! && = \varepsilon_E \otimes \omega_H,\, \text{with $\omega_H \in S^2H \cong \mathfrak{sp}(1)$} \nonumber \\
& \mathrlap{\omega^{\hp\mathfrak{cpl}} \in \Lambda^2_0E \otimes S^2H.} \nonumber
\end{alignat}
}%
Here, $S$ and $\Lambda$ denote symmetric respectively anti-symmetric tensor products, and $\Lambda_0^2E$ denotes the kernel of the contraction of $\Lambda^2E$ with the inverse of the symplectic form $\varepsilon_E$ (this kernel is trivial for \mbox{$n=1$}). Note that the quaternionic K\"ahler/hyperk\"ahler 2-forms $\omega_i$ are of pure $\mathfrak{sp}(1)$ type, and they provide a frame for the respective subbundle of \mbox{$\Lambda^2T^*\npt M$}.

On another hand, in the same $E$--$H$ frame the metric factorizes locally as
\mbox{$g = \varepsilon_E \otimes \varepsilon_H$}.
Recalling then that any skew-symmetric endomorphism $\Omega$ of $TM$ is related to a corresponding 2-form $\omega$ by $\smash{ \Omega = g^{-1}\omega }$, the statements above, and in particular the orthogonality claim, can be verified directly through simple algebraic calculations. 

\subsubsection{}

The curvature tensor also admits an irreducible $E$--$H$ decomposition \cite[Theorem 3.1]{MR664330}. In what follows, though, we will only need the fact that the curvature endomorphism of a quaternionic K\"ahler manifold satisfies the commutation property
\begin{equation} \label{[Ii,R]-eq}
[\pt I_i,R(X,Y)] = 2\hp s \, \varepsilon_{ijk} \, \omega_j(X,Y) I_k \mathrlap{.}
\end{equation}
For hyperk\"ahler manifolds the right-hand side vanishes, consistent with taking the limit to \mbox{$s=0$}. In the quaternionic K\"ahler case this is essentially the consistency condition for the differential equation \eqref{cd_Ii-QK}, with the Einstein condition \eqref{Einst_cond} taken into account. In the hyperk\"ahler case it is simply the consistency condition for the corresponding equation \eqref{cd_Ii-HK}. In either case its structure can be easily understood by purely algebraic arguments from the Ambrose--Singer theorem. 

Note in particular that if we multiply this equality from either side with an $I_l$ and then take the trace of the result we obtain\,---\,using successively the cyclicity property of the trace, the quaternionic algebra, and the fact that the trace of any element from $\text{SkewEnd}(TM)$ vanishes\,---\,the following consequence:
\begin{equation} \label{tr(I_iR)}
\langle \pt I_i , R(X,Y) \rangle = s \, \omega_i(X,Y) \mathrlap{.}
\end{equation}

\subsubsection{}

For the remainder of this section we will assume that the manifolds are irreducible, by which we mean that their holonomy groups are assumed to be equal to rather than included in $\smash{ Sp(n) \times_{\mathbb{Z}_2} \! Sp(1) }$ in the quaternionic  K\"ahler case (for \mbox{$n>1$}), and $Sp(n)$ in the hyperk\"ahler one. 

If the quaternionic K\"ahler or hyperk\"ahler manifold $M$ possesses a Killing vector field $X$, then by Theorem \ref{Kostant_th} its covariant derivative \mbox{$\nabla \nhp X$} belongs to the normalizer of the Lie subalgebra \mbox{$\mathfrak{sp}(n) \oplus \mathfrak{sp}(1)$} respectively $\mathfrak{sp}(n)$ in $\mathfrak{so}(4n)$, which one can easily see is in both cases \mbox{$\mathfrak{sp}(n) \oplus \mathfrak{sp}(1)$}. With the notations above, this means that \mbox{$\nabla \nhp X^{\mathfrak{cpl}} = 0$}, and so
\begin{equation}
\nabla \nhp X = \nabla \nhp X^{\mathfrak{sp}(n)} + \nabla \nhp X^{\mathfrak{sp}(1)}
\end{equation}
with the two remaining components both orthogonal and commuting. In four dimensions this decomposition is automatic for any skew-symmetric endomorphism of $TM$, so in this case one need not resort to holonomy arguments. The second component, in particular, is a linear superposition
\begin{equation} \label{sp1-Iis}
\nabla \nhp X^{\mathfrak{sp}(1)} = - \sum_{i=1}^3 \nu_{X i} \hp I_i
\end{equation}
with coefficients $\nu_{X i}$, where the minus sign is conventionally chosen for later convenience. The generators $I_i$ satisfy the orthogonality relations \mbox{$\langle \hp I_i, I_j \rangle = \delta_{ij}$} and \mbox{$\langle \hp I_i, \nabla\nhp X^{\mathfrak{sp}(n)} \rangle = 0$}, and thus we may identify
\begin{equation} \label{nu's}
\nu_{Xi} = - \pt \langle \hp I_i, \nabla \nhp X \rangle \mathrlap{.}
\end{equation}

\subsubsection{}

These coefficients play an important role in the description of Killing symmetries. In the quaternionic K\"ahler case, switching to an $\mathbb{R}^3$-vector notation, we have:

\begin{theorem}[Galicki, Lawson \cite{MR872143,MR960830}] \label{QK-mm-th}
Consider a quaternionic K\"ahler manifold $M$ with associated principal $SO(3)$-bundle $\mathcal{Q}$, and assume that $M$ has a continuous group of isometries $G$, with Lie algebra $\mathfrak{g}$. Then a vector field \mbox{$X \in \mathfrak{g}$}, that is $X$ is Killing, if and only if there exists a local section $\smash{ \vec{\nu}_X }$ of the tensor product bundle \mbox{$\mathfrak{g}^* \otimes \mathcal{Q}$} over $M$ such that\,\footnote{\,The symbol $\times$ denotes the usual $\mathbb{R}^3$-vector cross product. }
\begin{equation} \label{Kill_rot_omi}
\mathcal{L}_X \vec{\omega} = - \pt 2 \hp (\hp \iota_X\vec{\theta} + \vec{\nu}_{X}) \nhp\vprod\nhp \vec{\omega}
\end{equation}
or, equivalently, such that
\begin{equation} \label{QK-mom-map}
d\vec{\nu}_{X} + 2 \pt \vec{\theta} \vprod \vec{\nu}_{X} = s \, \iota_X \vec{\omega} \mathrlap{.}
\end{equation}
Explicitly, this section is given by the formula \mbox{$\vec{\nu}_{X} = - \pt \langle \hp \vec{I}, \nabla \nhp X \rangle$} for any \mbox{$X \in \mathfrak{g}$} and it is 
\begin{itemize}
\setlength\itemsep{0.2em}

\item[$\circ$] nowhere vanishing

\item[$\circ$] the unique solution to the quaternionic-K\"ahler moment map equation \eqref{QK-mom-map} 

\item[$\circ$]  $G$-equivariant

\item[$\circ$] and, for any \mbox{$X,Y \in \mathfrak{g}$}, we have
\begin{equation} \label{nu_[X,Y]-QK}
2\hp \vec{\nu}_{\hp [X,Y]} = 2\hp \vec{\nu}_{X} \nhp\vprod 2\hp \vec{\nu}_{\pt Y} - 2\hp s\, \vec{\omega}(X,Y) \mathrlap{.}
\end{equation}

\end{itemize}
\end{theorem}

\begin{proof}

Suppose $X$ is a Killing vector field on $M$. Then by the rules laid out in the list \eqref{ops_on_tens} we have
\begin{equation} \label{A_XI_i}
A_X I_i =  [\hp I_i, \nabla \nhp X \hp] = [\hp I_i, \nabla \nhp X^{\mathfrak{sp}(1)}]  = - \pt 2 \pt \varepsilon_{ijk} \pt \nu_{Xj} I_k  \mathrlap{.}
\end{equation}
The $\mathfrak{sp}(n)$ component of \mbox{$\nabla \nhp X$} drops out because it commutes with the $\mathfrak{sp}(1)$ generators, and for the remaining $\mathfrak{sp}(1)$ component  we use the representation \eqref{sp1-Iis}. Converting then the almost complex structures into 2-forms by means of the metric, which is preserved by the action of $A_X$, gives us promptly the rotation property \eqref{Kill_rot_omi}. 

Conversely, let us assume that we have a vector field $X$ whose Lie action rotates the quaternionic K\"ahler 2-forms, that is \mbox{$\mathcal{L}_X \omega_i = - \pt 2 \pt \varepsilon_{ijk} \hp r_j \pt \omega_k$} for some locally defined triplet of functions $\smash{ r_j }$. Using then any one of the expressions \eqref{g1g2g3} of the metric in terms of the quaternionic K\"ahler 2-forms 
one can easily check that this implies that $X$ is Killing, regardless of the particular form of the functions $\smash{ r_j }$. 

Suppose now again that $X$ is Killing. If we act with $\nabla_{\nhp Y}$ for some arbitrary \mbox{$Y \in TM$} on the expression \eqref{nu's} for $\smash{ \nu_{Xi} }$ we obtain, by applying the Leibniz rule and then the properties \eqref{cd_Ii-QK}, \eqref{Kill_R-prop} and \eqref{tr(I_iR)}, that
\mbox{$\nabla_{\nhp Y} \nu_{Xi} +  2 \pt \varepsilon_{ijk} \pt \theta_j(Y) \hp \nu_{Xk} = s \, \omega_i(X,Y)$}. The quaternionic K\"ahler moment map equation \eqref{QK-mom-map} is simply a rewriting of this equation in the language of differential forms. 

Conversely, let us assume that there exists a vector field $X$ satisfying the quaternionic K\"ahler moment map equation for some locally defined functions $\smash{ \nu_{Xi} }$. Then by acting on this equation with an exterior differential and making use of the Einstein property \eqref{Einst_cond} one can show that the rotation property \eqref{Kill_rot_omi} is implied. By the argument above, this implies further that $X$ must be Killing. 

The non-vanishing of the section $\smash{ \nu_{Xi} }$ follows directly from the quaternionic K\"ahler moment map equation. Should it vanish at some point, the equation would imply that $X$ would be a null vector for the 2-forms $\omega_i$ at that point, which would contradict their non-degeneracy. 

For an argument that the solution to the quaternionic K\"ahler moment map equation is unique we refer the reader to reference \cite{MR960830}. 

\begin{remark}
Note that, unlike the rest of the arguments in the proof, each of the arguments in the last three paragraphs fails in the limit when \mbox{$s = 0$} and thus cannot be extended to the hyperk\"ahler case. 
\end{remark}

Assume now that we have two Killing vector fields \mbox{$X, Y \in \mathfrak{g}$}. Using the representation \eqref{R-A-rep} for the curvature operator, and then the formulas \eqref{A_XI_i} and \eqref{[Ii,R]-eq}, we can write
\begin{equation}
A_{\smash{[X,Y]}}I_i = [A_X,A_Y] I_i - [R(X,Y),I_i \hp] = - \pt 2 \pt \varepsilon_{ijk} \pt \nu_{\smash{[X,Y]j}} \hp I_k 
\end{equation}
with $\smash{ \nu_{\smash{[X,Y]j}} }$ defined as in the theorem. 

Finally, since $\mathcal{L}_Y$ acts on scalars as $\iota_Yd$, from the quaternionic K\"ahler moment map equation it follows that
\mbox{$\mathcal{L}_Y \nu_{Xi} = \nu_{\hp [X,Y]i} - 2\pt \varepsilon_{ijk} \pt (\hp \iota_Y\theta_j + \nu_{\pt Yj}) \pt \nu_{Xk}$},
showing that the sections $\nu_{Xi}$ transform indeed equivariantly under the infinitesimal action of $G$.
\end{proof}

\subsubsection{}

In the hyperk\"ahler limit the coefficients $\smash{ \vec{\nu}_X }$ lose the moment map interpretation and become constant. 

\begin{theorem}
Let $M$ be a hyperk\"ahler manifold having a continuous group of isometries $G$, with Lie algebra $\mathfrak{g}$. Then a vector field \mbox{$X \in \mathfrak{g}$}, that is $X$ is Killing, if and only if there exists a constant map $\smash{ \nu : M \longrightarrow \mathfrak{g}^* \otimes \mathbb{R}^3 }$ such that
\begin{equation} \label{Kill_rot_omi-HK}
\mathcal{L}_X \vec{\omega} = - \pt 2 \pt \vec{\nu}_{X} \nhp\vprod \vec{\omega} \mathrlap{.}
\end{equation}
Explicitly, the map is given by the formula  \mbox{$\vec{\nu}_{X} = - \pt \langle \hp \vec{I}, \nabla \nhp X \rangle$} and, for any \mbox{$X, Y \in \mathfrak{g}$}, we have
\begin{equation} \label{nu_[X,Y]-HK}
2\hp \vec{\nu}_{\hp [X,Y]} = 2\hp \vec{\nu}_{X} \nhp\vprod 2\hp \vec{\nu}_{\pt Y} \mathrlap{.}
\end{equation}
Moreover, if $M$ is compact, $\nu$ vanishes entirely.
\end{theorem}

\begin{proof}

The proof in the quaternionic K\"ahler case carries over in the limit when $s$ and $\vec{\theta}$ vanish, apart from the three instances identified in the intermediary Remark. In particular, equation \eqref{QK-mom-map} now reads simply \mbox{$d \vec{\nu}_{X} = 0$}, implying that the functions \mbox{$ \vec{\nu}_{X}$} are constant in hyperk\"ahler setting (note also that rather than being local sections, they are now globally defined). This equation no longer has the form of a moment map condition, nor does it imply anymore the rotation property. As for the remaining equations, \eqref{Kill_rot_omi} and \eqref{nu_[X,Y]-QK}, they take in this limit, with less dramatic consequences, the forms \eqref{Kill_rot_omi-HK} and \eqref{nu_[X,Y]-HK}, respectively. Finally, the last statement of Theorem \ref{Kostant_th} implies that, when $M$ is compact, \mbox{$\nabla \nhp X^{\mathfrak{sp}(1)} \npt = 0$} for any \mbox{$X \in \mathfrak{g}$}.
\end{proof}

Killing vector fields $X$ for which $\smash{ \vec{\nu}_{X} }$ vanishes are known as \textit{tri-Hamiltonian}  (or \textit{tri-ho\-lo\-mor\-phic}), and the ones for which this is not the case as \textit{rotational} (or \textit{permuting}). (Note that the two notions do not depend on the particular frame chosen in the space of hyperk\"ahler symplectic 2-forms.) So the theorem states that Killing vector fields on hyperk\"ahler manifolds can only be either of tri-Hamiltonian or of rotational type. This generalizes a result proved by other means in four dimensions by Boyer and Finley \cite{MR660020}. In particular, when the manifold is compact, Killing vector fields, should they exist, are necessarily tri-Hamiltonian.

\subsection{Homothetic Euler vector fields} \label{ssec:HKCs}

We end this section with an examination of a different type of symmetry, conformal in nature rather than Killing, which plays a central role in Swann's celebrated quaternionic bundle construction, reviewed in the next section. 

Consider a Riemannian or pseudo-Riemannian manifold $M$ with Levi-Civita connection $\nabla$. We call a vector field $X$ on $M$ \textit{homothetic Euler} if
\begin{equation}
\nabla_{Y}X = Y
\end{equation}
for any vector field \mbox{$Y \in TM$}. In the notation of subsection~\ref{ssec:Killing} this is equivalent to requiring that \mbox{$\nabla \nhp X$} be equal to the identity endomorphism of $TM$. The terminology is justified by the fact that a manifold $M$ possessing such a vector field is automatically a metric cone \cite{MR1663805}, in which context the vector field is known as an Euler vector field. 

A vector field is homothetic Euler if and only if it is homothetic and hy\-per\-sur\-face-or\-tho\-go\-nal. In fact, we have the following list of equivalent characterizations: 

\begin{proposition} \label{HKC-criterions}
Let $M$ be a Riemannian or pseudo-Riemannian manifold with metric $g$ and corresponding Levi-Civita connection $\nabla$. Then the following conditions are equivalent:
\begin{itemize}
\setlength\itemsep{0.2em}

\item[a)] $M$ possesses a homothetic Euler vector field $X$.

\item[b)] There exists a function $\kappa$ on $M$ such that $g = \nabla^2 \kappa$. 

\item[c)] There exists a function $\kappa$ on $M$ such that its gradient vector field with respect to the metric, \mbox{$\iota_X g = d\kappa$}, satisfies the homothetic property $\smash{ \mathcal{L}_X g = 2 \hp g }$.

\end{itemize}
If $M$ is assumed in addition to be hyperk\"ahler, then we can add to this list the conditions
\begin{itemize}[resume]
\item[d)] There exists a function $\kappa$ on $M$ which is simultaneously a K\"ahler potential for every hyperk\"ahler K\"ahler 2-form with respect to its corresponding complex structure, \textit{i.e.} 
\begin{equation}
\omega_i = \frac{1}{2} d(d\kappa I_i) \mathrlap{.}
\end{equation}

\item[e)] There exist four vector fields $X, X_1, X_2, X_3 \in TM$ and a function $\kappa$ on $M$ such that
\begin{equation} \label{4vec_HKC}
\iota_{X_i} \omega_j = \varepsilon_{ijk} \hp \iota_X \omega_k - \pt \delta_{ij} \hp d\kappa
\qquad \text{and} \qquad
\mathcal{L}_X \omega_k = 2 \pt \omega_k \mathrlap{.}
\end{equation}
\end{itemize} %
\end{proposition}

\begin{remark}
A hyperk\"ahler manifold $M$ with this structure is known as a \textit{hyperk\"ahler cone}. Due to the property $(d)$ the function $\kappa$ is called in this context a \textit{hyperk\"ahler potential}. Homothetic Euler vector fields were first introduced in a physical setting in relation to theories of \mbox{$N=2$} hypermultiplets in four spacetime dimensions invariant under superconformal symmetries by de Wit, Kleijn and Vandoren in \cite{MR1744865} (see also \cite{MR1825214}). The characterizations $(b)$ and $(d)$ precede this definition and are due to Swann \cite{MR1096180}. Through the property $(e)$, involving exterior algebra conditions on the hyperk\"ahler symplectic \mbox{2-forms}, we add to this collection of characterizations one which is closest in spirit to the approach that we have developed in these notes. 
\end{remark}

\begin{proof}

We begin with the case when $M$ is simply Riemannian or pseudo-Riemannian, with metric $g$ and Levi-Civita connection $\nabla$. Let us recall, first, that for any \mbox{$X,Y,Z \in TM$} we have the identity $\smash{ (\mathcal{L}_X g)(Y,Z) = g(\nabla_YX,Z) + g(\nabla_ZX,Y) }$. And, second, that for any function $\kappa$, its Hessian tensor is given by \mbox{$\nabla^2_{Y,Z} \kappa = g(\nabla_YX,Z)$} for any \mbox{$Y, Z \in TM$}, where $X$ represents now the gradient vector field of $\kappa$ with respect to the metric; this formula holds in fact for any metric-preserving affine connection. 

\mbox{$ a) \Rightarrow b), c)$} Assuming the condition $(a)$ holds, define the function
\begin{equation} \label{kappa_gee}
\kappa = \frac{1}{2} \hp g(X,X) \mathfrak{.}
\end{equation}
Then, for any vector field \mbox{$Y \in TM$} we have $\smash{ \nabla_Y\kappa = g(X,\nabla_YX) = g(X,Y) }$, which is to say, $X$ is the gradient of $\kappa$ with respect to the metric. The first and second identities above imply then immediately the conditions $(c)$ and $(b)$, respectively. 

\mbox{$ b) \Rightarrow a)$} Let in this case $X$ be, by definition, the gradient vector field of the function $\kappa$ with respect to the metric. The implication follows then promptly from the Hessian formula above.

\mbox{$ c) \Rightarrow a)$} The homothetic condition $\smash{ \mathcal{L}_X g = 2 \hp g }$ implies on one hand, by way of the first Levi-Civita identity above, that \mbox{$g(\nabla_YX,Z) + g(\nabla_ZX,Y) = 2 \hp g(Y,Z)$} for any $Y,Z \in TM$. On another hand, from the symmetry of the Hessian tensor it follows, by way of the second identity, that \mbox{$g(\nabla_YX,Z) = g(\nabla_ZX,Y)$}. One gets then that \mbox{$g(\nabla_YX,Z) = g(Y,Z)$} for any $Y,Z \in TM$, which then by the fact that the metric $g$ is non-degenerate implies the condition $(a)$. 

Let us assume  now that the manifold $M$ is, moreover, hyperk\"ahler. 

\mbox{$b) \Rightarrow d)$}  We have in this case \mbox{$\omega_i(Y,Z) = g(Y,I_iZ) = \nabla^2_{Y,I_iZ} \kappa = \nabla_Y \nabla_Z \kappa - \nabla_{\nabla_Y(I_iZ)} \kappa$}, for any \mbox{$Y,Z \in TM$}. Since the hyperk\"ahler complex structures are preserved individually by the Levi-Civita connection we can replace the last term (ignoring the sign) with $\smash{  \nabla_{I_i\nabla_YZ} \kappa }$. In differential form language, the result can be written rather more nicely in the form $\smash{ \omega_i = \frac{1}{2} d_{\nabla} (d\kappa I_i) }$. As the Levi-Civita connection is torsion-free, which translates into an index symmetry for the Christoffel symbols, we can eventually replace in this expression $d_{\nabla}$ with the usual exterior derivative, $d$. 

\mbox{$d) \Rightarrow e)$} Given the potential function $\kappa$ on $M$, define the vector fields \mbox{$X, X_1,X_2,X_3$} as the gradients of $\kappa$ with respect to the metric respectively the three hyperk\"ahler symplectic forms, as follows:
\begin{equation}
\iota_{X_1}\omega_1 = \iota_{X_2}\omega_2 = \iota_{X_3}\omega_3 = - \hp \iota_Xg = - \hp d\kappa \mathrlap{.}
\end{equation}
By using the quaternionic properties of the hyperk\"ahler complex structures we get then that $\smash{ \iota_{X_1}\omega_2 = - \hp \iota_{X_1}\omega_1 I_3 = \iota_Xg I_3 = \iota_X\omega_3 }$. Cyclically permuting the indices in this last argument yields a set of formulas which, together with the definitions of the four vector fields, can be assembled in the form of the first condition \eqref{4vec_HKC}. To obtain the second condition \eqref{4vec_HKC}, on the other hand, note that we have $\smash{ \iota_X \omega_i = \iota_Xg I_i =  d\kappa \hp I_i }$. Then $\mathcal{L}_X \omega_i = d\hp \iota_X \omega_i = d(d\kappa \hp I_i) = 2 \pt \omega_i $.

\mbox{$e) \Rightarrow c)$} Consider, say, the following off-diagonal component of the first condition \eqref{4vec_HKC}: $\smash{ \iota_{X_1} \omega_2 = \iota_X \omega_3 }$. From this it follows that $\smash{ X_1 = \iota_X \omega_3 \pt \omega_{\mathrlap{2}}{}^{\raisebox{0.3pt}{$\scriptstyle -1$}} = \omega_{\mathrlap{2}}{}^{\raisebox{0.3pt}{$\scriptstyle -1$}} \omega_3 X = - I_1X }$, where in the last step we made use of the second expression for $I_1$ in \eqref{I_i-two_exprs}. Cyclically permuting the indices in this argument gives us
\begin{equation}
X_i = - \hp I_iX \mathrlap{.}
\end{equation}
Substituting this back into the condition we get then from the diagonal components that \mbox{$\iota_X g = d\kappa$}. On another hand, the second condition \eqref{4vec_HKC} implies immediately by way of any one of the expressions \eqref{g1g2g3} for the metric that $\smash{ \mathcal{L}_X g = 2 \hp g }$. This concludes the proof of the chain of equivalences.
\end{proof}

Hyperk\"ahler cones are naturally foliated by four-dimensional leaves. More precisely, we have:

\begin{proposition}
Let $M$ be a hyperk\"ahler cone. Then the four vector fields $X,\pt X_1,X_2,X_3$ at point $(e)$ of Proposition~\ref{HKC-criterions} generate the \mbox{$\mathbb{R} \oplus \mathfrak{sp}(1)$} algebra
\begin{equation} \label{so(3)-alg}
[X_i,X_j] = 2 \pt \varepsilon_{ijk} X_k
\qquad\qquad
[X_i,X \hp] = 0
\end{equation}
with $X$ homothetic Euler and \mbox{$X_1, X_2, X_3$} Killing vector fields of rotating type.
\end{proposition}

\begin{proof}

Acting with an exterior derivative on the first condition \eqref{4vec_HKC} and using the closure of the $\omega_i$ to pass to Lie derivatives by way of Cartan's homotopy formula we get that
\begin{equation} \label{HKC_rot}
\mathcal{L}_{X_i} \omega_j = 2 \pt \varepsilon_{ijk} \pt \omega_k  \mathrlap{.}
\end{equation}
This means that each one of the vector fields \mbox{$X_1, X_2, X_3$} generates a rotational action on $M$ and is therefore Killing. We then have
{\allowdisplaybreaks
\begin{align}
\iota_{[X_i,X_j]} \omega_l & = (\iota_{X_i}\mathcal{L}_{X_j} - \iota_{X_j}\mathcal{L}_{X_i} + [d,\iota_{X_i}\iota_{X_j}]) \hp \omega_l \\[2pt]
& = 2 \pt \varepsilon_{jlk} \hp \iota_{X_i} \omega_k - 2 \pt \varepsilon_{ilk} \hp \iota_{X_j} \omega_k - 2\pt \varepsilon_{ijl}\hp d\kappa  \nonumber \\
\intertext{where we used that $\omega_l(X_i,X_j) = g(I_iX,I_lI_jX) = - \hp g(X,I_iI_lI_jX) = \varepsilon_{ilj} \hp g(X,X) = - 2 \pt \varepsilon_{ijl} \hp \kappa$, as well as the closure of $\omega_l$; resorting again to the first condition \eqref{4vec_HKC} and the Jacobi property of the $\varepsilon$-symbol returns in the end}
& = 2 \pt \varepsilon_{ijk} \hp \iota_{X_k} \omega_l \mathrlap{.} \nonumber
\end{align}
}%
Since the $\omega_l$ are non-degenerate, this is equivalent to the first algebraic condition \eqref{so(3)-alg}. The remaining condition \eqref{so(3)-alg} can be derived by a similar argument.
\end{proof}

\section{Swann bundles} \label{sec:Sw_bdls}

In \cite{MR1096180} Swann introduced a powerful tool for the study of quaternionic K\"ahler manifolds by showing that over each such manifold one can construct a quaternionic bundle whose total space carries a hyperk\"ahler cone structure encoding the quaternionic K\"ahler geometry of the base. This allows one to describe the geometry of manifolds which in general have no integrable complex structures, apart perhaps from accidental ones inessential to their quaternionic K\"ahler structure, in terms of the hyperk\"ahler geometry of a space of four real dimensions higher possessing a wealth of integrable complex structures. Before we proceed to review the details of Swann's construction it is useful to recall a few facts about (a $\mathbb{Z}_2$ quotient of) the multiplicative group of non-zero quaternions, $\mathbb{H}^{\times}$, which plays a key role in this story.

\subsection{Two structures on the group $\mathbb{H}^{\times}\npt /\mathbb{Z}_2$}

\subsubsection{}

Consider an embedding of the group $Sp(1)$ of unitary quaternions into $\mathbb{H}^{\times}$. Clearly, every non-zero quaternion can be uniquely represented as a positive \pagebreak real number, its absolute value, times a unit quaternion; this gives a natural isomorphism 
\begin{equation}
\mathbb{H}^{\times} \cong \pt \mathbb{R}^+ \npt \otimes Sp(1) \mathrlap{.}
\end{equation}
Thus, $\mathbb{H}^{\times}$ has the structure of a Lie group, with Lie algebra generated by the standard quaternionic basis \mbox{$1$, $\mathbf{i}$, $\mathbf{j}$, $\mathbf{k}$}. In what follows we will denote the elements of this basis uniformly with \mbox{$u_0$, $u_1$, $u_2$, $u_3$}, respectively\,---\,or, simply, $u_a$.  The indices \mbox{$a,b,\dots \npt$} will be assumed to run from 0 to 3 while the indices \mbox{$i, j, k, \dots$} to run, as until now, from 1 to 3, and the summation convention over repeated indices will continue to be implied. 

The adjoint representation of $\mathbb{H}^{\times}$ is four-dimensional and defined by $\smash{ q^{-1} \hp u_a \pt q = R_{ab}(q) \hp u_b }$, for any \mbox{$q \in \mathbb{H}^{\times}$}. It satisfies the properties
\begin{align}
1. \ \, & R_{ab}(\lambda \hp q) = R_{ab}(q) \text{ \pt for any $\lambda \in \mathbb{R}^{\times}$} \\
2. \ \, & R_{ab}(q^{-1}) = R_{ba}(q) \mathrlap{.} \nonumber
\end{align}
The first one follows immediately from the definition. Taking in it in particular $\smash{ \lambda = 1/|q| }$ and \mbox{$\lambda = -1$}, one can see that the adjoint representation descends to a representation of the rotation group $SO(3)$, whose double cover is $Sp(1)$. This is reducible: one can write it in block-diagonal form as the direct sum of the one-dimensional trivial representation and the three-dimensional irreducible representation $R_{ij}(q)$ of $SO(3)$. On the other hand, by the associativity of quaternionic multiplication and the multiplicative property of the norm, $\smash{ |u_a| = |q^{-1} \hp u_a \hp q| = |R_{ab}(q) \hp u_b| }$, which is to say, adjoint transformations preserve the quaternionic norm. They preserve hence also the associated inner product, and the second relation is the expression of this orthogonality property.

The Lie algebra (\textit{i.e.}\! quaternion)-valued left and right-invariant Cartan--Maurer 1-forms of $\mathbb{H}^{\times}$ are given by
{\allowdisplaybreaks
\begin{alignat}{2}
&\sigma^L = q^{-1} \nhp dq = \sigma_{\mathrlap{a}}{}^L \hp u_a 
&\qquad\qquad&
\sigma^R = q \hp dq^{-1} \npt = \sigma_{\mathrlap{a}}{}^R \hp u_a \\
\intertext{for any element \mbox{$q \in \mathbb{H}^{\times}$}. Thus defined, each of these satisfy the same set of Cartan--Maurer equations, which in components read}
& d\sigma_{\mathrlap{i}}{}^L + \varepsilon_{ijk} \pt \sigma_{\mathrlap{j}}{}^L \!\wedge\npt \sigma_{\mathrlap{k}}{}^L = 0 && d\sigma_{\mathrlap{i}}{}^R + \varepsilon_{ijk} \pt \sigma_{\mathrlap{j}}{}^R \!\wedge\npt \sigma_{\mathrlap{k}}{}^R = 0 \\
& d\sigma^L_0 = 0 && d\sigma^R_0 = 0 \mathrlap{.} \nonumber
\end{alignat}
}%
The left and right-invariant Cartan--Maurer 1-forms are related to each other by an adjoint transformation up to a minus sign: $\smash{ \sigma_{\mathrlap{a}}{}^L = - \hp R_{ab}(q^{-1}) \hp \sigma_{\mathrlap{b}}{}^{\raisebox{0.3pt}{$\scriptstyle R$}}    }$. Note that in particular we have
\begin{equation}
\sigma^L_0 = - \hp \sigma^R_0 = \frac{d|q|^2}{2|q|^2} \mathrlap{.}
\end{equation}

Dual to the Cartan--Maurer 1-forms one has the left and right-invariant vector fields of the Lie group $\mathbb{H}^{\times}$. The duality requirement with respect to the corresponding set of invariant 1-forms determines them completely. They form two commuting copies of the Lie algebra \mbox{$\mathbb{R} \oplus \mathfrak{sp}(1)$}, \textit{i.e.}
\begin{alignat}{2} \label{ells}
[\pt  \ell_{\mathrlap{i}}{}^L, \ell_{\mathrlap{j}}{}^L \pt] & = 2\pt \varepsilon_{ijk} \pt \ell_{\mathrlap{k}}{}^L 
&\qquad\qquad& [\pt \ell_{\mathrlap{i}}{}^R,\ell_{\mathrlap{j}}{}^R \pt] = 2\pt \varepsilon_{ijk} \pt \ell_{\mathrlap{k}}{}^R \\
[\pt \ell_{\mathrlap{i}}{}^L,\ell^L_0 \pt] & = 0 
&& [\pt \ell_{\mathrlap{i}}{}^R,\ell_0^R \pt] = 0 \nonumber
\end{alignat}
which one may think of as the dual Cartan--Maurer equations, and the two sets of vector fields are related by $\smash{ \ell_{\mathrlap{a}}{}^L = - \hp R_{ab}(q^{-1}) \pt \ell_{\mathrlap{b}}{}^{\raisebox{0.3pt}{$\scriptstyle R$}}  }$. In particular, we have 
\begin{equation}
\ell^L_0 = - \hp \ell^R_0 = q_0\frac{\partial}{\partial q_0} + q_1\frac{\partial}{\partial q_1} + q_2\frac{\partial}{\partial q_2} + q_3\frac{\partial}{\partial q_3}
\end{equation}
where \mbox{$q_0$, $q_1$, $q_2$, $q_3$} are the components of the quaternion $q$.

\subsubsection{} \label{ssec:qu_in_pr}

In addition to the Lie group structure, $\mathbb{H}^{\times}$ carries also two \textit{hypercomplex structures}. That is, its tangent bundle is equipped with two natural actions of the algebra of imaginary quaternions, with the three generators defining integrable complex structures. Viewing complex structures in their dual guise as endomorphisms of the cotangent bundle, and in accordance with our convention that complex structures act from the right on 1-forms, these actions are induced by right and left quaternionic multiplication on a quaternionic coframe of $T^*\mathbb{H}^{\times}$ as follows:
\begin{equation}
dq \pt \mathcal{I}_{\hp\mathrlap{i}}{}^L = dq \pt u_i
\qquad\qquad
dq \pt \mathcal{I}_{\hp\mathrlap{i}}{}^R = - \hp u_i \hp dq \mathrlap{.}
\end{equation}
They are respectively left and right-invariant, and the two sets of hypercomplex generators commute with each other. 

Let $\smash{ \langle x, y \rangle = \frac{1}{2} (|x+y|^2 - |x|^2 - |y|^2)  }$ for any \mbox{$x,y \in \mathbb{H}$} be the inner product induced on the space of quaternions by the quaternionic norm. Observe that in the standard quaternionic basis this takes the Euclidean form \mbox{$\langle u_a, u_b \hp \rangle = \delta_{ab}$} and, moreover, that in terms of it the alternating quaternionic products of any two quaternions $x,y \in \mathbb{H}$ admit the representations
\begin{equation}
\bar{x} \hp y = \langle \hp x\hp u_a,y \hp \rangle \hp u_a
\qquad\qquad
x \hp \bar{y} = \langle \hp x, u_a \hp y \hp \rangle \hp u_a
\end{equation}
where the overhead bar indicates the operation of quaternionic conjugation. One can then verify that in the alternative coframes provided by the Cartan--Maurer 1-forms and their respective dual vector frames the generators of the two hypercomplex structures may be expressed as
\begin{equation}
\mathcal{I}_{\hp\mathrlap{i}}{}^L = \langle u_au_i,u_b \hp \rangle \, \ell_{\mathrlap{a}}{}^L \otimes \sigma_{\mathrlap{b}}{}^L
\qquad\qquad
\mathcal{I}_{\hp\mathrlap{i}}{}^R = \langle u_a,u_iu_b \hp \rangle \, \ell_{\mathrlap{a}}{}^R \otimes \sigma_{\mathrlap{b}}{}^R \mathrlap{.}
\end{equation}
By resorting to the orthogonality property of the adjoint representation and the left-right transition formulas one can also recast these in the form 
\begin{equation} \label{hcstrs_H*_2}
\hspace{-10pt} 
\mathcal{I}_{\hp\mathrlap{i}}{}^L = R_{ij}(q^{-1} ) \langle u_au_j,u_b \hp \rangle \, \ell_{\mathrlap{a}}{}^R \otimes \sigma_{\mathrlap{b}}{}^R
\qquad\qquad
\mathcal{I}_{\hp\mathrlap{i}}{}^R = R_{ij}(q) \langle u_a,u_ju_b \hp \rangle \, \ell_{\mathrlap{a}}{}^L \otimes \sigma_{\mathrlap{b}}{}^L \mathrlap{.} 
\end{equation}

\subsubsection{}

Finally, note that both the Lie group structure and the two hypercomplex structures of $\mathbb{H}^{\times}$ considered above descend on the quotient of $\mathbb{H}^{\times}$ with respect to the $\mathbb{Z}_2$-action generated by \mbox{$q \mapsto -q$}. In this case we have 
\begin{equation}
\mathbb{H}^{\times} \npt /\mathbb{Z}_2 \cong \pt \mathbb{R}^+ \npt  \otimes SO(3) \mathrlap{.}
\end{equation}
This is the conformal special orthogonal group in three dimensions.

\subsection{The Swann bundle of a quaternion K\"ahler manifold} 

\subsubsection{}

Let $M$ be a quaternionic K\"ahler manifold, with a choice of local frame \mbox{$\omega_1$, $\omega_2$, $\omega_3$}  for its $SO(3)$-bundle of 2-forms $Q$ and a principal connection with local connection \mbox{1-forms} \mbox{$\theta_1$, $\theta_2$, $\theta_3$} on the associated principal bundle $\mathcal{Q}$. For the considerations which follow it will be convenient to regard these two triplets as the components of two imaginary quaternion-valued differential forms, \mbox{$\omega = \omega_i \hp u_i$} and \mbox{$\theta = \theta_i \hp u_i$}. In this quaternionic formalism then the Einstein condition \eqref{Einst_cond} of Theorem~\ref{QK-criterion-2} can be cast in the form
\begin{equation} \label{Einst_cond-q}
d\theta + \theta \wedge \theta = s \pt\hp \omega
\end{equation}
where the wedge symbol denotes here the natural extension of the operation of exterior product to quaternion-valued differential forms. 

By definition, the \textit{Swann bundle} over the quaternion K\"ahler manifold $M$ is the associated bundle
\begin{equation}
\mathcal{U}(M) = \mathcal{Q} \times_{SO(3)} (\mathbb{H}^{\times}\npt /\mathbb{Z}_2) \longrightarrow M \mathrlap{.}
\end{equation}
On its total space one assembles the following $\Im \mathbb{H}$-valued 2-form:
\begin{equation} \label{Sw-Om-1}
\Omega = s \, \bar{q} \pt \omega \hp q + (\overline{dq + \theta q}) \nhp\wedge\nhp (dq + \theta q) \mathrlap{.}
\end{equation} 

Although this is made up of local quantities, it is globally defined. Indeed, consider an open covering $\mathcal{C}$ of $M$ and two overlapping open sets \mbox{$U,V \in \mathcal{C}$}. On the intersection \mbox{$U \cap V$} the corresponding local frames of $Q$ are related by an $SO(3)$ transition function, that is, $\smash{ \omega_{iV} = R_{ij}(u_{VU}) \pt \omega_{jU} }$, with \mbox{$u_{VU} \in Sp(1)$} being either one of the two opposite-sign unit quaternions which determine the rotation matrix through the double cover map \mbox{$Sp(1) \rightarrow SO(3)$} (recall that $\smash{ R_{ij}(u_{VU}) = R_{ij}(- u_{VU}) }$). In quaternionic notation this transition relation reads \mbox{$ \omega_V = u_{VU} \pt\hp \omega_U \hp u_{VU}{}^{\mathllap{\raisebox{1pt}{$\scriptstyle -1$} \hspace{1.5pt}} } $}. On another hand, the transition relation for the connection 1-form is \mbox{$\theta_V = u_{VU} \pt \theta_U u_{VU}{}^{\mathllap{\raisebox{1pt}{$\scriptstyle -1$} \hspace{1.5pt}} } + u_{VU} \hp du_{VU}{}^{\mathllap{\raisebox{1pt}{$\scriptstyle -1$} \hspace{1.5pt}} } $}. Thus, if the quaternionic fiber coordinates patch up according either to the rule $\smash{ q_V =  u_{VU} q_{\hp U} }$ or  to $\smash{ q_V =  - \hp u_{VU} q_{\hp U} }$, then we have \mbox{$\Omega_V = \Omega_{\hp U}$}, which then by the arbitrariness in the choice of $U$ and $V$ implies that $\Omega$ is globally defined on $\mathcal{U}(M)$. Note also that in either case $\smash{ |\hp q_V| = |q_{\hp U}| }$, so the function $|q|$ is globally defined on $\mathcal{U}(M)$ as well. 

In fact, there are more globally defined quantities one can consider. To see that, let \mbox{$\smash{ \sigma_{a U}{}^{\hspace{-11pt} \raisebox{1pt}{$\scriptstyle L$}} \hspace{4pt} }$ and $\smash{ \sigma_{a U}{}^{\hspace{-11pt} \raisebox{1pt}{$\scriptstyle R$}} \hspace{4pt} }$} for any \mbox{$U \in \mathcal{C}$} be the pull-backs by the local trivialization map of the left and right-invariant Cartan--Maurer 1-forms from $\mathbb{H}^{\times}/\npt \mathbb{Z}_2$ to the fibers above $U$\,---\,and then let $\smash{ \ell_{a U}{}^{\mathllap{\raisebox{1pt}{$\scriptstyle L$} \hspace{5.5pt}} } }$ and $\smash{ \ell_{a U}{}^{\mathllap{\raisebox{1pt}{$\scriptstyle R$} \hspace{5.2pt}} } }$ be their respective dual vector fields. Using the above patching rules for the quaternionic fiber coordinates one can then check that on any intersection $U \cap V$ we have \mbox{$\ell_{a V}{}^{\mathllap{\raisebox{1pt}{$\scriptstyle L$} \hspace{5.5pt}} } = \ell_{a U}{}^{\mathllap{\raisebox{1pt}{$\scriptstyle L$} \hspace{5.5pt}} }$} and \mbox{$\ell_{a V}{}^{\mathllap{\raisebox{1pt}{$\scriptstyle R$} \hspace{5.5pt}} } = R_{ab}(u_{VU})\hp \ell_{b \hp U}{}^{\mathllap{\raisebox{1pt}{$\scriptstyle R$} \hspace{4.7pt}} }$}. Thus, in particular, the left-invariant vertical vector fields are also globally defined, as one would expect to be the case for a bundle associated to a right principal bundle. 

Let \mbox{$\Omega_1$, $\Omega_2$, $\Omega_3$} be the real components of $\Omega$. The rationale behind the above definition is provided by the following result: 

\begin{theorem}[Swann \cite{MR1096180}]
Let $M$ be a quaternionic K\"ahler manifold. The bundle $\mathcal{U}(M)$ endowed with the 2-forms \mbox{$\Omega_1$, $\Omega_2$, $\Omega_3$} constructed as above is a hyperk\"ahler cone with homothetic Euler vector field $\smash{ \ell_0^L }$ and hyperk\"ahler potential \mbox{$|q|^2$}.
\end{theorem}

\begin{proof}

The fact that the three 2-forms determine a hyperk\"ahler structure on $\mathcal{U}(M)$ can be seen fairly quickly by checking that they satisfy the conditions of Theorem~\ref{QK-criterion-2}. Thus, one one hand, the verification of the differential condition of the theorem is considerably facilitated if one observes that Swann's formula \eqref{Sw-Om-1} can be equivalently rearranged in the form
\begin{equation}
\Omega = \bar{q} \hp ( s \, \omega - d\theta - \theta \wedge \theta \hp ) q - d \hp [\hp \bar{q} \hp (\sigma^R \!- \theta \hp) q \pt ] \mathrlap{.}
\end{equation}
By virtue of the Einstein condition \eqref{Einst_cond-q} $\Omega$ is then clearly locally exact and hence closed. On the other hand, for the algebraic condition, note that the same formula \eqref{Sw-Om-1} can also be equivalently rewritten as 
\begin{equation} \label{Om_iii}
\Omega = \bar{q} \hp [\hp s\, \omega + (\overline{\sigma^R \!- \theta \hp}) \nhp\wedge\nhp (\sigma^R \!- \theta \hp) \hp] \hp q
\end{equation}
or, in components, with notations introduced in the previous subsection, 
\begin{equation}
\Omega_{\hp i} = |q|^2 R_{ij}(q^{-1}) \hp [ \hp s \, \omega_j +  \langle u_au_j,u_b \hp \rangle \pt (\sigma_{\mathrlap{a}}{}^R \npt - \theta_a) \nhp\wedge\nhp (\sigma_{\mathrlap{b}}{}^R \npt - \theta_b)\hp]
\end{equation}
where it is understood that \mbox{$\theta_0 = 0$}. The last expression, in particular, displays the $\Omega_i$ as block-diagonally split along the horizontal and vertical subspaces determined locally by the fiber structure. So then, in order to show that they satisfy the algebraic condition of Theorem~\ref{QK-criterion-2}, it suffices to verify that their horizontal and vertical components satisfy it separately.  (It is worth noting also that the twisting prefactor can can be safely ignored since the algebraic condition is preserved by simultaneous rescalings and $SO(3)$ rotations, see the Remark following Theorem \ref{HK-criterion}.) This is indeed ensured, on the horizontal side, by the fact that the quaternionic K\"ahler \mbox{2-forms} $\omega_i$ satisfy themselves the algebraic condition, and on the vertical side, by the fact that the map \mbox{$u_i \mapsto \langle u_au_i, u_b \hp \rangle$} gives a four-dimensional real matrix representation of the standard  imaginary quaternions. 

Regarding the cone structure, note that Swann bundles possess a natural \mbox{$\mathbb{H}^{\times}\nhp/\mathbb{Z}_2${\hp-\hp}action} induced by right quaternionic multiplication in the fiber: taking \mbox{$q \mapsto q \hp v$} with \mbox{$v \in \mathbb{H}^{\times}$} constant results in \mbox{$\Omega \mapsto \bar{v} \pt \Omega \pt v$} and \mbox{$G \mapsto |v|^2 G$}. The $\mathbb{Z}_2$-quotient reflects the invariance of these transformations under \mbox{$v \mapsto - v$}. The infinitesimal generators of this action are the  left-invariant vertical vector fields, and we have
\begin{equation}
\iota_{\ell^L_i} \Omega_j = \varepsilon_{ijk} \pt \iota_{\ell^L_0} \Omega_k - \hp \delta_{ij} \pt d \hp |q|^2
\qquad\text{and}\qquad
\mathcal{L}_{\ell^L_0} \Omega_k = 2 \pt \Omega_k
\end{equation}
which then by part $(e)$ of Proposition \ref{HKC-criterions} means that the bundle $\mathcal{U}(M)$ has a hyperk\"ahler cone structure with hyperk\"ahler potential \mbox{$|q|^2$} and homothetic Euler vector field $\smash{ \ell_0^L }$. 
\end{proof}

Observe moreover, in line with our stance of regarding the metric as a composite object, that from the formulas \eqref{Sw-Om-1} and \eqref{Om_iii} for the hyperk\"ahler symplectic forms one can work out for the hyperk\"ahler metric on $\mathcal{U}(M)$ the expressions
{\allowdisplaybreaks
\begin{align}
G & = s \pt g \hp |q|^2 + | dq + \theta q |^2 
\intertext{and}
G & = |q|^2 \pt (s \pt g + | \sigma^R \!- \theta \hp |^2 ) \label{G_iii} 
\end{align}
}%
respectively. From either one of these it is clear that if the signature of the quaternionic K\"ahler metric $g$ is $\smash{ (4n_+,4n_-) }$, then the signature of Swann's hyperk\"ahler metric $G$ is either $\smash{ (4n_+ \npt +4,4n_-) }$ if \mbox{$s>0$}, or $\smash{ (4n_- \npt + 4,4n_+) }$ if \mbox{$s<0$}. 

Likewise, one can also work out the corresponding hyperk\"ahler complex structures on $\mathcal{U}(M)$ to obtain
\begin{equation}
I_i^{\phantom{.}} = R_{ij}(q^{-1}) \hp [\pt \mathscr{I}_{\mathrlap{j}}{}^{\pt H} + \langle u_au_j,u_b \hp \rangle \, \ell_{\mathrlap{a}}{}^R \otimes (\sigma_{\mathrlap{b}}{}^R \npt - \theta_b) \hp ]
\end{equation}
where the $\smash{ \mathscr{I}_{\mathrlap{i}}{}^{\pt \raisebox{0.4pt}{$\scriptstyle H$}} }$ represent the horizontal lifts to $\mathcal{U}(M)$ of the almost complex structures $\smash{ \mathscr{I}_i }$ associated to the triplet of 2-forms $\omega_i$ on $M$. (One-forms on $M$ can be lifted to $\mathcal{U}(M)$ by pullback, while the horizontal lift of a vector field $\mathscr{X}$ on $M$ to $\mathcal{U}(M)$ is given by $\smash{ \mathscr{X}^H = \mathscr{X} + (\iota_{\mathscr{X}} \theta_i) \hp \ell_{\mathrlap{i}}{}^{\raisebox{0.3pt}{$\scriptstyle R$}} }$; elements of $\text{End}(M) \cong TM \otimes \Lambda^1\npt M$ such as the $\smash{ \mathscr{I}_i }$ are then lifted in accordance with the tensor product rule.) This shows that the complex structures $I_i$, too, split block-diagonally along the local fibration structure into horizontal and vertical components, with the first ones induced by the quaternionic structure on the quaternionic K\"aher base and twisted by a fiber coordinate-dependent $SO(3)$ rotation, and the second ones induced by the left-invariant hypercomplex structure on the \mbox{$\mathbb{H}^{\times}\nhp/\mathbb{Z}_2$} fiber (see in particular the first formula \eqref{hcstrs_H*_2}). This is reminiscent of the way in which the integrable complex structure on the twistor space of a hyperk\"ahler manifold was constructed in \cite{MR877637}. 


The previous theorem admits the following converse, which we give here without proof and refer the reader to Theorem~5.9 in \cite{MR1096180}:

\begin{theorem}[Swann \cite{MR1096180}] \label{HKC=Sw}
Any hyperk\"ahler cone with positive hyperk\"ahler potential, whose four canonical vector fields generate a locally free $\smash{\mathbb{H}^{\times} \nhp / \mathbb{Z}_2}$-action is locally homothetic to a Swann bundle.
\end{theorem}

Note that when the metric of a hyperk\"ahler cone is positive definite, the formula \eqref{kappa_gee} makes it clear that its hyperk\"ahler potential can always be chosen to be strictly positive. In indefinite signature, on the other hand, the hyperk\"ahler potential can always be made at least locally positive, if it is not already so, by flipping the signature of the metric.

\subsubsection{}

Let us consider now how Killing symmetries fit into this construction. The main result in this regard is the following:

\begin{theorem} \label{Sw-symm}
Let $M$ be a quaternionic K\"ahler manifold with Swann bundle $\mathcal{U}(M)$. Then any Killing vector field $\mathscr{X}$ on $M$ can be lifted to a tri-Hamiltonian vector field $X$ on $\mathcal{U}(M)$ which commutes with the \mbox{$\mathbb{H}^{\times}\nhp/\mathbb{Z}_2${\hp-\hp}action}, and conversely, any tri-Hamiltonian vector field $X$ on $\mathcal{U}(M)$ which commutes with the \mbox{$\mathbb{H}^{\times}\nhp/\mathbb{Z}_2${\hp-\hp}action} descends to a Killing vector field $\mathscr{X}$ on $M$. If $\smash{ \nu_{\mathscr{X}i} }$ is the quaternionic K\"ahler moment map of $\mathscr{X}$, we have 
\begin{equation} \label{X_Sw}
X = \mathscr{X}^H \npt + \nu_{\mathscr{X}}{}_{\! i} \pt \ell_{\mathrlap{i}}{}^R \mathrlap{.}
\end{equation}
Moreover, there exists a choice of hyperk\"ahler moment map $\mu_{X i}$ for $X$ such that the two moment maps are related by 
\begin{equation} \label{Sw_mom_maps}
\vec{\mu}_X = \bar{q} \pt \vec{\nu}_{\mathscr{X}} q \mathrlap{.}
\end{equation}
Both the descent and the lift maps preserve the Lie bracket.
\end{theorem}

\begin{remark}
In equation \eqref{Sw_mom_maps} we have deliberately blurred the line between $\mathbb{R}^3$-vectors and imaginary quaternions in view of the isomorphism $\smash{ \Im \mathbb{H} \cong \mathbb{R}^3 }$. That is, despite employing the traditional $\mathbb{R}^3$-vector notation, we clearly regard the two moment maps as $\Im \mathbb{H}$-valued functions. In what follows we will institutionalize this practice and often resort to this notational ambivalence when the context allows for an unequivocal interpretation. 
\end{remark}

\begin{proof}

In what follows we denote, for conciseness, the $\mathbb{H}$-valued vertical 1-form $\smash{ \sigma^R \! - \theta }$ by $\alpha$, with $\smash{ \alpha_0 }$ and $\smash{ \vec{\alpha}  }$ designating its real and quaternionic imaginary components, respectively. Let us begin then by stating the following lemma: for any vector field $X$ on $\mathcal{U}(M)$ preserving the hyperk\"ahler potential\,---\,that is, such that \mbox{$X(|q|^2) = 0$}, or equivalently, such that \mbox{$\iota_X\alpha_0 = 0$}\,---\,we have
\begin{equation} \label{i_X_Om^Sw}
\iota_X \vec{\Omega} = \bar{q} \hp [\pt s \, \iota_X \vec{\omega} - d(\iota_X\vec{\alpha} \hp) - 2 \hp \vec{\theta} \times \npt (\iota_X\vec{\alpha} \hp) ] q + d \hp [\pt \bar{q} \hp (\iota_X\vec{\alpha} \hp) q \pt] \mathrlap{.}
\end{equation}
This follows by writing the formula \eqref{Om_iii} in the form \mbox{$\vec{\Omega} = \bar{q} \hp (\hp s\, \vec{\omega} + 2\pt \alpha_0 \wedge \vec{\alpha} - \vec{\alpha} \wedge \vec{\alpha} \,)  q$} and using the identity
\begin{equation} \label{dqrbq}
d(\hp \bar{q} \pt \vec{\rho} \, q \hp) = \bar{q} \hp (\hp d\vec{\rho} - 2 \pt \sigma^R_0 \vec{\rho} + 2 \pt \vec{\sigma}^{\hp R} \!\times \npt \vec{\rho} \pt\,)  q
\end{equation}
valid for any triplet of functions $\smash{ \vec{\rho} }$. 

Assume now that $\mathscr{X}$ is some Killing vector field on $M$. Let $\smash{ \vec{\nu}_{\mathscr{X}} }$ be its quaternionic K\"ahler moment map and define on $\mathcal{U}(M)$ the vector field $X$ to be the sum of the horizontal lift of $\mathscr{X}$ to $\mathcal{U}(M)$ and a vertical component as in equation \eqref{X_Sw}. Observe then on one hand that $X$, thus defined, commutes automatically with the generators $\smash{ \ell^L_a }$ of the \mbox{$\mathbb{H}^{\times}\nhp/\mathbb{Z}_2$-action} on $\mathcal{U}(M)$. On another hand we have $\smash{ \iota_X\alpha_0 = 0 }$ and $\smash{ \iota_X\vec{\alpha} = \vec{\nu}_{\mathscr{X}} }$, and so the formula \eqref{i_X_Om^Sw} applies and gives us immediately
\begin{equation} \label{i_X_Om_nu}
\iota_X \vec{\Omega} = \bar{q} \hp (\hp s \pt\hp \iota_{\mathscr{X}} \vec{\omega} - d\vec{\nu}_{\mathscr{X}} - 2 \pt \vec{\theta} \times \npt \vec{\nu}_{\mathscr{X}} ) q + d \hp (\hp \bar{q} \pt \vec{\nu}_{\mathscr{X}} q \pt) \mathrlap{.}
\end{equation}
The first term on the right-hand side vanishes by the Galicki--Lawson moment map equation for $\mathscr{X}$, leaving behind a hyperk\"ahler moment map equation for $X$, which is thus also tri-Hamiltonian, with hyperk\"ahler moment map \eqref{Sw_mom_maps}.

Conversely, let us consider on $\mathcal{U}(M)$ a tri-Hamiltonian vector field $X$ commuting with the canonical \mbox{$\mathbb{H}^{\times}\nhp/\mathbb{Z}_2${\hp-\hp}action}. Like any vector field on $\mathcal{U}(M)$, locally this can be decomposed with respect to the Swann fibration structure into horizontal and vertical components, $\smash{ X = \mathscr{X}^H + \nu_a \ell_{\mathrlap{a}}{}^R }$,  with $\mathscr{X}^H$ the horizontal lift of a vector $\mathscr{X} \in TM$ and the coefficients $\nu_a$ locally defined functions on $\mathcal{U}(M)$. The requirement that $X$ commute with the generators $\smash{ \ell^L_a }$ of the \mbox{$\mathbb{H}^{\times}\nhp/\mathbb{Z}_2${\hp-\hp}action} forces the coefficients $\nu_a$ to be independent of the fiber coordinates. Moreover, the condition that $X$ be tri-Hamiltonian implies that it must preserve the hyperk\"ahler potential $\smash{ |q|^2 }$ up to a constant additive shift (for this, note that tri-Hamiltonian vector fields are automatically tri-holomorphic; the result then follows by applying Lemma~\ref{plurih} for each hyperk\"ahler complex structure). Since $\smash{ X(|q|^2) = - \hp 2\pt \nu_{\hp 0} \pt |q|^2 }$, to avoid contradiction we must have $\smash{ \nu_{\hp 0} = 0}$. If we rename the non-vanishing coefficients $\smash{ \nu_i = \nu_{\mathscr{X}i} }$, then $X$ is formally of the form \eqref{X_Sw}. The formula \eqref{i_X_Om^Sw} gives us again just as above the equation \eqref{i_X_Om_nu}, from which we now conclude that
\begin{equation}
\mathcal{L}_X \vec{\Omega} = d\pt [\pt \bar{q} \hp (\hp s \pt\hp \iota_{\mathscr{X}} \vec{\omega} - d\vec{\nu}_{\mathscr{X}} - 2 \pt \vec{\theta} \times \npt \vec{\nu}_{\mathscr{X}} ) \hp q \pt] \mathrlap{.}
\end{equation}
By assumption this must vanish, which implies that $\smash{ \vec{\nu}_{\mathscr{X}} }$ must satisfy the Galicki--Lawson equation and is thus a quaternionic K\"ahler moment map for $\mathscr{X}$. By Theorem~\ref{QK-mm-th}, $\mathscr{X}$ is then necessarily a Killing vector field. 

Finally, given two pairs of vector fields $\mathscr{X}, X$ and $\mathscr{Y}, Y$ each of which satisfy the conditions of the theorem, one can verify explicitly that $\smash{ [X,Y] = [\mathscr{X},\mathscr{Y}]^H \npt + \nu_{[\mathscr{X},\mathscr{Y}]}{}_i \pt \ell_{\mathrlap{i}}{}^{\raisebox{0.3pt}{$\scriptstyle R$}} }$.
This proves the last statement of the theorem.
\end{proof}

\begin{remark}
Similarly to the hyperk\"ahler quotient construction \cite{MR877637}, the Galicki--Lawson generalized moment map can be used to define a quaternionic K\"ahler quotient construction by which a quaternionic K\"ahler space with continuous isometries is reduced to another quaternionic K\"ahler space in which the isometries are divided out \cite{MR872143, MR960830}. These quotient constructions are compatible with Swann bundles in the sense that the associated Swann bundle of the quaternionic K\"ahler quotient of $M$ by an isometric Lie group action is isomorphic to the hyperk\"ahler quotient at zero moment map level of the Swann bundle of $M$ by the corresponding lifted action \cite[Theorem 4.6]{MR1096180}.
\end{remark}

\section{Analogue of the extended Gibbons--Hawking Ansatz for quaternionic K\"ahler metrics} \label{sec:QK_analogue}

\subsection{The extended Gibbons--Hawking Ansatz} 

\subsubsection{}

The Gibbons--Hawking Ansatz \cite{Gibbons:1979zt} allows one to construct four-dimensional hyperk\"ahler metrics with a tri-Hamiltonian \mbox{$S^1$-action} out of the solutions of a set of partial differential monopole equations defined on an open subset of $\mathbb{R}^3$. This construction was extended by Pedersen and Poon in \cite{MR953820} (see also \cite{MR877637}) to \mbox{$4m$-dimensional} hyperk\"ahler metrics with a local tri-Hamiltonian \mbox{$\mathbb{R}^m$-action}, and in what follows we will refer to this generalization as the \textit{extended} Gibbons--Hawking Ansatz. 

Consider a principal $\mathbb{R}^m$-bundle $N$ over an open subset \mbox{$S \subset \mathbb{R}^m \npt\otimes \mathbb{R}^3$} equipped with a principal connection with Lie algebra-valued curvature 2-form $\smash{ (F_{\K})_{\K \in \pt \mathfrak{I}} }$, where $\mathfrak{I}$ denotes an index set of cardinality $m$, and a section \mbox{$(U_{\I\J})_{\I,\J \in \mathfrak{I}}$} of the symmetric second tensor power of the associated adjoint bundle, the Higgs field, such that $\smash{\det (U_{\I\J}) \neq 0}$. Regarding \mbox{$\mathbb{R}^m \npt\otimes \mathbb{R}^3$} as the space of configurations of $m$ distinguishable points in $\mathbb{R}^3$, one may coordinatize it globally with $m$ $\mathbb{R}^3$-vectors $\vec{x}^{\pt\K}$. Let $\smash{ \vec{\partial}_{\K} }$ denote the corresponding coordinate frame and assume that the following system of partial differential equations generalizing the abelian Bogomolny monopole equation in $\mathbb{R}^3$ holds on $S$:
\begin{equation} \label{ext_Bogo_eqs}
\vec{\partial}_{\I}U_{\K\J} = \vec{\partial}_{\J}U_{\K\I}
\qquad\text{and}\qquad
F_{\K}  = \star^{\I} dU_{\K\I} \mathrlap{.}
\end{equation}
The action of the linear Hodge-like star operators is defined by $\smash{ \star^{\I}d\vec{x}^{\pt\J} = \frac{1}{2} \pt d\vec{x}^{\pt\I} \!\wedge d\vec{x}^{\pt\J} }$, and summation over the repeated index {\small $I$} is understood. Consistency with the Bianchi identity for the curvature 2-form requires that the components of the Higgs field satisfy a set of second-order differential constraints generalizing the Laplace equation in $\mathbb{R}^3$.

To a solution $(F_{\K}, U_{\I\J})$ of these equations one then associates on the total space of the $\mathbb{R}^m$-bundle $N$ the triplet of 2-forms
\begin{equation} \label{Om_GH}
\vec{\Omega} = - \pt \frac{1}{2}\pt U_{\I\J}\pt d\vec{x}^{\pt\I} \!\wedge d\vec{x}^{\pt\J} - d\vec{x}^{\pt\K} \!\wedge (d\psi_{\K} + A_{\K}) 
\end{equation}
expressed here in a local trivialization, with $\psi_{\I}$ coordinates on the fibers and $A_{\K}$ a local connection 1-form on the base $S$ with curvature \mbox{$dA_{\K} = F_{\K}$}. The two extended Bogomolny equations guarantee then that $\smash{ \vec{\Omega} }$, thus defined, is closed. What is more (see an argument below), its components also satisfy the algebraic condition of Theorem~\ref{HK-criterion}. Therefore, by this theorem, they define on $N$ a hyperk\"ahler structure, with hyperk\"ahler metric 
\begin{equation} \label{metric_GH}
G = \frac{1}{2} \pt U_{\I\J}\pt d\vec{x}^{\pt\I} \!\cdot d\vec{x}^{\pt\J} + \frac{1}{2} \pt U^{\I\J}(d\psi_{\I} + A_{\I}) (d\psi_{\J} + A_{\J}) 
\end{equation}
where $U^{\I\J}$ denotes the matrix inverse of $U_{\I\J}$. The vertical vector fields $\smash{ \partial_{\psi_{\K}} }$ generate a tri-Hamiltonian $\mathbb{R}^m$-action with hyperk\"ahler moment maps $\vec{x}^{\pt\K}$. It is a signature characteristic of the Gibbons--Hawking setup and its higher-dimensional extensions that the rank of the action is just large enough for the group orbit parameters together with the moment map functions to provide a complete parametrization of the space.

\subsubsection{}

Conversely, any  \mbox{$4m$-dimensional} hyperk\"ahler manifold with a free local tri-Ha\-mil\-to\-ni\-an \mbox{$\mathbb{R}^m$-action} arises locally in this way. This follows immediately from Proposition~2.1 in \cite{MR1704547}.

\subsubsection{} \label{GT-quatern}

These formulas admit an equivalent formulation which brings to the fore their quaternionic rather than their $\mathbb{R}^m$-bundle structure. Consider the local coframe of the cotangent bundle of $N$ given by the $4m$ real components of the quaternionic-valued \mbox{1-forms}
\begin{equation} \label{H^I_def}
H^{\I} = U^{\I\J}(d\psi_{\J}+A_{\J}) + d\vec{x}^{\pt\I} 
\end{equation}
where in accordance with our established practice we regard $\mathbb{R}^3$-vectors as imaginary quaternions. Assuming the same natural definitions as before for the exterior and tensor products of quaternion-valued forms, we can then recast the above formulas for the triplet of 2-forms and metric in the following compact form: 
\begin{align} \label{GH_quat}
\Omega & = \frac{1}{2} \pt U_{\I\J} \bar{H}^{\I} \!\wedge H^{\J} \\
G & = \frac{1}{2} \pt U_{\I\J} \bar{H}^{\I} H^{\J} \mathrlap{.} \nonumber
\end{align}
Here we view again $\Omega$ as an imaginary quaternion-valued 2-form. Note by the way that this formula for $\Omega$ yields immediately, with the notations from \mbox{\S\,\ref{ssec:qu_in_pr}}, the further expression
\begin{equation}
\Omega_{\hp i} = \frac{1}{2} \pt U_{\I\J} \langle u_au_i,u_b \hp \rangle \pt H_{\mathrlap{a}}{}^{\I} \pt\wedge H_{\mathrlap{b}}{}^{\J} \mathrlap{.}
\end{equation}
In this coframe the algebraic condition of Theorem~\ref{HK-criterion} is straightforward to check if one recalls that the map \mbox{$u_{\hp i} \mapsto \langle u_a u_{\hp i},u_b \hp \rangle$} gives a four-dimensional representation of the standard basis of imaginary quaternions.

The second extended Bogomolny equation can be replaced in this framework by the differential condition
\begin{equation}
dH^{\I} = - \pt \frac{1}{2} \pt U^{\I\L} \partial_{x^{\K}_k} U_{\L\J} (H^{\K} \!\wedge \bar{H}^{\J})_k  \mathrlap{.}
\end{equation}
Another, more interesting, reformulation of the extended Bogomolny equations emerges if one considers instead the \textit{dual} frame to the above quaternionic coframe of \mbox{$T^*\npt N$}. This is given explicitly by the vertical vector fields $\smash{ E_{\I 0} = U_{\K\I} \partial_{\psi_{\K}} }$ together with the horizontal lifts $E_{\I i}$ to $N$ of the coordinate frame $\smash{ \partial_{x^{\I}_i}}$ on the base $S$ (the horizontal lift map is $\smash{ X \mapsto X - (\iota_XA_{\K}) \pt \partial_{\psi_{\K}} }$ for any $\smash{ X \in TS }$). Then the two extended Bogomolny equations \eqref{ext_Bogo_eqs} are respectively equivalent to the following two sets of commutator relations
\begin{equation}
[E_{\I 0}, E_{\J k}] = [E_{\J 0}, E_{\I k}]
\qquad\text{and}\qquad
[E_{\I i}, \pt E_{\J j} \hp] \pt = \varepsilon_{ijk} \pt [E_{\I 0}, E_{\J k}] 
\end{equation}
with a structure reminiscent of that of the self-dual Yang-Mills (SDYM) equations. For \mbox{$m=1$} it is indeed well known that the Bogomolny equations emerge via reduction of the SDYM equations from four to three dimensions, and this formulation makes that origin transparent.

\subsection{Extended Gibbons--Hawking spaces with a hyperk\"ahler cone structure} \label{sec:GH-HKC}

\subsubsection{} \label{ssec:collective}

Let us investigate now under what conditions an extended Gibbons--Hawking hyperk\"ahler space possesses a hyperk\"ahler cone structure compatible with its $\mathbb{R}^m$-bundle structure. The first thing to observe in this regard is that on the base of any extended Gibbons--Hawking fibration one has a natural action of the group $\smash{ \mathbb{H}^{\times} \npt /\mathbb{Z}_2 }$, namely the one generated by the vector fields 
\begin{equation} \label{the_Ls}
\vec{L} = - \, \vec{x}^{\pt\K} \!\times \vec{\partial}_{\K}
\qquad \text{and} \qquad
L_0 = \vec{x}^{\pt\K} \!\cdot \vec{\partial}_{\K} 
\end{equation}
satisfying the algebra
\begin{equation}
[\hp L_i, L_j \hp ] = \varepsilon_{ijk} \hp L_k
\qquad\qquad
[\hp L_i, L_0 ] = 0 \mathrlap{.}
\end{equation}

\begin{remark}
The $\mathfrak{so}(3)$ structure constants here differ from the ones we have encountered so far, such as the ones in equations \eqref{so(3)-alg} or \eqref{ells}, by a factor of 2. This discrepancy can be removed by a simple rescaling of the definitions of the generators and is a matter of convention and convenience rather than of any profound significance. In practice, one needs to be alert to this normalization issue when using here the results of section~\ref{ssec:HKCs} and scale accordingly.
\end{remark}

\noindent In the three-dimensional configuration-of-points picture for $\smash{ \mathbb{R}^m \npt\otimes \mathbb{R}^3 }$, $\smash{ \vec{L} }$ and $L_0$ are the generators of \textit{rigid rotations} respectively \textit{simultaneous scalings} of configurations. Borrowing a term from the classical mechanics of rigid bodies we will refer to such transformations as \textit{collective transformations}. 

Inserting the generating vector fields of collective transformations into the second Bogomolny equation \eqref{ext_Bogo_eqs} and then using the first one to manipulate the result, we get that
\begin{equation} \label{LaF}
\iota_{L_a} F_{\I} = d(U_{\I\J} \pt x_{\mathrlap{a}}{}^{\J}) + c_{\I a}
\end{equation}
where the index $a$ runs as usual from 0 to 3, $\smash{ x_{\mathrlap{a}}{}^{\raisebox{0.3pt}{$\scriptstyle \J$}} }$ are the components of the imaginary quaternions associated to the $\smash{ \mathbb{R}^3 }$-vectors $\smash{ \vec{x}^{\pt\J} }$ (with, therefore, $\smash{ x^{\J}_0 = 0}$), and $c_{\I a}$ denote the components of the 1-forms
\begin{equation}
c_{\I 0} = \vec{L}U_{\I\J} \sprod d\vec{x}^{\pt\J}
\quad\text{and}\quad
\vec{c}_{\I} = - \, \vec{L}U_{\I\J} \vprod d\vec{x}^{\pt\J} \npt - (L_0U_{\I\J} + U_{\I\J}) \pt d\vec{x}^{\pt\J} \mathrlap{.}
\end{equation}
Note, incidentally\,---\,and the relevance of this observation will soon become apparent\,---\,that by resorting further to the first Bogomolny equation we can show that
\begin{equation} \label{xc}
\vec{x}^{\pt\I} \!\sprod \vec{c}_{\I} =  \frac{1}{2} \pt U_{\I\J} \hp d(\hp \vec{x}^{\pt\I} \npt\sprod \vec{x}^{\pt\J}) - d(\hp U_{\I\J} \pt \vec{x}^{\pt\I} \!\sprod \vec{x}^{\pt\J})
\end{equation}
and then that
\begin{equation} \label{starc0}
\star^{\I} c_{\I 0} = d(\vec{x}^{\pt\I} \!\sprod \vec{c}_{\I}) \mathrlap{.}
\end{equation}

Formula \eqref{LaF} suggests that, when approaching the question of whether the above \mbox{$\mathbb{H}^{\times} \npt /\mathbb{Z}_2$}-action on the base can be lifted to a full hyperk\"ahler cone structure on $N$, we should consider not just the horizontal lifts of the generators $L_a$ but rather the lifts 
\begin{equation} \label{La-lift}
X_a = L_a^{\phantom{.}} - (\iota_{L_a}A_{\I} + U_{\I\J}\pt x_{\mathrlap{a}}{}^{\J} \hp) \pt \partial_{\psi_{\I}}
\end{equation}
containing additional vertical components. Using nothing but the definitions of the quantities involved and no field equations\,---\,yet, one can then show by a rather long but otherwise straightforward calculation that these vector fields satisfy the identities
\begin{gather}
\iota_{X_i} \Omega_{\pt j} = \varepsilon_{ijk} \, \iota_{X_0} \Omega_{k} - \frac{1}{2} \pt \delta_{ij} \pt U_{\I\J} \pt d(\hp \vec{x}^{\pt\I} \!\sprod \vec{x}^{\pt\J}) \\[-2pt]
\mathcal{L}_{X_0} \Omega_k = \Omega_k - \frac{1}{2}(L_0U_{\I\J} + U_{\I\J}) (\hp d\vec{x}^{\pt\I} \!\wedge d\vec{x}^{\pt\J})_k + (\iota_{L_0}F_{\I})  \npt\wedge\npt dx_{\mathrlap{k}}{}^{\raisebox{0.3pt}{$\scriptstyle \I$}} \mathrlap{.} \nonumber
\end{gather}
Note in particular that any other choice of lifting prescription would spoil the very convenient and simple tensorial structure of the right-hand side of the first equality. By deploying now in the first identity the relation \eqref{xc} and in the second one the \mbox{$a=0$} component of the relation \eqref{LaF}, both consequences of the field equations, we can eventually cast them in the form
\begin{equation} \label{HKC-obstr}
\iota_{X_i} \Omega_{\pt j} = \varepsilon_{ijk} \, \iota_{X_0} \Omega_{k} -  \delta_{ij} [\hp d(U_{\I\J} \pt \vec{x}^{\pt\I} \!\sprod \vec{x}^{\pt\J}) + \vec{x}^{\pt\I} \!\sprod \vec{c}_{\I} \hp ]
\quad\text{and}\quad
\mathcal{L}_{X_0} \Omega_k = \Omega_k + \star^{\I} c_{\I k} 
\end{equation} 
with the same triplets of 1-forms $\vec{c}_{\I}$ that we have encountered above. These two formulas should be understood as expressing the obstructions to the criterion \mbox{$(e)$} of Proposition~\ref{HKC-criterions} (with the normalization adjustments stipulated in the Remark above) being satisfied by the lifts $X_a$ of the collective transformations generators $L_a$. We stress that all the equations that we have written down so far hold in full generality for any extended Gibbons--Hawking space. 

We are now ready to prove the following criterion: 

\begin{proposition} \label{GH-HKC}
The collective $\smash{ \mathbb{H}^{\times} \npt /\mathbb{Z}_2 }$-action on the base of an extended Gibbons--Hawking hyperk\"ahler space $N$ can be lifted to a hyperk\"ahler cone structure on $N$ if and only if the Higgs field satisfies the linear differential constraints
\begin{equation} \label{HKC-conds}
\vec{L} \, U_{\I\J} = \vec{\hp 0}
\qquad \text{and} \qquad 
L_0 \pt U_{\I\J} = - \hp U_{\I\J} 
\end{equation}
that is, if and only if it is invariant at rigid rotations and homogeneous of degree $-1$ in the $\smash{ \vec{x}^{\pt\I} }$ variables. In this case the hyperk\"ahler potential is given by 
\begin{equation}
U = 2 \pt U_{\I\J} \pt \vec{x}^{\pt\I} \!\sprod \vec{x}^{\pt\J} \mathrlap{.}
\end{equation}
\end{proposition}

\begin{proof}

By the criterion criterion \mbox{$(e)$} of Proposition~\ref{HKC-criterions}, the lifts $X_a$ generate a hyperk\"ahler cone structure on $N$ if and only if there exists a function $U$ such that
\begin{equation}
\iota_{X_i} \Omega_{\pt j} = \varepsilon_{ijk} \, \iota_{X_0} \Omega_{k} - \frac{1}{2} \pt \delta_{ij} \pt dU
\qquad \text{and} \qquad
\mathcal{L}_{X_0} \Omega_k = \Omega_k \mathrlap{.}
\end{equation}

When the constraints \eqref{HKC-conds} hold, the 1-forms $\vec{c}_{\I}$ vanish, and then the formulas \eqref{HKC-obstr} make it readily clear that this is indeed the case, with $U$ defined as above.  This proves the converse implication. 

For the direct implication, note that if the vector fields $X_a$ do generate a hyperk\"ahler cone structure then $1)$~the coefficient of $\delta_{ij}$ on the right-hand side of the first equation \eqref{HKC-obstr} must be exact, and therefore closed, and $2)$~the last term on the right-hand side of the second equation \eqref{HKC-obstr} must vanish. The first condition entails that $\smash{ \vec{x}^{\pt\I} \!\sprod \vec{c}_{\I} }$ must be closed, which then by the formula \eqref{starc0} requires that $\smash{ \star^{\I} c_{\I 0} }$ must vanish. In view of the definition of $c_{\I 0}$, this immediately implies the first constraint \eqref{HKC-conds}. Together with this, the second condition implies then right away the remaining constraint \eqref{HKC-conds}, which concludes the proof.
\end{proof}

When the constraints \eqref{HKC-conds} are satisfied, \mbox{$c_{\I a} = 0$}, and from the relation \eqref{LaF} we have
\begin{equation}
\mathcal{L}_{X_a}A_{\I} = \mathcal{L}_{L_a}A_{\I} = d(\iota_{L_a}A_{\I} + U_{\I\J}\pt x^{\J}_a \hp) \mathrlap{.}
\end{equation}
That is, under the Lie action of the generators of the hyperk\"ahler cone structure the \mbox{1-forms} $A_{\I}$ shift by exact terms. But as connection 1-forms these are defined only up to exact terms, so in principle the shifts can be absorbed into their definitions. Equivalently put, one can always choose representative elements of their gauge equivalence classes which satisfy the gauge-fixing conditions
\begin{equation} \label{gauge-fix}
\iota_{L_a}A_{\I} = - \pt U_{\I\J}\pt x^{\J}_a
\end{equation}
and are thus fully invariant under the action of the generators. Note that these conditions do not fix the representative elements completely but only up to exact differentials of rotation and scaling-invariant functions. In this gauge we have simply $\smash{ X_a = L_a }$, that is, the generators of the hyperk\"ahler cone structure coincide with the generators of the collective $\smash{ \mathbb{H}^{\times} \npt /\mathbb{Z}_2 }$-action on the base. From now on we will assume this to always be the case.

\subsubsection{}

Remarkably, the conditions of Proposition~\ref{GH-HKC} entail that the corresponding hyperk\"ahler metrics are determined by a single function, which is not so surprisingly the hyperk\"ahler potential. In particular, the differential conditions on the Higgs field can be exchanged in this case with a dual set of differential conditions on the hyperk\"ahler potential. 

\begin{proposition} \label{U_duality}
Consider the following two assumptions holding on a open subset $S$ of $\mathbb{R}^m \otimes \mathbb{R}^3$:  
\begin{enumerate}
\setlength\itemsep{0.2em}

\item[a)] There exists an \mbox{$m \vprod m$} symmetric matrix-valued function $U_{\I\J}$ on $S$ whose elements satisfy 
\begin{itemize}
\setlength\itemsep{0.2em}

\item[$\circ$] the first-order differential equations \mbox{$\vec{\partial}_{\I}U_{\K\J} = \vec{\partial}_{\npt\J}U_{\K\I}$}
\item[$\circ$] the symmetry constraints \mbox{$\vec{L} \, U_{\I\J} = \vec{\hp 0}$} and \mbox{$L_0 \pt U_{\I\J} = - \hp U_{\I\J}$}. 
\end{itemize}

\item[b)] There exists a function $U$ on $S$ satisfying 
\begin{itemize}
\setlength\itemsep{0.2em}
\item[$\circ$] the second-order differential equations \mbox{$\vec{\partial}_{\I} \vprod \vec{\partial}_{\J} \hp  U = \vec{\hp 0}$} 
\item[$\circ$] the symmetry constraints \mbox{$\vec{L} \pt U = \vec{\hp 0}$} and \mbox{$L_0 \hp U = U$}. 
\end{itemize}
\end{enumerate}
Then the two assumptions are equivalent, with the direct and reverse implication maps given respectively by 
\begin{equation} \label{HK-pot/Higgs}
U = 2 \pt U_{\I\J} \pt \vec{x}^{\pt\I} \!\sprod \vec{x}^{\pt\J}
\qquad \text{and} \qquad
U_{\I\J} = \frac{1}{4} \pt \vec{\partial}_{\I} \sprod \vec{\partial}_{\J} \hp U \mathrlap{.}
\end{equation}
\end{proposition}

\begin{proof}

The proof is an exercise in three-dimensional differential vector calculus. Consider first the implication \mbox{$(a) \Rightarrow (b)$}. Acting with the operatorial identity
\begin{equation}
\vec{x}^{\pt\I} \vprod (\pt \vec{x}^{\pt\J} \vprod \vec{\partial}_{\K}) = \vec{x}^{\pt\J}(\pt \vec{x}^{\pt\I} \!\sprod \vec{\partial}_{\K}) - (\pt \vec{x}^{\pt\I} \!\sprod \vec{x}^{\pt\J}) \pt \vec{\partial}_{\K}
\end{equation}
on $U_{\I\J}$ and summing over the indices {\small $I,J$}, using the first-order differential equations to permute indices and then the symmetry constraints to simplify or eliminate terms, we obtain that
\begin{equation} \label{del_U}
\vec{\partial}_{\I} U = 2\pt U_{\I\J} \pt \vec{x}^{\pt\J}
\end{equation}
with the first equation \eqref{HK-pot/Higgs} taken as the definition of $U$. This relation implies immediately the symmetry constraints on $U$, and by a more circuitous route involving similar manipulations as above, the second equation \eqref{HK-pot/Higgs} and the second-order differential constraints on~$U$.

The reverse implication \mbox{$(b) \Rightarrow (a)$} follows rather more straightforwardly by acting on the function $U$ with the operatorial identities
\begin{align}
& \vec{\partial}_{\K} \vprod (\hp \vec{\partial}_{\I} \vprod \vec{\partial}_{\J}) = \vec{\partial}_{\I} \pt (\hp \vec{\partial}_{\K} \sprod \vec{\partial}_{\J}) - \vec{\partial}_{\J} \pt (\hp \vec{\partial}_{\K} \sprod \vec{\partial}_{\I}) \\[0pt]
& (\pt \vec{x}^{\pt\I} \npt\sprod \vec{x}^{\pt\J}) (\hp \vec{\partial}_{\I} \sprod \vec{\partial}_{\J}) + (\pt \vec{x}^{\pt\I} \npt\vprod\hp \vec{x}^{\pt\J}) \sprod (\hp \vec{\partial}_{\I} \vprod \vec{\partial}_{\J}) = \vec{L}^{\hp 2} + L_0^2 + L_0 \nonumber
\end{align}
and taking the second equation \eqref{HK-pot/Higgs} as the definition of $U_{\I\J}$. These yield respectively the first-order differential equation for $U_{\I\J}$ and the first equation \eqref{HK-pot/Higgs}. The remaining symmetry constraints on $U_{\I\J}$ are fairly easy then to verify directly. 
\end{proof}

\subsubsection{}

Theorem~\ref{HKC=Sw} tells us that, assuming that the $\smash{ \mathbb{H}^{\times} \npt/\mathbb{Z}_2 }$-action is locally free, such a space $N$ as defined by the conditions of Proposition~\ref{GH-HKC} is, modulo a possible sign redefinition of the metric, at least locally homothetic to a Swann bundle. Our main objective in what follows is to determine explicitly, in intrinsic terms, the quaternionic K\"ahler metric on the base of this Swann bundle, including the differential constraints on its building blocks induced by the Bogomolny equations. Notice that the base has quaternionic dimension one less than the quaternionic dimension $m$ of $N$, and that, on another hand, by Theorem~\ref{Sw-symm} the \mbox{$\mathbb{R}^m$-action} on $N$ descends to it with no alteration in rank. Thus, if we take \mbox{$m = n+1 \geq 2$}, then the quaternionic K\"ahler metric we end up with will be a $4n$-dimensional one with a locally free isometric \mbox{$\mathbb{R}^{n+1}$-action}. 

It is natural to assume in these circumstances that the subset $S$ of $\smash{ \mathbb{R}^m \npt\otimes \mathbb{R}^3 }$ is closed under the action generated by $\smash{ \vec{L} }$ and $L_0$. On $N$, the corresponding induced action can be represented by means of quaternions as follows:
\begin{equation} \label{HKC:GH-q}
\begin{tikzcd}
(\hp \psi_{\I},\vec{x}^{\pt\I}) \arrow[r, "q \pt \in \pt  \mathbb{H}^{\times}\! ", mapsto] & (\hp \psi_{\I}, \bar{q}\pt\hp\vec{x}^{\pt\I}q \hp) \mathrlap{.}
\end{tikzcd}
\end{equation}
The invariance of the $\psi_{\I}$ coordinates is predicated on the gauge-fixing conditions \eqref{gauge-fix} being enforced. In accordance with the above choice for the dimension $m$ the index {\small $I$} will be assumed to run from 0 to $n$. 

Simply formulated, our strategy will be to match the symmetry-centered extended Gibbons--Hawking description of $N$ against the quaternionic structure-centered Swann bundle one, and to extract in the process the quaternionic K\"ahler metric. For this, these transformation rules will be key. A nice geometrical interpretation is afforded if we introduce a certain auxiliary space $\smash{ \textrm{Im}\mathbb{HP}^{\pt n} }$ which will play for the quaternionic K\"ahler space a role similar to the one that $\smash{ \mathbb{R}^m \!\otimes \mathbb{R}^3 }$ plays in the extended Gibbons--Hawking hyperk\"ahler case. This and some other useful related concepts will be defined first in the next subsection, before we embark in a more head-on derivation of what in effect will be a quaternionic K\"ahler analogue of the extended Gibbons--Hawking Ansatz.

\subsection{The space {\normalfont $\textup{Im}\mathbb{HP}^{\pt n}$}} \label{ssec:ImHP^n}

\subsubsection{} \label{ssec:in/hom}

As we have remarked before, a point in $\smash{ \mathbb{R}^{n+1} \!\otimes \mathbb{R}^3 }$ can be viewed as a configuration of \mbox{$n+1$} distinguishable points in $\mathbb{R}^3 $ with position vectors $\smash{ \vec{x}^{\pt\I} }$. On this space one can define a natural action of the group $\smash{ \mathbb{H}^{\times} \npt /\mathbb{Z}_2 \cong \pt \mathbb{R}^+ \npt \otimes SO(3) }$ generated by the vector fields $L_0$ and $\smash{ \vec{L} }$ defined above, which acts on configurations either by simultaneously rescaling all their position vectors by a common factor or by rigidly rotating them around the origin of $\mathbb{R}^3$. Quaternions allow us to express these two actions together concisely as
\begin{equation} \label{H*-GH}
\begin{tikzcd}
\vec{x}^{\pt\I} \arrow[r, "q \pt \in \pt  \mathbb{H}^{\times}\! ", mapsto] &  \bar{q} \pt\hp \vec{x}^{\pt\I} q \mathrlap{.}
\end{tikzcd}
\end{equation}

We shall call a configuration \textit{degenerate} if its stabilizer under the $\smash{ \mathbb{H}^{\times}\npt / \mathbb{Z}_2 }$-action is non-trivial.  The set $\smash{ C_{\text{deg}} }$ of degenerate configurations comprises the null configuration (that is, the configuration will all vectors vanishing), which is a fixed point for the action and is thus stabilized by the whole group, and all the configurations for which all non-vanishing position vectors are collinear (we include in this category also the configurations with a single non-vanishing position vector), which are stabilized by an $S^1$ subgroup of the $SO(3)$ subgroup of $\smash{ \mathbb{H}^{\times}\npt / \mathbb{Z}_2 }$. On the set of non-degenerate configurations, the complement of $\smash{ C_{\text{deg}} }$ in $\smash{ \mathbb{R}^{n+1} \!\otimes \mathbb{R}^3 }$, the $\smash{ \mathbb{H}^{\times}\npt / \mathbb{Z}_2 }$-action acts freely. In analogy with the classical projective spaces we define the quotient space
\begin{equation}
\textrm{Im}\mathbb{HP}^{\pt n} = (\hp \mathbb{R}^{n+1} \!\otimes \mathbb{R}^3 \smallsetminus C_{\text{deg}} ) / (\mathbb{H}^{\times} \npt /\mathbb{Z}_2) \mathrlap{.}
\end{equation}
This has $\smash{ \dim \textrm{Im}\mathbb{HP}^{\pt n} = 3(n+1)-4 = 3n-1 }$. 

As in the case of projective spaces, on it we can define open atlases consisting of a finite number of \textit{inhomogeneous coordinate} charts. In practice, to build such an atlas one needs to cut a number of open slices across the $\smash{ \mathbb{H}^{\times} \npt/\mathbb{Z}_2 }$-orbits in the space of non-degenerate configurations in such a way that their projections onto $\smash{ \textrm{Im}\mathbb{HP}^{\pt n} }$ form an open covering. Each slice is in essence a uniform prescription giving us a representative configuration for every orbit intersecting it. 

Concretely, the data required to define an inhomogeneous coordinate atlas on $\smash{ \textrm{Im}\mathbb{HP}^{\pt n} }$ consists of a choice of half-plane in $\mathbb{R}^3$ with boundary line passing through the origin, together with a choice of orientation. Note that the condition that a configuration be non-degenerate can be equivalently stated as the requirement that the configuration have at least one pair of non-vanishing and non-collinear position vectors. To define a slice, choose two position vectors $\vec{x}^{\pt\I_1}$, $\vec{x}^{\pt\I_2}$ with \mbox{$\text{{\small $I_1$}} < \text{{\small $I_2$}}$}, and consider the subset $\smash{ C_{\I_1\I_2} \subset \mathbb{R}^{n+1} \!\otimes \mathbb{R}^3 \smallsetminus C_{\text{deg}} }$ of all the non-degenerate configurations for which these two position vectors are non-vanishing and non-collinear. Observe that this is closed under the $\smash{ \mathbb{H}^{\times} \npt/\mathbb{Z}_2 }$-action. Any configuration from $\smash{ C_{\I_1\I_2} }$ can be transformed through a rigid rotation around the origin of $\mathbb{R}^3$ into a configuration with the first position vector lying on the boundary line and the second position vector lying in the interior of the chosen half-plane. A little consideration reveals that there are actually two ways in which this can be achieved, related by a $180^{\circ}$ rotation, and we select between them by requiring that the cross product of the two vectors\,---\,taken in the given order\,---\,be parallel, as opposed to anti-parallel, to the normal vector to the half-plane (hence the need for orientation data). Furthermore, through a simultaneous rescaling, the length of the first vector can always be adjusted to have a given fixed value. 

\begin{example}[part I]
If we fix in $\mathbb{R}^3$ the half-plane bounded by the $\mathbf{i}$\pt-axis and containing the vector $\mathbf{j}$, with normal vector \mbox{$\mathbf{k}$}, and single out the pair of vectors $\smash{ \vec{x}^{\pt 0} }$, $\smash{ \vec{x}^{\hp 1} }$, then any configuration with these two vectors non-vanishing and non-collinear\,---\,that is, in our notation, from $C_{01}$\,---\,can be transformed as above into a configuration of the form
\begin{alignat}{2} \label{res-config}
\vec{\rho}^{\,\hp 0} \nhp & = && \mathbf{i} \\[-1pt]
\vec{\rho}^{\,1} & = \rho^{1}_1\pt && \mathbf{i} + \rho^{1}_2 \,\mathbf{j} \nonumber \\[-1pt]
\vec{\rho}^{\,\I} & = \rho^{\I}_1 && \mathbf{i} + \rho^{\I}_2 \,\mathbf{j} + \rho^{\I}_3 \pt\mathbf{k} \mathrlap{\text{\ \ for $\text{{\small $I$}} = 2, \dots, n$}} \nonumber
\end{alignat}
with \mbox{$\rho^{1}_2 > 0$}. Clearly, the vector product $\smash{ \vec{\rho}^{\,\hp 0} \nhp\vprod \vec{\rho}^{\,1} = \rho^{1}_2 \pt \mathbf{k} }$ points in the preferred normal direction.
\end{example}

\noindent The restricted configurations constructed in this way have, in kinematic terms, four \textit{frozen degrees of freedom} and, crucially, are in one-to-one correspondence with the $\smash{ \mathbb{H}^{\times} \npt/\mathbb{Z}_2 }$ orbits contained in $\smash{ C_{\I_1\I_2} }$, for which they can be thus considered to be representative configurations. Together they define an open slice transversal to these orbits, and their position vectors, which depend on \mbox{$3(n+1)-4 = 3n-1$} parameters, provide a local inhomogeneous coordinate chart on the projection of $\smash{ C_{\I_1\I_2} }$ to $\smash{ \textrm{Im}\mathbb{HP}^{\pt n} }$. 

The position vectors $\vec{x}^{\pt\I}$ of any configuration from $\smash{ C_{\I_1\I_2} }$ can be uniquely represented as
\begin{equation} \label{x_dec}
\vec{x}^{\pt\I} = \bar{q} \pt\hp\vec{\rho}^{\,\I} q 
\end{equation}
for some $\smash{ q \in \mathbb{H}^{\times} \npt/\mathbb{Z}_2 }$ and  position vectors $\smash{ \vec{\rho}^{\,\I} }$ of a restricted configuration. Following the projective space analogy one may  think of $\smash{ \vec{x}^{\pt\I} }$ and $\smash{ \vec{\rho}^{\,\I} }$ as global homogeneous respectively local inhomogeneous coordinates on $\smash{ \textrm{Im}\mathbb{HP}^{\pt n} }$. 

\begin{example}[part II]
In the example above one can show that, in agreement with uniqueness, this relation can indeed be inverted to yield
\begin{equation}
\vec{\rho}^{\,\I} = \frac{\vec{x}^{\pt 0} \nhp\sprod \vec{x}^{\pt\I}}{|\vec{x}^{\pt 0}|^2} \pt \mathbf{i}
+ \frac{(\vec{x}^{\pt 0} \nhp\vprod \vec{x}^{\hp 1})\sprod(\vec{x}^{\pt 0} \nhp\vprod \vec{x}^{\pt\I})}{|\vec{x}^{\pt 0} \nhp\vprod \vec{x}^{\hp 1}| \pt |\vec{x}^{\pt 0}|^2} \pt \mathbf{j}
+ \frac{\vec{x}^{\pt 0} \npt\sprod [(\vec{x}^{\pt 0} \nhp\vprod \vec{x}^{\hp 1}) \vprod (\vec{x}^{\pt 0} \nhp\vprod \vec{x}^{\pt\I})]}{|\vec{x}^{\pt 0} \nhp\vprod \vec{x}^{\hp 1}| \pt |\vec{x}^{\pt 0}|^3} \pt \mathbf{k}
 \mathrlap{.}
\end{equation}
The components are manifestly invariant under rigid rotations and simultaneous scalings, which reinforces the argument that they are, when they are not constant, natural candidates for the role of coordinates on (the projection of $C_{01}$ to) the quotient space $\smash{ \textrm{Im}\mathbb{HP}^{\pt n} }$. 
\end{example}

Finally, by considering all the possible choices of two position vectors we can construct in this way, using each time the same oriented half-plane, an entire open atlas on $\smash{ \textrm{Im}\mathbb{HP}^{\pt n} }$ with $n(n+1)/2$ inhomogeneous coordinate charts.

\subsubsection{}  \label{ssec:flt_ind_conn}

The extended Bogomolny field equations \eqref{ext_Bogo_eqs} are formulated in terms of the derivatives $\smash{ \vec{\partial}_{\I} }$, which, for our next purpose, it is useful to think of as defining a flat, torsion-free affine connection on the tangent bundle of $\smash{ \mathbb{R}^{n+1} \!\otimes \mathbb{R}^3 }$. There exists a notion of Lie derivative of an affine connection with respect to a vector field, see \textit{e.g.}~the definition $(1.16)$ in \cite{Ballmann}; in particular, a vector field whose Lie derivative preserves the connection is called an \textit{affine vector field}. In our case one can show, for instance by applying the criterion of Proposition~1.18 in \textit{op.\,cit.}, that the generating vector fields $L_a$ of the $\smash{ \mathbb{H}^{\times} \npt/\mathbb{Z}_2 }$-action on $\smash{ \mathbb{R}^{n+1} \!\otimes \mathbb{R}^3 }$ are affine vector fields with respect to this trivial affine connection. As we shall see later on, Bogomolny equations associated to hyperk\"ahler cones can be \textit{reduced} from an open subset of $\smash{ \mathbb{R}^{n+1} \!\otimes \mathbb{R}^3 }$ to one of  $\smash{ \textrm{Im}\mathbb{HP}^{\pt n} }$, which in practice is achieved, in the terminology that we have introduced earlier, by expressing them in inhomogeneous coordinates. The formulation of the reduced equations involves sections of certain vector bundles associated to the natural principal $\smash{ \mathbb{H}^{\times} \npt/\mathbb{Z}_2 }$-bundle over $\smash{ \textrm{Im}\mathbb{HP}^{\pt n} }$ and covariant derivatives of these sections induced by this flat, torsion-free affine connection on $\smash{ \mathbb{R}^{n+1} \!\otimes \mathbb{R}^3 }$. In the remainder of this subsection we want to show through a few examples what properties such sections have, and how to construct these induced covariant derivatives explicitly for a generic choice of inhomogeneous coordinates.

Consider the two coordinate systems that we have defined on the space $\smash{ \mathbb{R}^{n+1} \!\otimes \mathbb{R}^3 \smallsetminus C_{\text{deg}} }$ in the course of \mbox{\S\,\ref{ssec:in/hom}}: the  homogeneous global coordinates $\{\vec{x}^{\pt\I}\}$, and the $\smash{ \mathbb{H}^{\times} \npt / \mathbb{Z}_2 }$-adapted local coordinates \mbox{$\{\hp \vec{\rho}^{\pt\I} \npt, q \}$}. Recalling the differential identity \eqref{dqrbq}, the relation \eqref{x_dec} between them yields
{\allowdisplaybreaks
\begin{align} 
d\vec{x}^{\pt\I} & = \bar{q} \hp (\hp d\vec{\rho}^{\,\I} - 2 \pt \sigma^{R}_0 \vec{\rho}^{\,\I} + 2 \pt \vec{\sigma}^{\hp R}_{\phantom{0}} \!\times \npt \vec{\rho}^{\,\I} ) q \mathrlap{.} \label{dx_dec} \\
\intertext{Denoting \mbox{$ \mathfrak{L}_{a,}{}_{\mathrlap{j}}{}^{\raisebox{0.3pt}{$\I$}}  = \langle \hp u_au_j,\rho^{\I} \rangle $}, this can be equivalently re-expressed as}
dx_{\mathrlap{i}}{}^{\I} & = |q|^2 R_{ij}(q^{-1}) (\pt d\rho_{\mathrlap{j}}{}^{\I} - 2 \pt \sigma_a^R \pt \mathfrak{L}_{a,}{}_{\mathrlap{j}}{}^{\I} \pt ) \mathrlap{.} \\
\intertext{Dually, we have}
\partial_{x^{\I}_i} & = |q|^{-2} R_{ij}(q^{-1}) (\pt \mathfrak{D}_{\I j} - \frac{1}{2} \pt \mathfrak{A}_{\I j, a} \pt \ell_{\mathrlap{a}}{}^R \pt) \label{del_x-HV}
\end{align}
}%
where the local vector fields $\smash{ \mathfrak{D}_{\I i} }$ and quaternionic-valued coefficients $\smash{ \mathfrak{A}_{\I i} }$ depend exclusively on the homogeneous coordinates $\smash{ \vec{\rho}^{\pt\I} }$ and \textit{are completely fixed by their choice}. Concretely, by matching for instance the expressions of the exterior derivative on $\smash{ \mathbb{R}^{n+1} \!\otimes \mathbb{R}^3 \smallsetminus C_{\text{deg}} }$ in the two coordinate systems we obtain the following set of conditions which these must satisfy, and which suffice to determine them:
\begin{alignat}{2} \label{A-L}
& d\rho_{\hp \mathrlap{i}}{}^{\I} \pt \mathfrak{A}_{\I i, a} = 0 &\qquad\qquad
& \mathfrak{L}_{a,}{}_{\mathrlap{i}}{}^{\I} \pt \mathfrak{A}_{\I i,b} = \delta_{ab} \\
& d\rho_{\hp \mathrlap{i}}{}^{\I} \pt \mathfrak{D}_{\I i} = d_{\hp \textrm{Im}\mathbb{HP}^{\pt n}} &
& \mathfrak{L}_{a,}{}_{\mathrlap{i}}{}^{\I} \pt \mathfrak{D}_{\I i} = 0 \mathrlap{.} \nonumber
\end{alignat}
Summation over repeated indices is as usual understood.

\begin{example}[part III]
Let us see explicitly how this works for the  particular choice of inhomogeneous coordinates considered above. The coefficients $\smash{ \mathfrak{A}_{\I i, a} }$, on one hand, are determined by the two conditions on the first line: the first condition implies that $\smash{ \mathfrak{A}_{\I i, a}  = 0 }$ for all pairs of indices $\text{{\small $I$}},i$ such that $\smash{ \rho^{\I}_i \neq \text{constant} }$; this leaves us with \mbox{$4 \times 4 = 16$} unknown components, which are then fixed by the 16 linear equations of the second condition. Thus, for the choice \eqref{res-config} of inhomogeneous coordinates the solution for these is
{\allowdisplaybreaks
\begin{align} \label{A-coeffs}
\mathfrak{A}_{01} = 1 
\qquad
\mathfrak{A}_{02} = - \hp \mathbf{k}
\qquad
\mathfrak{A}_{03} & = \frac{\rho^1_1}{\rho^1_2} \pt \mathbf{i} + \mathbf{j} \\[-1pt]
\mathfrak{A}_{13} & = - \hp \frac{1}{\rho^1_2} \pt \mathbf{i} \mathrlap{.} \nonumber
\end{align}
}%
The second line of constraints, on the other hand, can be solved in terms of the first one to yield 
\begin{equation}
\mathfrak{D}_{\I i} = \partial_{\rho^{\I}_i} - \mathfrak{A}_{\I i, a} \pt \mathfrak{L}_{a,}{}_{\mathrlap{j}}{}^{\J} \hp \partial_{\rho^{\J}_j}
\end{equation}
where we assume the convention that only the terms for which the derivatives $\smash{ \partial_{\rho^{\I}_i} }$ are well defined are considered.
\end{example}

\subsubsection{}

Note now that in the \mbox{$\{\hp \vec{\rho}^{\pt\I}\npt, q \}$} coordinate system the generators of the $\smash{ \mathbb{H}^{\times} \npt / \mathbb{Z}_2 }$-action act strictly on the $q$-coordinates and take the form
\begin{equation} \label{L=lR}
L_a = \frac{1}{2} \pt \ell_{\mathrlap{a}}{}^L \mathrlap{.}
\end{equation}

\subsubsection{}

In order to formulate our results later on we will need to consider two basic types of vector bundles over $\smash{ \textrm{Im}\mathbb{HP}^{\pt n} }$ or subsets thereof{\pt}: on one hand, an $\mathbb{R}^{n+1}$-bundle, and on the other, vector bundles associated to the natural principal $\mathbb{H}^{\times}\npt / \mathbb{Z}_2$-bundle.
\begin{equation}
\begin{tikzcd}[column sep = 2ex]
M \arrow[dr, "\mathbb{R}^{n+1}" ' near start] & & E\mathrlap{{}^{\hp(\mathfrak{rep})}} \arrow[dl, "\mathbb{H}^{\times} \npt / \mathbb{Z}_2" near start] \\
& \textup{Im}\mathbb{HP}\mathrlap{{}^{\pt n}}  &
\end{tikzcd}
\end{equation}
The latter bundles can be classified in accordance with the finite-dimensional representations of their structure group $\smash{ \mathbb{H}^{\times} \npt /\mathbb{Z}_2 } $. Since this is isomorphic to $\smash{ \mathbb{R}^+ \npt \otimes SO(3) }$, specifying its representations means specifying the representations of its two simple factors. The representations of $\smash{ \mathbb{R}^+ \npt }$, the multiplicative group of strictly positive real numbers, are characterized by their scaling weight $w$, an integer number. On the other hand, the only representations of the group $SO(3)$ that we will be concerned with here will be the trivial and the vector representations, which we denote with reference to their dimensions by $\mathbf{1}$ and $\mathbf{3}$ respectively.

\subsubsection{} \label{ssec:nabla_Ii}

To develop a sense of the properties that sections of such bundles have, it is useful to consider two basic examples corresponding to representations with one factor trivial and one not. 

Thus, suppose we have a section $\mathscr{F}$ of the bundle $E^{\hp(w,\mathbf{1})}$ locally defined on the open set of $\smash{ \textrm{Im}\mathbb{HP}^{\pt n} }$ coordinatized by the inhomogeneous coordinates $\smash{ \vec{\rho}^{\pt\I} }$. We can lift this section to a function
\begin{equation}
F = |q|^{2w} \mathscr{F}
\end{equation}
on (an open subset of) $\smash{ \mathbb{R}^{n+1} \!\otimes \mathbb{R}^3 \smallsetminus C_{\text{deg}} }$, and then using the expression \eqref{L=lR} for the $\smash{ \mathbb{H}^{\times} \npt /\mathbb{Z}_2 } $ generators verify that this satisfies
\begin{equation} \label{sym_F}
\vec{L}F = \vec{\hp 0}
\qquad\qquad
L_0 F =  w F \mathrlap{.}
\end{equation}
That is, sections of $E^{\hp(w,\mathbf{1})}$ correspond to functions on $\smash{ \mathbb{R}^{n+1} \!\otimes \mathbb{R}^3 \smallsetminus C_{\text{deg}} }$ invariant at rigid rotations and homogeneous of degree $w$ in the homogeneous coordinates. Conversely, such functions always descend to sections of $E^{\hp(w,\mathbf{1})}$. On another hand, by resorting to the relation \eqref{del_x-HV} we get that
\begin{equation} \label{delF}
\partial_{x^{\I}_i} F = |q|^{2w-2} R_{ij}(q^{-1}) \hp \nabla_{\I j} \mathscr{F} 
\end{equation}
where by definition
\begin{equation}
\nabla_{\I i} \mathscr{F} = \mathfrak{D}_{\I i} \mathscr{F} + w \pt\hp \mathfrak{A}_{\I i, 0} \pt \mathscr{F} \mathrlap{.}
\end{equation}
This is the connection induced on the bundle $\smash{ E^{\hp(w,\mathbf{1})} }$ over $\smash{ \textrm{Im}\mathbb{HP}^{\pt n} }$ by the flat connection on $\smash{ \mathbb{R}^{n+1} \!\otimes \mathbb{R}^3 }$. It maps sections of $\smash{ E^{\hp(w,\mathbf{1})} }$ to sections of $\smash{ M \otimes E^{\hp(w-1,\mathbf{3})}  }$. A simple counting argument shows that its components cannot be all be independent since there are four more of them than the dimension of  $\textrm{Im}\mathbb{HP}^{\pt n}$. One has to have therefore four constraints, which we shall call \textit{section type constraints}. This follows indeed from the two conditions on the right in \eqref{A-L}. However, an easier way to arrive at them is to use equivariance to strip off the $q$-dependence from the symmetry properties \eqref{sym_F}, a process we will refer to as \textit{reduction}. We obtain in this way
\begin{equation}
- \pt \vec{\rho}^{\,\I} \npt\vprod \vec{\nabla}_{\I} \mathscr{F} = \vec{\hp 0}
\qquad\qquad
\vec{\rho}^{\,\I} \npt\sprod \vec{\nabla}_{\I} \mathscr{F} = w \hp \mathscr{F} \mathrlap{.}
\end{equation}

Suppose now that we have instead a section $\mathscr{F}_i$ of the bundle $\smash{ E^{\hp(0,\mathbf{3})} }$, again, locally defined on the open set of $\smash{ \textrm{Im}\mathbb{HP}^{\pt n} }$ coordinatized by the inhomogeneous coordinates $\smash{ \vec{\rho}^{\pt\I} }$. We lift this to a triplet of functions
\begin{equation}
F_i = R_{ij}(q^{-1}) \mathscr{F}_j
\end{equation}
or in quaternionic notation, \mbox{$\vec{F} = q^{-1} \nhp \vec{\mathscr{F}} q$}, on a corresponding subset of $\smash{ \mathbb{R}^{n+1} \!\otimes \mathbb{R}^3 \smallsetminus C_{\text{deg}} }$. Proceeding in a similar manner as above one can check that these satisfy
\begin{equation} \label{sym_Fi}
L_i F_j = \varepsilon_{ijk} \hp F_k
\qquad\qquad
L_0 F_j = 0 \mathrlap{.}
\end{equation}
Conversely, functions on $\smash{ \mathbb{R}^{n+1} \!\otimes \mathbb{R}^3 \smallsetminus C_{\text{deg}} }$ with these transformation properties under rigid rotations and simultaneous rescalings give rise to sections of $\smash{ E^{\hp(0,\mathbf{3})} }$ over $\smash{ \textrm{Im}\mathbb{HP}^{\pt n} }$. Using the derivative formula \eqref{del_x-HV} we then get
\begin{equation}
\partial_{x^{\I}_i} F_j = |q|^{-2} R_{ik}(q^{-1}) R_{jl}(q^{-1}) \hp \nabla_{\I k} \mathscr{F}_l
\end{equation}
with
\begin{equation}
\nabla_{\I i} \mathscr{F}_j = \mathfrak{D}_{\I i} \mathscr{F}_j - \varepsilon_{jkl} \pt \mathfrak{A}_{\I i,k} \pt \mathscr{F}_l  \mathrlap{.}
\end{equation}
This connection, mapping sections of $\smash{ E^{\hp(0,\mathbf{3})} }$ to sections of $\smash{ M \otimes E^{\hp(-1,\mathbf{3} \hp\otimes\hp \mathbf{3})} }$, satisfies as well section type constraints, which can be argued as before to be given by the reduction of the symmetry properties \eqref{sym_Fi} and thus read
\begin{equation}
- \hp (\pt \vec{\rho}^{\,\I} \npt\vprod \vec{\nabla}_{\I})_{\hp i} \mathscr{F}_j = \varepsilon_{ijk} \pt \mathscr{F}_k 
\qquad\qquad
\vec{\rho}^{\,\I} \npt\sprod \vec{\nabla}_{\I} \mathscr{F}_j = 0 \mathrlap{.}
\end{equation}

This discussion extends naturally to sections of more general vector bundles over $\smash{ \textrm{Im}\mathbb{HP}^{\pt n} }$ constructed out of these two basic types in accordance with the rules of the tensor product. On each of these bundles the trivial connection on $\smash{ \mathbb{R}^{n+1} \!\otimes \mathbb{R}^3 }$ induces a connection $\smash{ \vec{\nabla}_{\I} }$, which takes its sections to sections of the bundle obtained by tensoring with $\smash{ M \otimes E^{\hp(-1,\mathbf{3})} }$. Having at face count four more components than the dimension of $\smash{ \textrm{Im}\mathbb{HP}^{\pt n} }$, the $\smash{ \vec{\nabla}_{\I} }$ form an overcomplete set: while the two conditions on the left in \eqref{A-L} ensure that they always satisfy the completeness relation $\smash{ d\vec{\rho}^{\,\I} \npt\sprod \vec{\nabla}_{\I} = d_{\hp \textrm{Im}\mathbb{HP}^{\pt n}} }$, the ones on the right duly imply the requisite number of what we have called section type constraints, whose precise form depends on the bundle in question. Moreover, it is easy to see that the flatness of the inducing connection implies that the commutators $\smash{ [\nabla_{\I i}, \nabla_{\J j}] }$ vanish. Note also that the inhomogeneous coordinates $\rho^{\I}_i$ on $\smash{ \textrm{Im}\mathbb{HP}^{\pt n} }$ themselves can be regarded as local sections of the bundle $\smash{ M^* \npt \otimes E^{\hp(1,\mathbf{3})} }$, their natural lifts to $\smash{ \mathbb{R}^{n+1} \!\otimes \mathbb{R}^3 \smallsetminus C_{\text{deg}} }$ being the homogeneous coordinates $\smash{ x_{\mathrlap{i}}{}^{\raisebox{0.3pt}{$\I$}} = |q|^2 R_{ij}(q^{-1}) \hp \rho_{\mathrlap{j}}{}^{\raisebox{0.3pt}{$\I$}} }$, and we have $\smash{ \nabla_{\I i}\rho_{\mathrlap{j}}{}^{\raisebox{0.3pt}{$\J$}} = \delta_{ij} \pt \delta^{\J}_{\I} }$. An upshot of these considerations is that the covariant calculus on $\smash{ \textup{Im}\mathbb{HP}^{\pt n} }$ emulates formally the regular differential vector calculus on $\smash{ \mathbb{R}^{n+1} \!\otimes \mathbb{R}^3 }$, with the added feature that one can now make use of section type constraints.

\subsection{The reduction} \label{ssec:Reduction}

\subsubsection{}


This technical detour complete, let us now return to pick up and continue the discussion that we left off at the end of subsection~\ref{sec:GH-HKC}. \pagebreak 

The subset S of $\smash{ \mathbb{R}^{n+1} \!\otimes \mathbb{R}^3 }$, which we are assuming to be closed under the collective $\smash{ \mathbb{H}^{\times} \npt / \mathbb{Z}_2 }$-ac\-tion on the latter space, can also be assumed to be in fact a subset of $\smash{ \mathbb{R}^{n+1} \!\otimes \mathbb{R}^3 \smallsetminus C_{\text{deg}} }$, in concordance with the requirement that the corresponding action on $N$ be locally free. Let $\mathscr{S}$ be the projection of $S$ to $\smash{ \textrm{Im}\mathbb{HP}^{\pt n} }$ through the quotient map. The structure of the action \eqref{HKC:GH-q} implies then that the \mbox{$\mathbb{R}^{n+1}$-bundle} $N$ over $S$ descends to an \mbox{$\mathbb{R}^{n+1}$-bundle} $M$ over $\mathscr{S}$, the total space of which is locally isomorphic to the base of $N$, viewed now as a Swann bundle. This is expressed by the commutative diagram
\begin{equation*}
\hspace{-7em}
\begin{tikzcd}
N \rar["\mathbb{R}^{n+1}"] \dar["\mathbb{H}^{\times}\npt/\mathbb{Z}_2" '] & S \mathrlap{\ \subset \mathbb{R}^{n+1} \!\otimes \mathbb{R}^3 \smallsetminus C_{\textup{deg}}}  \dar["\mathbb{H}^{\times}\npt/\mathbb{Z}_2"] \\
M \rar["\mathbb{R}^{n+1}"] & \mathscr{S} \mathrlap{\, \subset \textup{Im}\mathbb{HP}^{\pt n}}
\end{tikzcd}
\smallskip
\end{equation*} 

To each of the two fibration structures on $N$ we can associate an adapted system of coordinates. For the $\smash{ \mathbb{R}^{n+1} }$-fibration  this is of course the extended Gibbons--Hawking one, given by $\psi_{\I}$ and $\vec{x}^{\pt\I}$. If, however, in this coordinate set we replace what we have called in the previous subsection the homogeneous coordinates $\vec{x}^{\pt\I}$ with the inhomogeneous ones $\vec{\rho}^{\pt\I}$ and the quaternionic coordinates $q$, both of which are now only locally defined, we obtain an alternative system of coordinates which is adapted to the $\smash{ \mathbb{H}^{\times}\npt/\mathbb{Z}_2 }$-fibration structure, with $\psi_{\I}$ and $\vec{\rho}^{\pt\I}$ coordinatizing the base $M$ and $q$ the quaternionic fiber. The two coordinate systems and their component counts are summarized in Table~\ref{adapted_coords}.
\begin{center}
\medskip
\begin{tabular}[H]{lcr|lcr} 
\multicolumn{3}{c|}{$\mathbb{R}^{n+1}$-bundle structure} 	& \multicolumn{3}{c}{\mbox{$\mathbb{H}^{\times} \npt/\mathbb{Z}_2$}-bundle structure} \\
\multicolumn{3}{c|}{adapted coordinates} 	& \multicolumn{3}{c}{adapted coordinates} \\[2pt] \hline
fiber	\rule[1.1em]{0pt}{0pt} 	& \quad $\psi_{\I}$		& $n+1$	& \multirow{2}{*}{base}	& \quad $\psi_{\I}$			& $n+1$  \\[3pt] 
base						& \quad $\vec{x}^{\pt\I}$	& $3n+3$ 	& 					& \quad $\vec{\rho}^{\,\I}$ 	& $3n-1$ \\[2pt] 
						&					&		& fiber 				& \quad $q\ $ 				& 4 \\[3pt] 
&& $\overline{4n+4 \rule[1em]{0pt}{0pt}}$	&					&					& $\overline{4n+4 \rule[1em]{0pt}{0pt}}$ 
\end{tabular}
\captionof{table}{Coordinate systems on $N$.\rule[1.1em]{0pt}{0pt} \vspace{-7pt} }  \label{adapted_coords}
\end{center}

Our goal in what follows will be to re-express the hyperk\"ahler 2-forms and metric on $N$ as well as the extended Bogomolny field equations in terms of the $\smash{ \mathbb{H}^{\times}\npt/\mathbb{Z}_2 }$-fibration-adapted coordinates. Let us begin with the last ones.

\subsubsection{The reduction of the first set of extended Bogomolny equations} \hfill \medskip

The symmetry constraints \eqref{HKC-conds} of Proposition~\ref{GH-HKC} mean that the Higgs field $U_{\I\J}$ determines, in the language coined above, a section $\mathscr{U}_{\I\J}$ of the bundle $\smash{ \textup{Sym}(M \otimes M) \otimes E^{\hp(-1,\mathbf{1})} }$ over $\mathscr{S}$ such that in $\smash{ \mathbb{H}^{\times}\npt/\mathbb{Z}_2 }$-fibration-adapted coordinates we have
\begin{equation} \label{U_IJ_dec}
U_{\I\J} = \frac{\mathscr{U}_{\I\J}}{|q|^2} \mathrlap{.}
\end{equation}
By the considerations in \mbox{\S\,\ref{ssec:nabla_Ii}} it is straightforward to see then that the reduction of the first set of extended Bogomolny equations \eqref{ext_Bogo_eqs} yields the conditions
\begin{equation} \label{red_Bogo_1}
\vec{\nabla}_{\I} \mathscr{U}_{\K\J} = \vec{\nabla}_{\J} \mathscr{U}_{\K\I} \mathrlap{.}
\end{equation}
One should not forget also that, as a section of the bundle $\smash{ E^{\hp(-1,\mathbf{1})} }$, the reduced Higgs field automatically satisfies the section type constraints
\begin{equation} \label{U_IJ-hor-c}
 \vec{\rho}^{\,\I} \npt\vprod \vec{\nabla}_{\I} \mathscr{U}_{\K\J} = \vec{0}
 \qquad\qquad
 \vec{\rho}^{\,\I} \npt\sprod \vec{\nabla}_{\I} \mathscr{U}_{\K\J} = - \pt \mathscr{U}_{\K\J} \mathrlap{.}
\end{equation}

\subsubsection{The reduction of the second set of extended Bogomolny equations} \hfill \medskip

In $\smash{ \mathbb{H}^{\times}\npt/\mathbb{Z}_2 }$-fibration-adapted coordinates the generators of the $\smash{ \mathbb{H}^{\times}\npt/\mathbb{Z}_2 }$-action on $N$ act along the fibers and take the form \eqref{L=lR}. Exploiting this fact, one can argue that in these coordinates connection 1-forms satisfying the gauge-fixing conditions \eqref{gauge-fix} must be of the form
\begin{equation} \label{A_K_dec}
A_{\K} = \mathscr{A}_{\K} + 2 \pt (\mathscr{U}_{\K\J} \vec{\rho}^{\, \J} ) \sprod \vec{\sigma}^{\hp R}
\end{equation}
with the components $\mathscr{A}_{\K}$ supported in $\mathscr{S}$, although still allowed in principle\,---\,for now\,---\,to depend on the $q$-coordinates. 

On another hand, from the decomposition relations \eqref{dx_dec} and \eqref{U_IJ_dec}, by way of the formula \eqref{delF}, we get that
\begin{equation}
\star^{\I} dU_{\K\I} = \frac{1}{2} \pt \vec{\nabla}_{\I}\mathscr{U}_{\K\J} \cdot  (\hp d\vec{\rho}^{\,\I} \npt - 2 \pt \sigma^{R}_0 \vec{\rho}^{\,\I} \npt + 2 \pt \vec{\sigma}^{\hp R}_{\phantom{0}} \!\times \npt \vec{\rho}^{\,\I} ) \nhp\wedge\nhp (\hp d\vec{\rho}^{\,\J} \npt - 2 \pt \sigma^{R}_0 \vec{\rho}^{\,\J} \npt + 2 \pt \vec{\sigma}^{\hp R}_{\phantom{0}} \!\times \npt \vec{\rho}^{\,\J} ) \mathrlap{.}
\end{equation}
If we distribute the wedge product, then use the reduced Bogomolny equation \eqref{red_Bogo_1} in tandem with the section type constraints \eqref{U_IJ-hor-c} to eliminate or simplify terms, and eventually the Cartan--Maurer equation for $\smash{ \vec{\sigma}^R }$ to consolidate the result, this becomes
\begin{equation}
\star^{\I} dU_{\K\I} = \frac{1}{2} \pt \vec{\nabla}_{\I} \mathscr{U}_{\K\J} \sprod (\hp d\vec{\rho}^{\,\I} \!\wedge d\vec{\rho}^{\,\J}) + d(\hp 2 \pt \mathscr{U}_{\K\J} \vec{\rho}^{\, \J} \!\sprod \vec{\sigma}^{\hp R}) \mathrlap{.}
\end{equation}

Recalling that $\smash{ F_{\K} = dA_{\K} }$, the second set of extended Bogomolny equations \eqref{ext_Bogo_eqs} implies then the conditions
\begin{equation} \label{red_Bogo_2}
d\mathscr{A}_{\K}  = \frac{1}{2}\pt \vec{\nabla}_{\I} \mathscr{U}_{\K\J} \sprod (d\vec{\rho}^{\,\I} \!\wedge d\vec{\rho}^{\,\J}) \mathrlap{.}
\end{equation}
This shows in particular that the 1-forms $\mathscr{A}_{\K}$ can be chosen in fact so as \textit{not} to depend on the \mbox{$q$-coordi}\-nates at all, and can be thus thought of as \textit{bona fide} connection 1-forms on the \mbox{$\mathbb{R}^{n+1}$-bundle} \mbox{$M \rightarrow \mathscr{S}$} induced by the connection 1-forms $A_{\K}$ on the \mbox{$\mathbb{R}^{n+1}$-bundle} \mbox{$N \rightarrow S$}.

\subsubsection{The reduction of the hyperk\"ahler 2-forms and metric} \hfill \medskip

At this point, it is opportune to switch temporarily to a quaternionic formalism. The operative observation here is that the key decomposition relations \eqref{x_dec} (or, rather, its differential version \eqref{dx_dec}), \eqref{U_IJ_dec} and \eqref{A_K_dec} can be pieced together into the following quaternionic decomposition formula for the quaternionic-valued coframe elements \eqref{H^I_def} of $T^*\npt N$:
\begin{equation}
H^{\I} = \bar{q}\pt( h^{\I} \npt - 2 \rho^{\I}\sigma^R ) q
\end{equation}
where $\rho^{\I}$ are the imaginary quaternionic avatars of the vectors $\smash{ \vec{\rho}^{\,\I} }$ and, by definition,
\begin{equation} \label{h^I_def}
h^{\I} = \mathscr{U}^{\I\J}(d\psi_{\J}+\mathscr{A}_{\J}) + d\vec{\rho}^{\,\I} \mathrlap{.}
\end{equation}
Note that these form in turn a complete set of quaternionic-valued coframe elements for $T^*\npt M$. 

Substituting then this expression for the elements $H^{\I}$ together with the formula \eqref{U_IJ_dec} for the Higgs field $U_{\I\J}$ into the quaternionic formulas \eqref{GH_quat} for the hyperk\"ahler 2-forms and metric on $N$ gives us after a series of rather straightforward algebraic manipulations the Swann-like expressions 
\begin{align} \label{Sw-scaled}
\Omega & = \mathscr{U}  \bar{q} \hp [\hp s\, \omega + (\overline{\sigma^R \! - \theta\hp}) \nhp\wedge\nhp (\sigma^R \! - \theta\hp) \hp] \hp q \\[2pt]
G & = \mathscr{U}  |q|^2 \hp (\hp s\pt g + |\sigma^R \! - \theta\hp|^2 \hp ) \nonumber
\end{align}
where we have denoted $\smash{ \mathscr{U} = 2 \hp \mathscr{U}_{\I\J} \hp \bar{\rho}^{\hp\I} \rho^{\J} = 2 \hp \mathscr{U}_{\I\J} \pt \vec{\rho}^{\,\I} \!\sprod \vec{\rho}^{\,\J} }$, which we assume to be non-vanishing,
\begin{equation} \label{theta_quat}
\theta = \frac{\mathscr{U}_{\I\J}\hp \bar{\rho}^{\I} h^{\J}}{2 \hp \mathscr{U}_{\M\N} \hp \bar{\rho}^{\hp\M} \rho^{\N}}
\end{equation}
and
\begin{align} \label{QK-quat-1}
s\, \omega & = \frac{(\mathscr{U}_{\K\L}\bar{\rho}^{\hp\K}\rho^{\L}) (\mathscr{U}_{\I\J}\bar{h}^{\hp\I} \!\wedge h^{\J}) - (\mathscr{U}_{\I\L}\bar{h}^{\I}\rho^{\L}) \nhp\wedge\nhp (\mathscr{U}_{\K\J}\bar{\rho}^{\hp \K}h^{\J})}{(\hp 2\hp \mathscr{U}_{\M\N}\bar{\rho}^{\hp\M}\rho^{\N})^2} \\
s\pt g & = \frac{(\mathscr{U}_{\K\L}\bar{\rho}^{\hp\K}\rho^{\L}) (\mathscr{U}_{\I\J}\bar{h}^{\hp\I} h^{\J}) - (\mathscr{U}_{\I\L}\bar{h}^{\I}\rho^{\L}) (\mathscr{U}_{\K\J}\bar{\rho}^{\hp \K}h^{\J})}{(\hp 2\hp \mathscr{U}_{\M\N}\bar{\rho}^{\hp\M}\rho^{\N})^2} \mathrlap{.} \nonumber
\end{align}
The two formulas \eqref{Sw-scaled} differ from the standard Swann bundle ones \eqref{Om_iii} and \eqref{G_iii} in two respects: first, by the presence of the overall factor $\mathscr{U}$,  and second, and perhaps less obviously, since obscured by the notation, by the fact that $\theta$ is \textit{not} an imaginary quaternion-valued 1-form but rather a \textit{full} quaternion-valued one; in other words, by the fact that its real part is non-vanishing. However, notice that this can be cast in the form
\begin{equation} \label{theta_0}
\theta_0 = \frac{1}{\mathscr{U}} \mathscr{U}_{\I\J} \vec{\rho}^{\,\J} \!\sprod d\vec{\rho}^{\,\I}   = \frac{d\mathscr{U}}{2\mathscr{U}} \mathscr{.}
\end{equation}
The first expression follows directly from the definition. To arrive at the second one, let us observe first that $\mathscr{ U }$ can be viewed as originating from the reduction of the hyperk\"ahler potential $U$, more precisely, that we have $\smash{ U = |q|^2 \mathscr{U} }$. The symmetry constraints of part $(b)$ of Proposition~\ref{U_duality} which the latter satisfies imply then that $\mathscr{ U }$ is a section of the line bundle $E^{\hp(1,\mathbf{1})}$. This means in particular that the covariant derivatives $\smash{ \vec{\nabla}_{\I} \mathscr{U} }$ are well-defined. In light of these facts the relation \eqref{del_U} can be easily reduced to yield 
\begin{equation} \label{del_curl-U}
\vec{\nabla}_{\I} \mathscr{U}  = 2\hp \mathscr{U}_{\I\J} \vec{\rho}^{\,\J} \mathrlap{.}
\end{equation} 
From this, the second expression for $\theta_0$ follows then immediately. Let us record also while at this that for the remaining imaginary/vectorial part of $\theta$ we obtain directly, without any further argument, 
\begin{equation} \label{theta_vec}
\vec{\theta} = - \frac{1}{\mathscr{U}} [\pt \mathscr{U}_{\I\J}\hp \vec{\rho}^{\,\J} \npt\vprod d\vec{\rho}^{\,\I} + \vec{\rho}^{\,\I} (d\psi_{\I} + \mathscr{A}_{\I}) \hp] \mathrlap{.}
\end{equation}
The final expression \eqref{theta_0} for $\theta_0$ shows that the departure from the Swann canonical forms is rectifiable: the overall \mbox{$\mathscr{ U }$-factor} and the real component of $\theta$ in the formulas \eqref{Sw-scaled} can both be removed simultaneously through a rescaling
\begin{equation}
q \longmapsto \frac{q}{|\mathscr{U}|\mathrlap{{}^{1/2}}} 
\end{equation}
of the fiber coordinates. When \mbox{$\mathscr{U} \nhp < 0$} this must be supplemented by a sign redefinition \mbox{$\Omega \mapsto - \pt \Omega$}, and consequently \mbox{$G \mapsto - \pt G$}. None of these maneuvers, however, affects the expressions for \mbox{$s\, \omega$}, \mbox{$s\pt g$} or the imaginary part of $\theta$, which can now be legitimately interpreted as quaternionic K\"ahler 2-forms, metric and $SO(3)$ connection 1-forms, respectively.

\subsubsection{A direct check} \hfill \medskip

In fact, since we are now in possession of explicit intrinsic expressions for these, we can verify \textit{directly} that they satisfy the requisite quaternionic K\"ahler properties without resorting to the extrinsic device of the Swann bundle. For this, let us observe first that similarly to the extended Gibbons--Hawking hyperk\"ahler case (see \mbox{\S\,\ref{GT-quatern}}), the reduced Bogomolny equation \eqref{red_Bogo_2} can be replaced in the quaternionic formalism with the equivalent condition
\begin{equation}
dh^{\I} = - \pt \frac{1}{2} \hp \mathscr{U}^{\I\L} \nabla_{\K k} \mathscr{U}_{\L\J} (h^{\K} \!\wedge \bar{h}^{\J})_k  \mathrlap{.}
\end{equation}
Using this along with the other reduced Bogomolny equation \eqref{red_Bogo_1} and the section type constraints \eqref{U_IJ-hor-c}, we can then show through a somewhat involved and rather technical exercise in differential form-valued vector calculus that the following two differentiation formulas
\begin{align}
d(\mathscr{U}_{\I\J} \bar{\rho}^{\I}\rho^{\J}) & = \frac{1}{2}\mathscr{U}_{\I\J} (\bar{h}^{\I}\rho^{\J} \npt+ \bar{\rho}^{\J}h^{\I}) \\
d(\mathscr{U}_{\I\J} \bar{\rho}^{\I}h^{\J}) & = \frac{1}{2} \mathscr{U}_{\I\J} \pt \bar{h}^{\I} \!\wedge h^{\J} \nonumber
\end{align}
hold. From these and the above definition of $\theta$ we eventually infer that
\begin{equation}
d\theta + \theta \wedge \theta = s \, \omega \mathrlap{.}
\end{equation}
Note that the real component of $\theta$ drops out identically from this equality leaving behind only its imaginary quaternionic part, and so this is precisely the Einstein condition \eqref{Einst_cond} of the quaternionic K\"ahler criterion laid out in Theorem~\ref{QK-criterion-2}. 

 \subsubsection{Alternative expressions} \hfill \medskip

 Let us return now to the quaternionic formulas \eqref{QK-quat-1} for the quaternionic K\"ahler 2-forms and metric, and observe that they can be alternatively recast through a purely algebraic effort in the\,---\,also quaternionic\,---\,form
\begin{align} \label{QK-quat-2}
s\, \omega & = \frac{1}{2\mathscr{U}} \pt \mathscr{U}_{\I\J} (\overline{h^{\I} \npt - 2\hp \rho^{\I} \theta \hp \rule[8.2pt]{0pt}{0pt} } ) \nhp\wedge\nhp (h^{\J} \npt - 2\hp \rho^{\J} \theta \hp) \\
s\pt g & = \frac{1}{2\mathscr{U}} \pt \mathscr{U}_{\I\J} (\overline{h^{\I} \npt - 2\hp \rho^{\I} \theta \hp \rule[8.2pt]{0pt}{0pt} }) (h^{\J} \npt - 2\hp \rho^{\J} \theta \hp) \nonumber
\end{align}
with $\theta$ given by the expression \eqref{theta_quat}. These are essentially the quaternionic K\"ahler analogues of the hyperk\"ahler formulas \eqref{GH_quat}. 
 
From them, we can transition back to a vectorial formalism by writing
\begin{equation}
h^{\I} \npt - 2\hp \rho^{\I} \theta = \mathscr{U}^{\I\J}(d\psi_{\J} + \mathscr{A}_{\J} + 2\hp\mathscr{U}_{\J\K}\hp\vec{\rho}^{\,\K} \!\sprod \vec{\theta} \,)  +  d\vec{\rho}^{\,\I} \npt - 2\pt \theta_0\pt \vec{\rho}^{\,\I} \npt  + 2 \pt \vec{\theta} \vprod \vec{\rho}^{\,\I}
\end{equation}
which then, upon substitution into the two formulas above, gives us the equivalent\,---\,and also most explicit\,---\,expressions
{\allowdisplaybreaks
\begin{alignat}{2}
s\, \vec{\omega} = & - \frac{1}{2\mathscr{U}} \mathscr{U}_{\I\J} (\hp d\vec{\rho}^{\,\I} \npt  - 2\pt \theta_0\pt \vec{\rho}^{\,\I} \npt + 2 \pt \vec{\theta} \vprod \vec{\rho}^{\,\I})  \nhp\wedge\nhp  (\hp d\vec{\rho}^{\,\J} \npt - 2\pt \theta_0\pt \vec{\rho}^{\,\J} \npt + 2 \pt \vec{\theta} \vprod \vec{\rho}^{\,\J}) \label{QK_2-form-Ans} \\
                              & - \frac{1}{\mathscr{U}} (\hp d\vec{\rho}^{\,\I} \npt - 2\pt \theta_0\pt \vec{\rho}^{\,\I} \npt + 2 \pt \vec{\theta} \vprod \vec{\rho}^{\,\I})  \nhp\wedge\nhp  (d\psi_{\I} + \mathscr{A}_{\I} + 2\hp\mathscr{U}_{\I\K}\hp\vec{\rho}^{\,\K} \!\sprod \vec{\theta} \,) \nonumber \\
s\pt g = & \, \frac{1}{2\mathscr{U}} \mathscr{U}_{\I\J} (\hp d\vec{\rho}^{\,\I} \npt - 2\pt \theta_0\pt \vec{\rho}^{\,\I} \npt + 2 \pt \vec{\theta} \vprod \vec{\rho}^{\,\I}) \sprod (\hp d\vec{\rho}^{\,\J} \npt - 2\pt \theta_0\pt \vec{\rho}^{\,\J} \npt + 2 \pt \vec{\theta} \vprod \vec{\rho}^{\,\J}) \label{QK_metric-Ans} \\
          +  & \, \frac{1}{2\mathscr{U}} \mathscr{U}^{\I\J} (d\psi_{\I} + \mathscr{A}_{\I} + 2\hp\mathscr{U}_{\I\K}\hp\vec{\rho}^{\,\K} \!\sprod \vec{\theta} \,) (d\psi_{\J} + \mathscr{A}_{\J} + 2\hp\mathscr{U}_{\J\L}\hp\vec{\rho}^{\,\L} \!\sprod \vec{\theta} \,) \mathrlap{.} \nonumber 
\end{alignat}
}%
The 1-forms $\theta_0$ and $\vec{\theta}$ are given in turn by the formulas \eqref{theta_0} and \eqref{theta_vec}. In this formulation, in particular, it is straightforward to see that the vector fields $\smash{ \partial_{\psi_{\I}} }$ generate an isometric \mbox{$\mathbb{R}^{n+1}$-action} with Galicki--Lawson moment map 
\begin{equation}
\partial_{\psi_{\I}}
\longmapsto
\vec{\nu}^{\,\I} = \frac{\vec{\rho}\mathrlap{{}^{\pt\I}}}{\mathscr{U}} \mathrlap{.}
\end{equation}

In the next subsection we compile for easy reference and present in an intrinsic geometric language the quaternionic K\"ahler construction which emerged from these considerations, in a form mirroring our earlier presentation of the extended Gibbons--Hawking hyperk\"ahler construction.

\subsection{The quaternionic K\"ahler Ansatz} \label{QK-Ansatz} \hfill \medskip

Consider a principal $\mathbb{R}^{n+1}$-bundle $M$ over an open subset $\mathscr{S}$ of $\smash{ \textup{Im}\mathbb{HP}^{\pt n} }$ and, alongside it, the restrictions to $\mathscr{S}$ of the associated $\smash{ \mathbb{H}^{\times} \npt /\mathbb{Z}_2 \pt}$-bundles $\smash{ E^{\hp(\mathfrak{rep})} }$ over $\smash{ \textup{Im}\mathbb{HP}^{\pt n} }$ defined in subsection \ref{ssec:ImHP^n}, which we denote with the same symbols. 
\begin{equation}
\begin{tikzcd}[column sep = 3ex]
M \arrow[dr, "\mathbb{R}^{n+1}" ' near start] & & E\mathrlap{{}^{\hp(\mathfrak{rep})}} \arrow[dl, "\mathbb{H}^{\times} \npt / \mathbb{Z}_2" near start] \\
& \mathscr{S}  \mathrlap{\ \subset \textup{Im}\mathbb{HP}\mathrlap{{}^{\pt n}}}  &
\end{tikzcd}
\end{equation}
Let $\vec{\nabla}_{\I}$ denote the connections induced on the latter bundles by the trivial connection $\smash{ \vec{\partial}_{\I} }$ on $\smash{ \mathbb{R}^{n+1} \!\otimes \mathbb{R}^3 }$. 

In this setup we want to formulate two sets of \textit{field equations}. For the first one let us assume that we have either 
\textit{
\begin{itemize}
\item[\textbf{I\pt a.}] a section $\mathscr{U}_{\I\J}$ of the bundle $\smash{ \textup{Sym} \pt (M \otimes M) \otimes E^{\hp (-1,\mathbf{1})} }$ over $\mathscr{S}$ satisfying the first-order differential constraints
\begin{equation}
\vec{\nabla}_{\I} \mathscr{U}_{\K\J} = \vec{\nabla}_{\J} \mathscr{U}_{\K\I} 
\end{equation}
\end{itemize}
\textup{or, alternatively,}
\begin{itemize}[resume]
\item[\textbf{I\pt b.}]  a section $\mathscr{U}$ of the line bundle $\smash{ E^{\hp(+1,\mathbf{1})} }$ over $\mathscr{S}$ satisfying the second-order differential constraints
\begin{equation} \label{red-U-constr}
\vec{\nabla}_{\I} \vprod \vec{\nabla}_{\J} \mathscr{U} = \vec{0} \mathrlap{.}
\end{equation}  
\end{itemize}
}%

\medskip

\noindent For which one of these conditions we opt is immaterial because
\begin{lemma}
The two conditions are equivalent, with the maps between them given, respectively, by
\begin{equation} \label{red-U-U_IJ}
\mathscr{U} = 2 \pt \mathscr{U}_{\I\J} \pt \vec{\rho}^{\,\I} \!\sprod \vec{\rho}^{\,\J}
\qquad \text{and} \qquad
\mathscr{U}_{\I\J} = \frac{1}{4} \pt \vec{\nabla}_{\I} \sprod \vec{\nabla}_{\J} \hp \mathscr{U} \mathrlap{.}
\end{equation}
\end{lemma}

\noindent This follows immediately from Proposition~\ref{U_duality} by reduction. One can in fact give a completely intrinsic proof of this result simply by mimicking the proof of Proposition~\ref{U_duality} in covariant calculus language and noticing that the symmetry constraints there show up here in the guise of section type constraints. 

On the other hand, for the second set of field equations let us assume that we have, in addition, 
\textit{
\begin{enumerate}
\item[\textbf{II.}]  a principal connection on $M$ with curvature 2-form $\smash{ (\mathscr{F}_{\K})_{\K = 0, \dots, n} }$ determined by the Higgs field through the following Bogomolny-like condition:
\begin{equation} \label{red_Bogo_2bis}
\mathscr{F}_{\K}  = \frac{1}{2}\pt \vec{\nabla}_{\I} \mathscr{U}_{\K\J} \sprod (d\vec{\rho}^{\,\I} \!\wedge d\vec{\rho}^{\,\J}) \mathrlap{.}
\end{equation}
\end{enumerate}
}


Based on the previous considerations we can now state the following 

\begin{theorem} \label{QK-GH-analog}
Let $\smash{(\mathscr{U}_{\I\J}, \mathscr{F}_{\K})}$ or, equivalently, $\smash{(\mathscr{U}, \mathscr{F}_{\K})}$ be a set of solutions to the above field equations, and assume that we can adjust the open subset $\mathscr{S}$ of $\smash{ \textup{Im}\mathbb{HP}^{\pt n} }$ on which they are defined in such a way as to have both $\smash{ \det(\mathscr{U}_{\I\J}) \neq 0 }$ and $\smash{ \mathscr{U} \neq 0 }$ everywhere on it. On the total space of the $\mathbb{R}^{n+1}$-bundle $M$ we associate to this solution the triplet of 1-forms
\begin{equation}
\vec{\theta} = - \, \mathscr{W}_{\I\J} \hp \vec{\nu}^{\,\J} \npt\vprod d\vec{\nu}^{\,\I} - \vec{\nu}^{\,\I} (d\psi_{\I} + \mathscr{A}_{\I})
\end{equation}
and triplet of 2-forms
{\allowdisplaybreaks
\begin{align}
s\, \vec{\omega}  = & - \frac{1}{2} \pt \mathscr{W}_{\I\J} (d\vec{\nu}^{\pt\I} + 2\pt \vec{\theta} \vprod \vec{\nu}^{\pt\I}) \nhp\wedge\nhp (d\vec{\nu}^{\pt\J} + 2\pt \vec{\theta} \vprod \vec{\nu}^{\pt\J}) \\[2pt]
& - (d\vec{\nu}^{\pt\I} + 2\pt \vec{\theta} \vprod \vec{\nu}^{\pt\I}) \nhp\wedge\nhp (d\psi_{\I} + \mathscr{A}_{\I} + 2 \hp \mathscr{W}_{\I\K}\vec{\nu}^{\pt\K} \!\sprod \vec{\theta} \,) \nonumber \\
\intertext{expressed here in a local coordinate trivialization, with $\psi_{\I}$ coordinates on the fibers and $\vec{\rho}^{\,\I}$ inhomogeneous coordinates on the base; by definition, $s$ is a non-vanishing real constant, $\mathscr{A}_{\K}$ is a local connection 1-form with curvature \mbox{$d\mathscr{A}_{\K} = \mathscr{F}_{\K}$},
$$
\mathscr{W}_{\I\J} = \mathscr{U}\mathscr{U}_{\I\J}
\qquad \text{and} \qquad
\vec{\nu}^{\pt\I} = \frac{\vec{\rho}\mathrlap{{}^{\pt\I}}}{\mathscr{U}} \mathrlap{.}
$$
These then form a quaternionic K\"ahler structure on $M$, with quaternionic K\"ahler metric}
s\pt g  = & \,\pt \frac{1}{2} \pt \mathscr{W}_{\I\J} (d\vec{\nu}^{\pt\I} + 2\pt \vec{\theta} \vprod \vec{\nu}^{\pt\I}) \sprod (d\vec{\nu}^{\pt\J} + 2\pt \vec{\theta} \vprod \vec{\nu}^{\pt\J}) \\[-2pt]
            + & \,\pt \frac{1}{2} \pt \mathscr{W}^{\I\J} (d\psi_{\I} + \mathscr{A}_{\I} + 2\hp \mathscr{W}_{\I\K}\vec{\nu}^{\pt\K} \!\sprod \vec{\theta} \,) (d\psi_{\J} + \mathscr{A}_{\J} + \mathrlap{ 2 \hp \mathscr{W}_{\J\L}\vec{\nu}^{\pt\L} \!\sprod \vec{\theta} \,) } \nonumber 
\end{align}
}%
possessing an isometric locally free $\mathbb{R}^{n+1}$-action generated by the vector fields $\smash{ \partial_{\psi_{\I}} }$, with  corresponding Galicki--Lawson moment map images $\smash{ \vec{\nu}^{\pt\I} }$.

Conversely, any $4n$-dimensional quaternionic K\"ahler manifold with an isometric locally free $\mathbb{R}^{n+1}$-action arises locally in this way.
\end{theorem}

The direct claim encapsulates the results of the previous subsection. Our arguments there have yielded in fact three more equivalent sets of formulas for the quaternionic K\"ahler connection \mbox{1-forms}, \mbox{2-forms} and metric, which we list by reference number in Table~\ref{alt_exprs}. 
\begin{center}
\def\arraystretch{1.35} 
\begin{tabularx}{0.93\textwidth}{r*{3}{Y}l}
& $\theta$ & $s\,\omega$ & $s\pt g$ & Formalism \\ \hline
1. & eq.~\eqref{theta_quat} & eq.~\eqref{QK-quat-1} & eq.~\eqref{QK-quat-1} & quaternionic \\ 
2. & eq.~\eqref{theta_quat} & eq.~\eqref{QK-quat-2} & eq.~\eqref{QK-quat-2} & quaternionic \\[-0.5pt]
3. & \mbox{ eqs.~\eqref{theta_0} \& \eqref{theta_vec} } & eq.~\eqref{QK_2-form-Ans} & eq.~\eqref{QK_metric-Ans} & vectorial \\[1pt] \hline 
\end{tabularx}
\captionof{table}{\rule[14pt]{0pt}{0pt}Alternative expressions for the quaternionic K\"ahler connection 1-forms, 2-forms and metric.} 
\vskip-8pt \label{alt_exprs}
\end{center} 

For the converse statement, if we assume that we have a $4n$-dimensional quaternionic K\"ahler space $M$ with an isometric locally free $\smash{ \mathbb{R}^{n+1} }$-action, then by Theorem~\ref{Sw-symm} this action lifts to an \mbox{$\smash{ \mathbb{R}^{n+1} }$-action} on the \mbox{$4(n+1)$}-dimensional Swann bundle $\mathcal{U}(M)$. Bielawski's Proposition~2.1 from \cite{MR1704547} guarantees then that the hyperk\"ahler structure on $\mathcal{U}(M)$ is given locally by an extended Gibbons--Hawking construction. In particular, corresponding to these actions, the quaternionic K\"ahler moment maps are lifted to hyperk\"ahler moment maps in accordance with the rule \eqref{Sw_mom_maps}, and so we conclude that the canonical \mbox{$\smash{ \mathbb{H}^{\times}\npt/\mathbb{Z}_2 }$-action} on $\mathcal{U}(M)$ acts on the base of the extended Gibbons--Hawking fibration by what we have called collective transformations. We find ourselves therefore within the bounds of Proposition~\ref{GH-HKC}, which means that we can repeat \textit{ad litteram} the reduction procedure of subsection \mbox{\S\,\ref{ssec:Reduction}} to obtain an explicit local description of the quaternionic K\"ahler structure on $M$ in the precise same category as the one above, proving thus the generality of this construction. 

The theorem encompasses the range of both definite and indefinite signature metrics. The construction it describes is in a very clear sense a quaternionic K\"ahler analogue of the extended Gibbons--Hawking hyperk\"ahler Ansatz. The role of hyperk\"ahler moment map is played here by Galicki and Lawson's concept of quaternionic K\"ahler moment map; however, unlike in the  Gibbons--Hawking case, the moment map images do not define coordinates (for which purpose they would be too many) but, rather, they are related to these through a scaling factor. Another, more surprising departure from the extended Gibbons--Hawking case is the existence of a dual description implying that the quaternionic K\"ahler structure is, quite remarkably, completely determined locally in all dimensions by a \textit{single} real-valued potential $\mathscr{U}$ satisfying a set of linear differential constraints. The linearity, in particular, means that in a very definite sense one can superpose two such quaternionic K\"ahler structures to obtain a third.

\subsection{The case of four dimensions} \hfill \medskip

We end this section with a closer look at the four-dimensional (\textit{i.e.}~\mbox{$n=1$}) specialization of this construction. In this case, $M$ is the four-dimensional avatar of a quaternionic K\"ahler space\,---\,namely, a self-dual Einstein space with non-vanishing scalar curvature\,---\,possessing two linearly independent commuting Killing vector fields. The local geometry of this class of spaces was described explicitly by Calderbank and Pedersen in \cite{MR1950174} in terms of an eigenfunction of the Laplacian on the Poincar\'e half-plane, following an intrinsic, four-dimensional analysis based on previous work by Joyce \cite{MR1324633}, Ward \cite{MR1045295} and Tod \cite{MR1423177}, among others. We will now show that our results reproduce theirs in this dimension, while also offering a few additional bonuses. In this sense, therefore, the construction of Theorem~\ref{QK-GH-analog} may be viewed as a higher-dimensional generalization of the  Calderbank--Pedersen construction\,---\,similarly, although much less trivially, to the way in which, in hyperk\"ahler context, the extended Gibbons--Hawking construction of \cite{MR877637, MR953820} can be viewed as a higher-dimensional generalization of  the original four-dimensional Gibbons--Hawking one. 

First, let us observe that for \mbox{$n=1$} the space $\smash{ \text{Im}\mathbb{HP}^1 }$ is isomorphic to the \textit{upper-half plane} $\mathscr{H}^2$ and can be covered by a single inhomogeneous coordinate chart. We choose the inhomogeneous coordinates as in the Example in \mbox{\S\,\ref{ssec:in/hom}} to be given by
\begin{alignat}{2} 
\vec{\rho}^{\,\hp 0} \nhp & = && \mathbf{i} \\[-1pt]
\vec{\rho}^{\,1} & = \rho_1\pt && \mathbf{i} + \rho_2 \,\mathbf{j} \nonumber
\end{alignat}
with \mbox{$\rho_2 > 0$}, where, for simplicity, we drop the now inessential upper index 1 from the notation of the components of the second vector. As argued subsequently in \mbox{\S\,\ref{ssec:flt_ind_conn}}, a choice of inhomogeneous coordinates fixes completely the form of the covariant connections induced on bundles $\smash{ E^{\hp(\mathfrak{rep})} \rightarrow \text{Im}\mathbb{HP}^1 }$ by the trivial connection on, in this case, $\smash{ \mathbb{R}^2 \!\otimes \mathbb{R}^3 }$. To determine these, we need both the connection coefficients $\smash{ (\mathfrak{A}_{\I i})_{\I = 0,1} }$, which are already listed in \eqref{A-coeffs},  and the operators $\smash{ (\mathfrak{D}_{\I i})_{\I = 0,1} }$, which in this case can be worked out to be
\begin{align}
\vec{\mathfrak{D}}_0 & = - \pt \mathbf{i} \pt (\rho_1\partial_{\rho_1} + \rho_2 \partial_{\rho_2}) + \mathbf{j} \pt (\rho_2\partial_{\rho_1} - \rho_1 \partial_{\rho_2}) \\
\vec{\mathfrak{D}}_1 & = \mathbf{i} \pt \partial_{\rho_1} + \mathbf{j} \pt \partial_{\rho_2} \mathrlap{.} \nonumber
\end{align}

Note that for this choice of inhomogeneous coordinates the only non-vanishing contribution on the right-hand side of the second reduced Bogomolny equation \eqref{red_Bogo_2bis} comes possibly from the term $\smash{ \nabla_{1 3} \mathscr{U}_{\K1} \pt d\rho_1 \npt\wedge d\rho_2 }$. However, this covariant derivative vanishes, too, and the equation becomes simply 
\begin{equation}
\mathscr{F}_{\K} = 0 \mathrlap{.}
\end{equation}
This means that in four dimensions we can consistently set in the formulas, which are locally defined, $\smash{ \mathscr{A}_{\K} = 0  }$. 

In fact, a simple inspection shows that for this choice of inhomogeneous coordinates \textit{any} covariant derivative $\smash{ \nabla_{\I i} }$ with \mbox{$i = 3$} of no matter what section vanishes. The second-order differential constraints \eqref{red-U-constr} reduce therefore to a single one, which reads
\begin{equation} \label{U-diff-constr-4D}
\rho_2 \hp (\mathscr{U}_{\rho_1\rho_1} + \mathscr{U}_{\rho_2\rho_2}) = \mathscr{U}_{\rho_2}
\end{equation}
with the indices indicating derivatives. To obtain this result it is important to recall that $\mathscr{U}$ is a section of the bundle $\smash{ E^{\hp(1,\mathbf{1})} }$, and so $\smash{ \nabla_{\I i}\mathscr{U} }$ is a section of the bundle $\smash{ M^* \npt\otimes E^{\hp(0,\mathbf{3})} }$. The calculation of covariant derivatives of both such sections has been described in detail in \mbox{\S\,\ref{ssec:nabla_Ii}}.  
The constraint can be equivalently recast in the form
\begin{equation}
\Delta_{\mathscr{H}^2} \bigg( \frac{\mathscr{U}}{\sqrt{\rho_2}} \bigg) = \frac{3}{4} \pt \frac{\mathscr{U}}{\sqrt{\rho_2}}
\end{equation}
where $\smash{ \Delta_{\mathscr{H}^2} = \rho_2^2 \pt (\hp \partial^2_{\rho_1} + \partial^2_{\rho_2}) }$ is the Laplacian corresponding to the Poincar\'e metric 
\begin{equation}
g_{\mathscr{H}^2} = \frac{d\rho_1^2 + d\rho_2^2}{4\hp\rho_2^2}
\end{equation}
on the upper half-plane. We retrieve thus a first key result of Calderbank and Pedersen. 

Let us define now the following 1-forms supported on the $\mathbb{R}^2$-fibers
\begin{equation}
\lambda_i^{\phantom{.}} = 2 \pt \rho^{\I}_i \hp d\psi_{\I} = 2 \npt
\begin{pmatrix*}[r]
d\psi_0 + \rho_1d\psi_1 \\ 
\rho_2 \hp d\psi_1 \\ 
0\pt\pt
\end{pmatrix*}
\end{equation}
and similarly fiberwise acting commuting vector fields
\begin{equation}
v_i = \nabla_{\I i} \mathscr{U} \pt \partial_{\psi_{\I}} = 
\begin{pmatrix*}[r]
(\mathscr{U} - \rho_1\mathscr{U}_{\rho_1} - \rho_2\mathscr{U}_{\rho_2}) \hp \partial_{\psi_0} + \mathscr{U}_{\rho_1} \partial_{\psi_1} \\[1pt]
(\rho_2\mathscr{U}_{\rho_1} - \rho_1\mathscr{U}_{\rho_2}) \hp \partial_{\psi_0} + \mathscr{U}_{\rho_2} \partial_{\psi_1}  \\[1pt] 
0\pt\pt\hp
\end{pmatrix*} \!\npt\mathrlap{.}
\end{equation}
The third component vanishes in each case. The remaining components, $\lambda_1$, $\lambda_2$ and $v_1$, $v_2$, form a local coframe and frame, respectively, for the $\mathbb{R}^2$-fibers. By resorting to the relation \eqref{del_curl-U}, we obtain then that
\begin{equation}
\mathscr{U}_{\I\J} \hp \rho^{\I}_i \hp \rho^{\J}_j = \frac{1}{4} \pt \iota_{v_j} \lambda_i^{\phantom{.}} = \frac{1}{2} \npt
\begin{pmatrix*}[r]
\mathscr{U} - \rho_2\mathscr{U}_{\rho_2} & \rho_2 \mathscr{U}_{\rho_1} &  0\pt\pt\hp \\[1pt]
\rho_2 \mathscr{U}_{\rho_1} & \rho_2 \mathscr{U}_{\rho_2} & 0\pt\pt\hp \\[1pt]
0\pt\pt\hp & 0\pt\pt\hp & 0\pt\pt\hp
\end{pmatrix*} \!\npt\mathrlap{.}
\end{equation}
Again, all components along the third direction vanish, leaving a two-dimensional non-vanishing square block. Since in this case the Higgs field matrix $(\mathscr{U}_{\I\J})_{\I,\J = 0,1}$ is also two-di\-men\-sio\-nal, this relation can be used to determine it explicitly through a straightforward exercise in linear algebra. One obtains the same result, although at a greater expense of calculational effort, from the second relation \eqref{red-U-U_IJ}, upon use of the constraint \eqref{U-diff-constr-4D}, which serves to eliminate the second derivatives of $\mathscr{U}$. 

Knowledge of the Higgs field and principal $\smash{ \mathbb{R}^2 }$-connection allows us now to work out explicitly the geometric structure on $M$ given by Theorem~\ref{QK-GH-analog}.  Upon substitution, the formula \eqref{theta_vec} yields, on one hand,
\begin{equation}
\vec{\theta} = - \frac{1}{2\mathscr{U}} [\pt \lambda_1 \mathbf{i} + \lambda_2 \pt \mathbf{j} + (\mathscr{U}_{\rho_1} d\rho_2 - \mathscr{U}_{\rho_2} d\rho_1) \hp \mathbf{k} \pt ] \mathrlap{.}
\end{equation}
To formulate the remaining results it is useful to define the quaternionic-valued 1-form
\begin{equation}
\xi = \frac{d\rho_1 \mathbf{i} + d\rho_2 \pt \mathbf{j}}{2\rho_2} + \frac{\varepsilon(v_1,d\psi) + \varepsilon(v_2,d\psi) \pt \mathbf{k}}{\varepsilon(v_1,v_2)}
\end{equation}
where $\smash{ \varepsilon = d\psi_0 \wedge d\psi_1 }$ is the standard symplectic form on the $\mathbb{R}^2$-fibers and $\smash{ \varepsilon(v_i,d\psi) }$ are the same thing as $\smash{ \iota_{v_i}\varepsilon }$. The formulas \eqref{QK_2-form-Ans} and \eqref{QK_metric-Ans} reduce then, on the other hand, following an extended calculation, to the simple expressions
\begin{align}
s\, \omega & =  \frac{\rho_2 \pt \varepsilon(v_1,v_2)}{\mathscr{U}^2} \, \bar{\xi} \wedge \xi \\
s\pt g & =  \frac{\rho_2 \pt \varepsilon(v_1,v_2)}{\mathscr{U}^2} \, |\xi|^2 \mathrlap{.} \nonumber
\end{align}
One can easily verify that this metric coincides indeed, up to a function redefinition, with the one found by Calderbank and Pedersen in \cite{MR1950174} or, rather, with the version given by formula (2.4) from \cite{MR2225692}. The structure it describes is that of an $\smash{ \mathbb{R}^2 }$-fibration over the Poincar\'e upper half-plane, conformally rescaled. In particular, this form makes it clear that, as observed in \cite{MR1950174}, if we assume that $g$ is positive definite, the sign of the scalar curvature is dictated by the sign of the factor $\smash{ \varepsilon(v_1,v_2) = \mathscr{U}\mathscr{U}_{\rho_2} - \rho_2(\mathscr{U}_{\rho_1}^2 + \mathscr{U}_{\rho_2}^2) }$.

\section{Hyperk\"ahler cones via the Legendre transform construction} \label{sec:HK-LT}

\subsection{The Legendre transform construction} \hfill \medskip

Translating into geometric terms insights derived from a supersymmetric quantum field-the\-o\-re\-tic problem \cite{Lindstrom:1983rt}, Hitchin, Karlhede, Lindstr\"om and Ro\v{c}ek introduced in \cite{MR877637} a method of constructing hyperk\"ahler metrics of extended Gibbons--Hawking type known as the \textit{Legendre transform construction}. In this approach, the metric information is stored in a single real-valued function $L$ defined on some open subset $S$ of \mbox{$\mathbb{R}^m \npt\otimes \mathbb{R}^3$}, satisfying the linear second-order differential constraints 
\begin{equation}
\vec{\partial}_{\I} \sprod \vec{\partial}_{\J} L = 0
\qquad\text{and}\qquad
\vec{\partial}_{\I} \vprod \vec{\partial}_{\J} L = \vec{\pt 0} \mathrlap{.}
\end{equation}
This function is associated to a choice of direction in $\mathbb{R}^3$, which we take here to be the direction of the $\mathbf{i}\pt$-axis. Accordingly, it is convenient to view the space \mbox{$\mathbb{R}^m \npt\otimes \mathbb{R}^3$} as \mbox{$\mathbb{C}^m \!\times\nhp \mathbb{R}^m$} 
and work with the complex linear combinations
$\smash{ z^{\I} = \frac{1}{2}(x^{\I}_2 + i\hp x^{\I}_3)  }$, $\smash{ x^{\I} = x^{\I}_1 }$, $\smash{ \bar{z}^{\I} = \frac{1}{2}(x^{\I}_2 - i\hp x^{\I}_3) }$
of the real coordinates $\smash{ x^{\I}_i }$. In terms of these, the differential constraints take the form
\begin{equation} \label{L_constr_LT}
L_{x^{\I}x^{\J}} = - \pt L_{z^{\I}\bar{z}^{\J}}
\qquad\text{and}\qquad
L_{x^{\I}z^{\J}} = L_{x^{\J}z^{\I}} \mathrlap{.} \phantom{--}
\end{equation}
The hyperk\"ahler geometric structure is then extracted from $L$ based on the following two properties:
\begin{itemize}
\setlength\itemsep{0.2em}

\item[1)] The (flipped-sign) Legendre transform
\begin{equation} 
\kappa(z,\bar{z},u,\bar{u}) = \langle \hp L(z,\bar{z},x) - 2 \Im u_{\I}\hp x^{\I} \hp \rangle_{\hp x}
\end{equation}
of the function $L$ with respect to the $x^{\I}$-variables, where one assumes that the Legendre constraints $\smash{L_{x^{\I}} = 2 \Im u_{\I}}$ can be implicitly solved to give unequivocally $\smash{ x^{\I} = x^{\I}(z,\bar{z},u,\bar{u})}$, yields a K\"ahler potential for the hyperk\"ahler symplectic form $\Omega_1$ along the chosen direction, together with a complete set of coordinates $z^{\I}$ and $u_{\I}$ holomorphic with respect to the corresponding complex structure~$I_1$. 

\item[2)] These coordinates are, simultaneously, complex Darboux coordinates for the trans\-ver\-sal complex symplectic form, that is,
\begin{equation} \label{Om_transv}
\Omega_+ \equiv \frac{1}{2}(\hp\Omega_2 + i \hp \Omega_3)= du_{\I} \!\wedge dz^{\I} \mathrlap{.}
\end{equation}
\end{itemize}
The two conditions suffice to determine all three hyperk\"ahler symplectic forms \mbox{$\Omega_1$, $\Omega_2$, $\Omega_3$}, and hence also the hyperk\"ahler metric $G$. And indeed, one can verify that these are of extended Gibbons--Hawking type, with
{\allowdisplaybreaks
\begin{align}
U_{\I\J} & = - \frac{1}{2} L_{x^Ix^J}  \label{U_IJ_LT} \\
A_{\I} & = \Im (L_{x^Iz^J}dz^{\J}) + d\phi_{\I} \label{A_I_LT} \\
\intertext{and also $\smash{ \psi_{\I} = \Re u_{\I} - \phi_{\I} }$, where $\smash{ \phi_{\I} }$ is an arbitrary real shift. The last relation can be combined with the Legendre constraints to give} 
u_{\I} & = \psi_{\I} + \phi_{\I} + \frac{i}{2} L_{x^{\I}} \mathrlap{.} \label{u_I-LT}
\end{align}
}%
The differential constraints \eqref{L_constr_LT} guarantee moreover that these connection 1-forms and Higgs field satisfy as required the extended Bogomolny field equations. 

This construction is general, in the sense that any hyperk\"ahler manifold of extended Gibbons--Hawking type is locally given by such a Legendre transform construction for some potential $L$ (see Proposition 2.1 in \cite{MR1704547}).

\subsection{Hyperk\"ahler cone structure conditions}

\subsubsection{}

Let us work out now how hyperk\"ahler cone structures are described in this approach, or, in other words, how the conditions of subsection~\ref{sec:GH-HKC} translate in terms of the potential $L$. 

Consider the following two complex linear combinations of the generators of the collective $\smash{ \mathbb{H}^{\times} \npt /\mathbb{Z}_2 }$-action on the space \mbox{$\mathbb{R}^m \npt\otimes \mathbb{R}^3$}:
\begin{align} \label{cplx_Ls}
L_1 + i L_0 & = i \pt (\hp 2 \bar{z}^{\I} \partial_{\bar{z}^{\I}} + x^{\I} \partial_{x^{\I}}) \\[1pt]
L_2 + i L_3 & = i \pt (\hp 2 z^{\I} \partial_{x^{\I}} - x^{\I} \partial_{\bar{z}^{\I}} ) \mathrlap{.} \nonumber
\end{align}
On the domain of definition of $L$, the differential constraints \eqref{L_constr_LT} imply the relations
\begin{align}
(L_1 + i L_0)(L_{x^{\I}}) & = - \pt (L_2 - i L_3) (L_{\bar{z}^{\I}}) \\[0pt]
(L_2 + i L_3) (L_{x^{\I}}) & = \phantom{- \pt} (L_1 - i L_0)(L_{\bar{z}^{\I}}) \mathrlap{.} \nonumber
\end{align}

Observe now, on one hand, that from the expression \eqref{U_IJ_LT} for the Higgs field we have that
{\allowdisplaybreaks
\begin{align}
(L_1 + i L_0) \hp U_{\I\J} & = - \hp \frac{1}{2}\pt \partial_{x^{\I}} \partial_{x^{\J}} [\hp (L_1 + i L_0)(L) - i L \pt] - i \hp U_{\I\J} \\[-1pt]
(L_2 - i L_3) \hp U_{\I\J} & = - \hp \frac{1}{2} \pt \partial_{z^{\I}} \hp \partial_{x^{\J}} [\hp (L_1 + i L_0)(L) - i L \pt] \mathrlap{.} \nonumber \\
\intertext{The first equality is simply an identity. To obtain the second one we have made use of the differential constraints of $L$ as follows:}
(L_2 - i L_3) \hp U_{\I\J} 
& = - \pt \frac{1}{2} (L_2 - i L_3) \npt\underbracket[0.4pt][4pt]{ L_{x^{\I}x^{\J}} }_{\underset{\displaystyle \mathllap{- \pt} L_{z^{\I}\bar{z}^{\J}} } {\scalebox{0.8}[0.65]{\verteq}} }
= \mathrlap{ \frac{1}{2} \pt \partial_{z^{\I}} \!\nhp\underbracket[0.4pt][4pt]{ (L_2 - i L_3) (L_{\bar{z}^{\J}}) }_{\underset{\displaystyle \mathllap{- \pt} (L_1 + i L_0)(L_{x^{\J}}) } {\scalebox{0.8}[0.65]{\verteq}} } } \\[-2pt]
& = - \hp \frac{1}{2} \pt \partial_{z^{\I}} \partial_{x^{\J}} [\hp (L_1 + i L_0)(L) - i L \pt]  \nonumber
\end{align}
}%
where in the last step we have used the commutation property $\smash{ [L_1 + i L_0, \partial_{x^{\J}}] = - \hp i \hp \partial_{x^{\J}} }$. Note that both equalities hold in full generality, for any hyperk\"ahler space of extended Gibbons--Hawking type. In view of these considerations the criterion of Proposition~\ref{GH-HKC} can then be equivalently rephrased as follows:

\begin{proposition} \label{L-HKC}
The collective $\smash{ \mathbb{H}^{\times} \npt /\mathbb{Z}_2 }$-action on the base of an extended Gibbons--Hawking hyperk\"ahler space $N$ with potential $L$ can be lifted to a hyperk\"ahler cone structure on $N$ if and only if $L$ satisfies the condition
\begin{equation}
L_1(L) + i\hp ( L_0(L) - L) \in \, \bigcap\limits_{\I,\J} \, [\pt \ker(\hp \partial_{x^{\I}} \partial_{x^{\J}}) \cap \ker(\hp \partial_{z^{\I}} \partial_{x^{\J}}) ] \mathrlap{.}
\end{equation}
An immediate corollary is then that a sufficient set of constraints for that to happen is given by \cite{Ionas:2008gh}
\begin{equation} \label{HKC_constr_LT}
L_1(L) = 0
\qquad\text{and}\qquad
L_0(L) =  L \mathrlap{.}
\end{equation}
\end{proposition}

On another hand, from the Legendre transform approach expression for the connection 1-forms $A_{\I}$ we obtain the following identities
{\allowdisplaybreaks
\begin{align}
\iota_{L_1 + iL_0} A_{\I} + U_{\I\J} (x^{\J}_1 + i \hp x^{\J}_0) & = (L_1 + iL_0)(\phi_{\I} + \frac{i}{2} L_{x^{\I}}) \\[-1pt]
\iota_{L_2 + iL_3} A_{\I} + U_{\I\J} (x^{\J}_2 + i \hp x^{\J}_3) & = (L_2+ iL_3)(\phi_{\I} + \frac{i}{2} L_{x^{\I}}) \nonumber
\end{align}
}%
where, of course, $\smash{x^{\J}_0  = 0}$. We stress that these are indeed identities, with no constraints whatsoever being needed for their derivation. In the light of formula \eqref{u_I-LT} it is clear then that

\begin{lemma} \label{u-kers-lem}
The connection 1-forms \eqref{A_I_LT} satisfy the gauge-fixing condition \eqref{gauge-fix} if and only if there exist shifts $\phi_{\I}$ such that
\begin{equation} \label{u-kers}
u_{\I}  \in \ker(L_1 + iL_0) \cap \ker(L_2+ iL_3)
\end{equation}
for all values of {\small $I$}.
\end{lemma}

\subsubsection{}

Let us assume now that the function $L$ satisfies the linear differential constraints \eqref{HKC_constr_LT}. One can then check that with these conditions its Legendre transform is both invariant at collective rotations and scales with weight $1$ under the action of $L_0$. That is to say, the Legendre transform gives us in this case not just a K\"ahler potential, but in fact a hyperk\"ahler potential. Accordingly, we write
\begin{equation}
\kappa = U \mathrlap{.}
\end{equation}
As we have discussed at the end of \S\,\ref{ssec:collective}, when the gauge-fixing condition \eqref{gauge-fix} holds, the collective $\smash{ \mathbb{H}^{\times} \npt /\mathbb{Z}_2 }$-action on $S$ lifts trivially to the hyperk\"ahler cone as its standard quaternionic action. In our case it can be shown that its (complexified) generators assume in holomorphic coordinates the form  
\begin{align} \label{cplx_Ls_hol}
L_1 + i L_0 & = 2 \hp i \pt \bar{z}^{\I} \frac{\partial}{\partial \bar{z}^{\I}} \\
L_2 + i L_3 & = \frac{\partial U}{\partial \bar{u}_{\I}} \frac{\partial}{\partial \bar{z}^{\I}}  - \frac{\partial U}{\partial \bar{z}^{\I}} \frac{\partial}{\partial \bar{u}_{\I}} \mathrlap{.} \nonumber
\end{align}

\subsubsection{}

So far, what we have done was simply to rephrase the description of a hyperk\"ahler cone structure on an extended Gibbons--Hawking space in the language of the Legendre transform approach. Ultimately, though, we are interested in determining the quaternionic K\"ahler metric on the base of the hyperk\"ahler cone. To achieve that, we combine the results of this and the previous section into the following procedure: 
\begin{itemize}
\setlength\itemsep{0.2em}

\item[1.] Given a function $L$ satisfying the differential conditions \eqref{L_constr_LT}, usually although not necessarily in the form of a contour integral, one should first test whether it satisfies also the conditions of Proposition~\ref{L-HKC}. In case it does, one is then guaranteed that the ensuing metric through the Legendre transform construction will be of hyperk\"ahler cone type. 

\item[2.] The second partial derivatives of $L$ with respect to the $x^{\I}$-variables give the Higgs field $U_{\I\J}$ via the formula \eqref{U_IJ_LT}. 

\item[3.] More challengingly, one must also solve the differential conditions of Lemma~\ref{u-kers-lem} to find shift functions $\phi_{\I}$ for which the connection 1-forms $A_{\I}$ satisfy the gauge-fixing condition \eqref{gauge-fix}. As we shall see in the next example, depending on the particularities of the case, additional symmetry considerations might need to be taken into account as well.

\item[4.] Once the Higgs field and the appropriately gauge-fixed connection 1-forms on the hyperk\"ahler cone  are determined, one should then substitute in them the change of variables \eqref{x_dec} for some choice of restricted configuration $\vec{\rho}^{\,\I}$. The outcomes should be amenable to the form \eqref{U_IJ_dec} and \eqref{A_K_dec}, respectively, from  which one can then read off the reduced Higgs field $\mathscr{U}_{\I\J}$ and induced connection 1-forms $\mathscr{A}_{\I}$. This is effectively a quotienting procedure which takes us from the hyperk\"ahler cone down to its base. 

\item[5.] The Ansatz of Theorem~\ref{QK-GH-analog}, or any one of the equivalent Ans\"atze listed in Table~\ref{alt_exprs}, gives us eventually the quaternionic K\"ahler metric and structure on the base. 
\end{itemize}

\section{An example: the local c-map} \label{sec:c-map}

As an application of this procedure and, more generally, of the ideas that we have developed in this article, we present an explicit construction of a quaternionic K\"ahler metric known as the Ferrara--Sabharwal metric \cite{Ferrara:1989ik}. This metric emerged in physics from the construction known as the \textit{local c-map} (or, sometimes, the \textit{projective} or \textit{supergravity c-map}) \cite{Cecotti:1988qn}, a map derived from a duality of the moduli spaces of type IIA and type IIB string theories, which yields a quaternionic K\"ahler manifold of real dimension \mbox{$4n$} for each projective special K\"ahler manifold of real dimension \mbox{$2n-2$}. The approach we pursue was initiated by Ro\v{c}ek, Vafa and Vandoren in \cite{Rocek:2005ij, Rocek:2006xb}. The idea is as follows: the $4n\pt$-dimensional quaternionic K\"ahler manifold has an associated $4n+4\pt$-di\-men\-sio\-nal Swann bundle which, in turn, as a hyperk\"ahler space, has a $4n+6\pt$-dimensional twistor space. The twistor space, these authors posit, is characterized as a complex space by a certain holomorphic gluing function (in physics terms, this is related to a projective superspace lagrangian density). The challenge is then to extract the metric from this holomorphic function by coming down first two real dimensions, to describe the geometry of the Swann bundle, and then another four real dimensions, to the quaternionic K\"ahler base. Here we fill in a number of critical details in the arguments of \cite{Rocek:2005ij, Rocek:2006xb}, which were further developed in \cite{Neitzke:2007ke}, and use our previous considerations to perform what these papers refer to as a \textit{superconformal quotient}\,---\,\textit{i.e.}~the descent from the hyperk\"ahler cone to the quaternionic K\"ahler base\,---\,explicitly as well as equivariantly.

\subsubsection{}

In close analogy with the Legendre transform construction of the rigid c-map metric (also known as the affine c-map or semi-flat metric), the $L$-po\-ten\-tial of the Swann bundle $N_c$ over the quaternionic K\"ahler image of the local c-map was shown in \cite{Rocek:2005ij, Rocek:2006xb} to be given by the following contour integral:
\begin{equation}
L = \frac{1}{2\pi} \oint \frac{d\zeta}{\zeta}  \scalebox{0.8}[1]{\bigg(} \frac{F(\eta(\zeta))}{\eta^0(\zeta)} - \frac{\bar{F}(\eta(\zeta))}{\eta^0(\zeta)}  \scalebox{0.8}[1]{\bigg)}
\end{equation}
where
\begin{itemize}
\setlength\itemsep{0.2em}

\item  the $\eta^{\I}(\zeta)$ variables are so-called tropical components of sections of an $\mathcal{O}(2)$ line bundle over the twistor space of the Swann bundle, that is,
\begin{equation*}
\eta^{\I}(\zeta) = \frac{z^{\I}}{\zeta}  + x^{\I} - \bar{z}^{\I}\zeta \mathrlap{;}
\end{equation*}

\item $\smash{ F(\eta(\zeta)) \equiv F(\eta^1(\zeta),\dots,\eta^n(\zeta)) }$ is a holomorphic function called \textit{prepotential}, assumed to be homogeneous of degree two in its variables;

\item the integration contour wraps around the two roots of $\smash{ \eta^0(\zeta) }$: anti-clockwise around $\smash{ \zeta^0_+ }$ for the first term, and clockwise around $\smash{ \zeta^0_- }$ for the second one, where 
\begin{equation*}
\zeta^0_{\pm} = \frac{x^0 \mp r^0}{2\bar{z}^0} = - \frac{2z^0}{x^0 \pm r^0} \mathrlap{,}
\end{equation*}
with $\smash{ r^0 = |\vec{x}^{\pt 0}| }$, correspond to antipodally opposite points on the twistor Riemann sphere.
\end{itemize}
The integrand can be understood as a  twistor space holomorphic symplectic gluing function \cite{MR877637}. Its holomorphic dependence on $\mathcal{O}(2)$ sections guarantees that the contour integral satisfies automatically the generalized Laplace equations \eqref{L_constr_LT}. Moreover, the specific manner in which it depends on these\,---\,namely, the fact that it only depends on $\zeta$ \textit{implicitly} through the variables $\eta^{\I}(\zeta)$ and the fact that it scales with weight $1$ at a scaling of the $\eta^{\I}(\zeta)$'s\,---\,implies that the integral satisfies the two constraints \eqref{HKC_constr_LT}, and it is thus indeed, as claimed, the $L$-potential of a hyperk\"ahler cone. 

The contour integral can be computed explicitly with Cauchy's residue theorem. Using the homogeneity property of the prepotential we obtain
\begin{equation} \label{L_c-map}
L =  2\hp r^0 \Im F(\chi)
\end{equation}
where, by definition,
\begin{equation} \label{the_chi's}
\chi^{\A} = \frac{\eta^{\A}(\zeta^0_+)}{r^0}
\qquad \text{and} \qquad
\bar{\chi}^{\A} = \frac{\eta^{\A}(\zeta^0_-)}{r^0} \mathrlap{.}
\end{equation}
Note that the twistor sections $\smash{ \zeta \mapsto \eta^{\A}(\zeta) }$ take antipodally conjugated variables to complex conjugated ones. Here and throughout this section we use the following index notation conventions:
\begin{equation} \label{index_conv}
\def\arraystretch{1.30} 
\begin{tabular}{ll}
Indices & Range \\  \hline
{\small $I$,\,$J$, \dots} & $0, \dots, n$ \\
{\small $A$,\,$B$, \dots} & $1, \dots, n$ \\[2pt] 
\end{tabular}
\end{equation}

\subsubsection{}

Before we continue, we pause for a moment to list for future reference a number of properties that these variables satisfy. First, direct evaluation gives us the expression 
\begin{equation}
\chi^{\A} = \frac{x^{\A}}{r^0} - \frac{x^0}{r^0} \Re \frac{z^{\A}}{z^0} - i  \Im \frac{z^{\A}}{z^0} \mathrlap{.}
\end{equation}
This can be equivalently recast in the vectorial form
{\allowdisplaybreaks
\begin{align}
\chi^{\A}
 & = \frac{(\hp \vec{x}^{\pt 0} \nhp\vprod \vec{x}^{\hp\A}) \sprod (\hp \vec{x}^{\pt 0} \nhp\vprod\pt \mathbf{i} \pt) - i \, |\vec{x}^{\pt 0}| \, \vec{x}^{\pt 0}  \sprod (\hp \vec{x}^{\hp\A} \nhp\vprod\pt \mathbf{i} \pt) }{|\vec{x}^{\pt 0}| | \pt \vec{x}^{\pt 0} \nhp\vprod\pt \mathbf{i} \pt|^2} \mathrlap{.} \\
\intertext{Proceeding in the same vein, we can then show that the following two properties hold:}
\chi^{\A} \bar{\chi}^{\B}
 & = \frac{(\hp \vec{x}^{\pt 0} \nhp\vprod \vec{x}^{\hp\A}) \sprod (\hp \vec{x}^{\pt 0} \nhp\vprod \vec{x}^{\hp\B}) - i \, |\vec{x}^{\pt 0}| \, \vec{x}^{\pt 0}  \sprod (\hp \vec{x}^{\hp\A} \nhp\vprod \vec{x}^{\hp\B}) }{| \pt \vec{x}^{\pt 0} \nhp\vprod\pt \mathbf{i} \pt|^2} \label{chi-chibar} \\
\frac{\chi^{\A}}{\chi^{\B}} 
 & = \frac{(\hp \vec{x}^{\pt 0} \nhp\vprod \vec{x}^{\hp\A}) \sprod (\hp \vec{x}^{\pt 0} \nhp\vprod \vec{x}^{\hp\B}) - i \, |\vec{x}^{\pt 0}| \, \vec{x}^{\pt 0}  \sprod (\hp \vec{x}^{\hp\A} \nhp\vprod \vec{x}^{\hp\B}) }{| \pt \vec{x}^{\pt 0} \nhp\vprod \vec{x}^{\hp\B}|^2} \label{chi/chi} \mathrlap{.}
\end{align}
}%
The last one, in particular, makes it clear that ratios of $\chi^{\A}$'s are invariant under collective transformations and thus descend naturally to functions on the base of the hyperk\"ahler cone. 

In order to present the next property in a concise manner it will be convenient to introduce a number of temporarily notations. Thus, if we denote the parameters of the $\mathcal{O}(2)$-bundle sections by $\smash{ \eta^{\I}_{+1} = z^{\I} }$, $\smash{ \eta^{\I}_{0} = x^{\I} }$, $\smash{ \eta^{\I}_{-1} = - \pt \bar{z}^{\I} }$ such that we have $\smash{ \eta^{\I}(\zeta) = \sum_{m=-1}^1 \eta^{\I}_m \hp \zeta^{-m} }$ (this is sometimes called the complex spherical basis), then in terms of these the various derivatives of $\smash{ \chi^{\A} }$ can be summarized in the formula
\begin{equation} \label{chi_diff}
d\chi^{\A} = \frac{1}{r^0} \sum_{m=-1}^1 (\zeta^0_+)^{-m} \Big[\hp d\eta^{\A}_m + \Big(m \chi^{\A} - \frac{\vec{x}^{\pt 0} \npt\sprod \vec{x}^{\hp\A}}{\vec{x}^{\pt 0} \npt\sprod \vec{x}^{\pt 0}} \Big) d\eta^0_m \Big] \mathrlap{.}
\end{equation}

In view of the imminent application of Lemma~\ref{u-kers-lem} it would be useful to have a list of elementary functions belonging to the intersection of kernels from the condition \eqref{u-kers}. And indeed, such a list exists and reads as follows:
\begin{equation} \label{phi-build-list}
\bigg\{ \psi_0,\, \psi_{\A},\, \frac{z^{\A}}{z^0},\, \frac{x^0 + r^0}{2r^0} \chi^{\A},\, \frac{x^0 - r^0}{2 r^0} \bar{\chi}^{\A} \bigg\} \subset \ker(L_1 + iL_0) \cap \ker(L_2+ iL_3) \mathrlap{.}
\end{equation}
For the first three elements membership in the list follows immediately from a cursory survey of the formulas \eqref{cplx_Ls}. For the last two, however, this is a little less obvious. One can verify it directly, which is rather tedious, or by observing successively that
{\allowdisplaybreaks
\begin{alignat}{2}
& (L_1 + iL_0)(\zeta^0_{\pm}) \pt = - i \hp \zeta^0_{\pm} &\qquad\qquad& (L_2 + iL_3)(\zeta^0_{\pm}) \pt = i \hp (\zeta^0_{\pm})^2 \\
\intertext{and, using these facts in the definitions \eqref{the_chi's}, that}
& (L_1 + iL_0)(\chi^{\A}) = 0  && (L_2 + iL_3)(\chi^{\A}) = i \hp \zeta^0_+ \hp \chi^{\A} \\
& (L_1 + iL_0)(\bar{\chi}^{\A}) = 0 && (L_2 + iL_3)(\bar{\chi}^{\A}) = i \hp \zeta^0_- \hp \bar{\chi}^{\A} \mathrlap{.} \nonumber
\end{alignat}
}%
The remaining gap in the verification can then be bridged by a very simple calculation. 

Lastly, let us also make a note of the following identity:
\begin{equation} \label{psi-tilde-ident}
\frac{z^{\A}}{z^0} + \frac{x^0 + r^0}{2r^0} \chi^{\A} + \frac{x^0 - r^0}{2 r^0} \bar{\chi}^{\A} = \frac{\vec{x}^{\pt 0} \npt\sprod \vec{x}^{\hp\A}}{\vec{x}^{\pt 0} \npt\sprod \vec{x}^{\pt 0}} \mathrlap{.}
\end{equation}
The manifestly invariant expression on the right-hand side, the same one which shows up in the differentiation formula \eqref{chi_diff}, will play an important role in the considerations to come.

\subsubsection{}

In order to compute the Legendre transform of $L$ one needs to compute first its partial derivatives with respect to the variables $x^{\I}$. The components with \mbox{$m=0$} of the differentiation formula \eqref{chi_diff} give us promptly, on one hand,
{\allowdisplaybreaks
\begin{align} 
& L_{x^{\A}} =  2\Im F_{\A}(\chi) \label{L_x^A_c-map} \\
\intertext{and on the other,}
& (\pt \vec{x}^{\pt 0} \npt\sprod \vec{x}^{\pt\I})  L_{x^{\I}} = x^0 L \mathrlap{.} \label{L-prop}
\end{align}
}%
This last relation can be used to determine $L_{x^0}$. With the help of these formulas one can then show that the Legendre transform of $L$ yields the hyperk\"ahler potential
{\allowdisplaybreaks
\begin{align} 
U & = - \frac{4 |z^0|^2}{r^0} \Im [\pt \bar{\chi}^{\A}F_{\A}(\chi)] \mathrlap{.} \label{U=chi_expr} \\
\intertext{In addition, by using the homogeneity property of the prepotential to replace $\smash{ F_{\A}(\chi) }$ with $\smash{ F_{\A\B}(\chi)\hp \chi^{\B} }$, and then the formula \eqref{chi-chibar}, we get also the alternative expression}
U & = \mathrlap{ - \frac{(\hp \vec{x}^{\pt 0} \vprod \vec{x}^{\hp\A}) \sprod (\hp \vec{x}^{\pt 0} \vprod \vec{x}^{\hp\B})}{|\hp\vec{x}^{\pt 0}|^3}  \Im F_{\A\B} }\label{U_c-inv}
\end{align}
}%
where it is understood that $F_{\A\B}$ refers to $F_{\A\B}(\chi)$. 

One calls \textit{dualization} the result of replacing the prepotential function $F$ in the definition of $L$ with (minus) its Legendre transform. In the ensuing formulas, this operation results in the replacement of $F(\chi)$ with $\smash{ F(\chi) - \chi^{\A}F_{\A}(\chi) }$, of $\chi^{\A}$ with $\smash{ F_{\A}(\chi) }$, and of the latter with $\smash{  -\pt\chi^{\A} }$; similar statements hold also for the complex conjugated quantities. On the other hand, the zero-indexed variables such as $x^0$, $z^0$, $\bar{z}^0$ are inert under dualization. 

The first expression above for the hyperk\"ahler potential is the one best suited for an analysis of its dualization properties. From it, one can see immediately that the hyperk\"ahler potential is self-dual. This strongly suggests that one might be able to lift dualization to a symmetry of the metric, which, we shall see, will indeed be the case. 

The second expression of the hyperk\"ahler potential brings instead to the fore its collective transformation properties. Note that since the function $\smash{ F_{\A\B}(\chi) }$ is homogeneous of degree zero, and since ratios of $\chi^{\A}$'s are invariant under collective transformations, it follows that so is the Hessian $\smash{ F_{\A\B}(\chi) }$. It is then evident that the hyperk\"ahler potential is, as required, invariant under collective rotations and scales with weight $1$ under the action of $L_0$.

\subsubsection{}

Following \cite{Neitzke:2007ke}, let us consider on $N_c$ the vector fields
{\allowdisplaybreaks
\begin{align} \label{QPIW-defs}
Q^{\A} \npt & = \frac{\partial}{\partial u_{\A}} +  \textrm{\ c.c.} 
\hspace{82pt}
I = \frac{\partial}{\partial u_0} + \textrm{\ c.c.} \\
P_{\A} & = z^0\frac{\partial}{\partial z^{\A}} - u_{\A}\frac{\partial}{\partial u_0} + \textrm{\ c.c.}  \nonumber \\
W  & = 2 z^0\frac{\partial}{\partial z_0} + z^{\A} \frac{\partial}{\partial z^{\A}} - 2 u_0 \frac{\partial}{\partial u_0} - u_{\A} \frac{\partial}{\partial u_{\A}} + \textrm{\ c.c.} \nonumber
\end{align}
}%
where \mbox{c.c.} means the complex conjugate of the preceding expression. They are manifestly real-holomorphic with respect to the hyperk\"ahler complex structure $I_1$ and, moreover, satisfy the following \textit{graded Heisenberg algebra}:
{\allowdisplaybreaks
\begin{gather} \label{H-berg}
[P_{\A},Q^{\B}] = \delta_{\A}^{\B} I \\[1pt]
[P_{\A},P_{\B}] = 0 
\qquad\qquad
[Q^{\A},Q^{\B}] = 0 \nonumber \\[1pt]
[W,P_{\A}] = P_{\A} 
\qquad\qquad
[W,Q^{\A}] = Q^{\A} 
\qquad\qquad
[W,I \pt ] = 2 I \mathrlap{.} \nonumber
\end{gather}
}%

Observe now that, out of these, the vector fields $Q^{\A}$ and $I$ are entirely vertical with respect to the $\mathbb{R}^{n+1}$-fibration structure. On the other hand,  the vector fields $P_{\A}$ and $W$ have non-trivial horizontal components given by the horizontal lifts of the vector fields 
{\allowdisplaybreaks
\begin{alignat}{2}
& \mathcal{P}_{\A} = \vec{x}^{\pt 0} \nhp\sprod \frac{\partial}{\partial \vec{x}^{\hp\A}} & \qquad \text{and} \qquad & \mathcal{W} = 2\pt \vec{x}^{\pt 0} \nhp\sprod \frac{\partial}{\partial \vec{x}^{\pt 0}} + \mathrlap{\vec{x}^{\hp\A} \nhp\sprod \frac{\partial}{\partial \vec{x}^{\hp\A}} }\\
\intertext{living on the base of the $\mathbb{R}^{n+1}$-fibration. Note incidentally that by resorting to the contour-integral representation of $L$ one can easily see that we have}
& \mathcal{P}_{\A}(L) = 0 & \qquad \text{and} \qquad & \mathcal{W}(L) = 0 \mathrlap{.} \\
\intertext{Furthermore, from the definition of the $\chi^{\A}$-variables we get successively:}
& P_{\A}(\chi^{\B}) = \mathrlap{ \mathcal{P}_{\A}(\chi^{\B}) = \frac{\mathcal{P}_{\A}(\eta^{\B}(\zeta^0_+))}{r^0} = \frac{\delta^{\B}_{\A} \, \eta^0(\zeta^0_+) }{r^0} = 0 } \\
& W(\chi^{\A}) = \mathrlap{ \mathcal{W}(\chi^{\A}) = - \hp\chi^{\A} \mathrlap{.} } \nonumber
\end{alignat}
}%
Corresponding statements hold for the complex conjugated variables as well.

In light of these facts one can then readily infer using the expression \eqref{U=chi_expr} for the hyperk\"ahler potential that this is preserved by all of the generators of the graded Heisenberg algebra: 
\begin{equation} \label{HK-pot_inv}
Q^{\A}(U) = 0 
\qquad\qquad 
I(U) = 0
\qquad\qquad
P_{\A}(U) = 0 
\qquad\qquad 
W(U) = 0 \mathrlap{.}
\end{equation}

Let us recall now the following general result: 

\begin{lemma} \label{plurih}
A real-holomorphic vector field on a K\"ahler manifold preserves the K\"ahler form if and only if it preserves locally K\"ahler potentials up to pluriharmonic functions.
\end{lemma}

\noindent A hyperk\"ahler manifold can be viewed as a K\"ahler manifold with respect to any one of its hyperk\"ahler complex structures. In our case, each one of the vector fields $Q^{\A}$, $P_{\A}$, $I$ and $W$ satisfies the conditions of this Lemma when one regards the hyperk\"ahler space $N_c$ as a K\"ahler space with respect to the complex structure $I_1$, with K\"ahler form $\Omega_1$. Indeed, they are real-holomorphic with respect to $I_1$ and, as we have shown, preserve the  hyperk\"ahler potential, which is in particular a K\"ahler potential for $\Omega_1$. By the Lemma, they all must then preserve as well the K\"ahler form $\Omega_1$. Moreover, a straightforward check using their defining formulas shows that they preserve also both the transversal symplectic form \eqref{Om_transv} and its complex conjugate. We are then entitled to conclude that $Q^{\A}$, $P_{\A}$, $I$ and $W$ are all tri-Hamiltonian vector fields. 

What is more, they also commute with the generators of the standard quaternionic action on the hyperk\"ahler cone. This can be seen using the holomorphic coordinate expressions \eqref{QPIW-defs} and \eqref{cplx_Ls_hol}, while taking also into account, once again, the invariance properties \eqref{HK-pot_inv} of the hyperk\"ahler potential function $U$. By Theorem~\ref{Sw-symm} they descend therefore to the quaternionic K\"ahler base.

These considerations suggest that in holomorphic coordinates the dualization map takes the form 
\begin{alignat}{2}
& \tilde{u}^{\A} = \frac{z^{\A}}{z^0} &\qquad\qquad& \tilde{z}_{\A} = - z^0 u_{\A} \\[-2pt]
& \tilde{u}_{0} = u_0 + \frac{z^{\A}}{z^0} u_{\A} && \tilde{z}^{0} = z^0 \mathrlap{.} \nonumber
\end{alignat}
Indeed, this represents a holomorphic symplectomorphism with respect to the complex symplectic form $\Omega_+$, which preserves moreover the graded Heisenberg algebra by mapping $Q^{\A}$ to $P_{\A}$ and the latter to $-\pt Q^{\A}$ while keeping $I$ and $W$ unchanged. Notice also that the double dual of a non-zero-indexed coordinate returns \textit{minus} that coordinate, whereas the double dual of a zero-indexed one returns it trivially, without any sign change.

\subsubsection{}

In order to be able to compute explicitly the connection 1-forms $A_{\I}$ we need to choose a set of shift functions $\phi_{\I}$ for which the conditions of Lemma~\ref{u-kers-lem} are satisfied. This is perhaps the most difficult step in the construction of a hyperk\"ahler cone in the Legendre transform approach. Fortunately, in our case we can use the duality symmetry as a guide. In other words, rather than solve these conditions directly, in what follows we shall exploit duality to find a convenient set of such functions, and then verify \textit{a posteriori} that they provide indeed a solution to the required conditions. 

Note to begin with, that in view of the expression \eqref{L_x^A_c-map}, the Legendre transform construction formula \eqref{u_I-LT} takes in this case for non-zero values of the index the form
\begin{equation}
u_{\A} = \psi_{\A} + \phi_{\A} + i \Im F_{\A}(\chi) \mathrlap{.}
\end{equation}
Dualization then yields 
\begin{equation}
\tilde{\psi}^{\A} + \tilde{\phi}^{\A} = \Re \frac{z^{\A}}{z^0} = \frac{\vec{x}^{\pt 0} \npt\sprod \vec{x}^{\hp\A}}{\vec{x}^{\pt 0} \npt\sprod \vec{x}^{\pt 0}} - \frac{x^0}{r^0} \Re \chi^{\A} 
\end{equation}
where the second equality is the real part of the identity \eqref{psi-tilde-ident}. This shows that if we take
\begin{equation} \label{shift_A}
\phi_{\A} = \frac{x^0}{r^0} \Re F_{\A}(\chi)
\qquad\text{such that}\qquad
\tilde{\phi}^{\A} = - \frac{x^0}{r^0} \Re \chi^{\A}
\end{equation}
then we get \cite{Neitzke:2007ke}
\begin{equation} \label{psi-tld}
\tilde{\psi}^{\A} = \frac{\vec{x}^{\pt 0} \npt\sprod \vec{x}^{\hp\A}}{\vec{x}^{\pt 0} \npt\sprod \vec{x}^{\pt 0}}
\end{equation}
which is manifestly invariant under collective transformations, in agreement with the expectation that dualization should commute with the hyperk\"ahler cone quaternionic action. For this choice the initial formula becomes
\begin{equation}
u_{\A} = \psi_{\A} + \frac{x^0 + r^0}{2r^0} F_{\A}(\chi) + \frac{x^0 - r^0}{2 r^0} \hp \bar{F}_{\A}(\bar{\chi}) 
\end{equation}
that is, precisely the dual of the identity \eqref{psi-tilde-ident}. Since the functions $F_{\A}$ are holomorphic and homogeneous of degree 1, then in light of the list \eqref{phi-build-list} it is clear that these $u_{\A}$'s satisfy the condition of Lemma~\ref{u-kers-lem}. This check validates the above choice of shift functions $\phi_{\A}$. 


On another hand, let us observe that the relation \eqref{L-prop} can be equivalently stated as
\begin{equation}
\Im (u_0 + \tilde{\psi}^{\A} u_{\A}) = \frac{x^0}{2(r^0)^2} \hp L \mathrlap{.}
\end{equation}
By noticing then that the function $L$ is anti-self-dual, \mbox{\textit{i.e.}} that \mbox{$\tilde{L} = - L$}, we can infer from this that  
$\smash{ \Im (u_0 + \tilde{u}_0 + \tilde{\psi}^{\A} u_{\A} - \psi_{\A} \tilde{u}^{\A}) = 0 }$.
This suggests that we choose a shift $\phi_0$ by setting
\begin{equation}
\psi_0 = \frac{1}{2} (u_0 + \tilde{u}_0 + \tilde{\psi}^{\A} u_{\A} - \psi_{\A} \tilde{u}^{\A})
\end{equation}
which is thus real and, moreover, explicitly self-dual. This is the same as having \cite{Neitzke:2007ke}
\begin{equation}
u_0 = \psi_0 + \frac{1}{2}(\psi_{\A} \tilde{u}^{\A} - \tilde{\psi}^{\A} u_{\A}) - \frac{1}{2} u_{\A} \tilde{u}^{\A} \mathrlap{.}
\end{equation}
Each elementary component of the expression on the right-hand side belongs to the set $\smash{ \ker(L_1 + iL_0) \cap \ker(L_2+ iL_3)  }$, which then by way of the Leibniz rule implies that the same is true for $u_0$. From this formula one can read off the shift $\phi_0$, which can then be shown to satisfy
\begin{equation} \label{shift_0}
\phi_0 + \tilde{\psi}^{\A} \phi_{\A} = \frac{1}{2} \scalebox{0.8}[1]{\bigg(}\nhp \frac{x^0}{r^0} \scalebox{0.8}[1]{\bigg)}^{\!\nhp 2} \npt \Re \chi^{\A} \Re F_{\A}(\chi) - \frac{1}{2} \Im \chi^{\A} \Im F_{\A}(\chi) - \frac{1}{2} \tilde{\psi}^{\A} \psi_{\A} \mathrlap{.}
\end{equation}
A rather unusual feature of this choice of shift, imposed on us by the requirements of manifest duality, is that it has an explicit linear dependence on the $\psi_{\A}$-variables.

\subsubsection{}

So far we have used two coordinate systems on $N_c$: the extended Gibbons--Hawking coordinates $\psi_{\I}$ and $\smash{ \vec{x}^{\pt\I} }$ (with a closely related complexified variant given by \mbox{$\psi_{\I}$, $x^{\I}$, $z^{\I}$, $\bar{z}^{\I}$}), and the Legendre transform-related holomorphic coordinates $u_{\I}$ and $z^{\I}$. The first system is adapted to the $\mathbb{R}^{n+1}$-fibration structure. The second one is adapted to the transversal complex symplectic structure which is holomorphic with respect to $I_1$. For the current construction, however, it is useful to introduce yet a third system of coordinates, adapted to its duality symmetry this time, given by
\begin{equation}
\{\hp \psi_0,\, x^0,\, z^0,\, \bar{z}^0,\, \psi_{\A},\, \tilde{\psi}^{\A},\, \chi^{\A},\, \bar{\chi}^{\A} \hp \} \mathrlap{.}
\end{equation}

The generators of the Heisenberg algebra take in these coordinates the form
{\allowdisplaybreaks
\begin{gather}
Q^{\A} = \frac{\partial}{\partial\psi_{\A}} + \frac{1}{2} \tilde{\psi}^{\A} \frac{\partial}{\partial \psi_0}
\qquad\qquad
P_{\A} = \frac{\partial}{\partial\tilde{\psi}^{\A}} - \frac{1}{2} \psi_{\A} \frac{\partial}{\partial \psi_0} \\
I = \frac{\partial}{\partial \psi_0} \nonumber
\end{gather}
}%
and, moreover, the grading generator $W$ scales the non-zero-indexed coordinates with weight $-1$ and the zero-indexed ones with weight $2$, except for $\psi_0$, which it scales with weight $-2$. 

The advantage of this coordinate system becomes apparent when one tries to compute the connection 1-forms $A_{\I}$. For the shift choices \eqref{shift_A} and \eqref{shift_0}, from the Legendre transform formula \eqref{A_I_LT} we obtain following a rather laborious computation the relatively simple expressions
\begin{align}
& A_{\A} = \Re F_{\A\B} d\tilde{\psi}^{\B} - \Im F_{\A\B}  \big[\hp  \Im \chi^{\B} dx^0_{\diamond} + 4 \Re \chi^{\B} \Im (\hp \bar{z}^{\hp 0}_{\hp\diamond}\hp dz^0_{\hp\diamond})  \big] \\
& A_{\hp 0} + \tilde{\psi}^{\A} A_{\A} = 2 \Im F_{\A\B} \big[\pt |z^0_{\hp\diamond}|^2 \Im(\bar{\chi}^{\A}d\chi^{\B}) + \bar{\chi}^{\A}\chi^{\B} \hp x^0_{\diamond} \pt \Im (\hp \bar{z}^{\hp 0}_{\hp\diamond}\hp dz^0_{\hp\diamond}) \big]- \frac{1}{2} d(\tilde{\psi}^{\A} \psi_{\A}) \nonumber
\end{align}
where, for simplicity, we have used the notations $\smash{ x^0_{\diamond} = x^0/r^0 }$ and $\smash{ z^0_{\hp\diamond} = z^0/r^0 }$.

\subsubsection{}

The Higgs field, on the other hand, can be computed with significantly less effort. By taking derivatives with respect to $x^{\A}$ and $x^0$ of the equation \eqref{L-prop} we obtain, respectively,
\begin{align}
(\hp \vec{x}^{\pt 0} \npt\sprod \vec{x}^{\pt\I}) \pt L_{x^{\I}x^{\A}} & = 0 \\[1pt]
(\hp \vec{x}^{\pt 0} \npt\sprod \vec{x}^{\pt\I}) \pt L_{x^{\I}x^{0}} \pt\hp & = U \nonumber
\end{align}
(recall the index notation convention \eqref{index_conv}!). Moreover, from the preceding expression for the derivative $\smash{ L_{x^{\A}} }$ we can easily compute the second derivative $\smash{ L_{x^{\A}x^{\B}} }$. Together with the relations above this gives us 
\begin{equation} \label{Higgs_c-map}
U_{\I\J} = 
- \frac{1}{r^0}
\left(\begin{array}{c|c}
R + \tilde{\psi}^{\C} \Im F_{\C\D} \tilde{\psi}^{\D} & - \pt \tilde{\psi}^{\C} \Im F_{\C\B} \\[5pt] \hline
- \pt \Im F_{\A\D} \tilde{\psi}^{\D} & \Im F_{\A\B}  \rule{0pt}{16pt}
\end{array}\right)
\end{equation}
where, by definition,
\begin{equation}
R = \frac{U}{2r^0} \mathrlap{.}
\end{equation}

\subsubsection{}

We are now ready to perform the reduction and descend from the hyperk\"ahler cone down to its quaternionic K\"ahler base. In practice, this step consists of substituting into the formulas for the Higgs field $U_{\I\J}$ and connection 1-forms $A_{\I}$ the change of variables \eqref{x_dec} for some choice of restricted configuration with position vectors $\smash{ \vec{\rho}^{\,\I} }$,  and then reading off from the result the reduced Higgs field $\mathscr{U}_{\I\J}$ and reduced connection 1-forms $\mathscr{A}_{\I}$. 

In the case of the Higgs field, this is rather straightforward. Recall, on one hand, that the variables $\smash{ \tilde{\psi}^{\A} }$ and, as argued before, the functions $\smash{ F_{\A\B}(\chi) }$ are invariant at collective transformations, and hence at the particular collective transformation represented by the equation \eqref{x_dec}. On another hand, from the multiplicative property of the quaternionic norm we have $\smash{ r^0 = |\vec{x}^{\pt 0}| = |q|^2 |\vec{\rho}^{\,\hp 0} \nhp| }$, and from the representation \eqref{U_c-inv} for the hyperk\"ahler potential, $\smash{ U = |q|^2 \mathscr{U} }$. These last two properties entail in particular that $R$ is also invariant at collective transformations, and so, like $\smash{ \tilde{\psi}^{\A} }$ and $\smash{ F_{\A\B}(\chi) }$, descends naturally on the base of the hyperk\"ahler cone. It is obvious then that if we substitute into the above Higgs field matrix the relation \eqref{x_dec} we find as expected from general arguments that this is of the form \eqref{U_IJ_dec}, with the reduced Higgs field matrix given by
\begin{equation}
\mathscr{U}_{\I\J} = 
- \frac{1}{|\vec{\rho}^{\,\hp 0} \nhp|}
\left(\begin{array}{c|c}
 R + \tilde{\psi} N \tilde{\psi} & - \pt (\tilde{\psi} N)_{\B} \\[5pt] \hline
- \pt (N \tilde{\psi})_{\A} & N_{\A\B}  \rule{0pt}{16pt}
\end{array}\right) \!\mathrlap{.}
\end{equation}
Aside from the overall factor, this matrix is precisely the same as the one in the equation \eqref{Higgs_c-map}; however, here we have introduced the shorthand notations $\smash{ N_{\A\B} = \Im F_{\A\B} }$, $\smash{ (N\tilde{\psi})_{\A} = N_{\A\B} \tilde{\psi}^{\B} }$, a.s.o.  Let us also record here for subsequent use that the inverse of this matrix is 
\begin{equation}
\mathscr{U}^{\I\J} = 
- \frac{|\vec{\rho}^{\,\hp 0} \nhp|}{R\,}
\left(\begin{array}{c|c}
1 & \tilde{\psi}^{\B} \\[5pt] \hline
\tilde{\psi}^{\A} &  \tilde{\psi}^{\A}\tilde{\psi}^{\B} + R N^{\A\B}   \rule{0pt}{16pt}
\end{array}\right)
\end{equation}
with $N^{\A\B}$ representing the matrix inverse of $ N_{\A\B}$. 

In the case of the connection 1-forms $A_{\I}$, if we consider a restricted configuration with
\begin{equation}
\vec{\rho}^{\, \hp 0} \nhp = \mbox{constant}
\end{equation}
then the change of variables \eqref{x_dec} yields, after a rather more strenuous calculation, a result of the form \eqref{A_K_dec}, with the reduced connection 1-forms given by
\begin{align}
& \mathscr{A}_{\A} = \Re F_{\A\B} \pt d\tilde{\psi}^{\B} \\
& \mathscr{A}_{\pt 0} + \tilde{\psi}^{\A} \mathscr{A}_{\A} =  \frac{1}{2} \Im F_{\A\B} \frac{\vec{\rho}^{\,\hp 0} \nhp \sprod (\hp \vec{\rho}^{\, \A} \nhp\vprod d \vec{\rho}^{\, \B})}{(\rho^0)^3} - \frac{1}{2} d(\tilde{\psi}^{\A} \psi_{\A}) \mathrlap{.} \nonumber
\end{align}

\subsubsection{}

Let us specialize now further to the particular choice of restricted configuration considered in \eqref{res-config}. For this we then have, via the formulas \eqref{psi-tld} and \eqref{chi/chi},
\begin{equation}
\tilde{\psi}^{\A} = \rho^{\A}_1
\qquad \text{and} \qquad
\frac{\chi^{\A}}{\chi^{\B}} = \frac{\rho^{\A}_2 + i \rho^{\A}_3}{\rho^{\B}_2 + i \rho^{\B}_3} 
\end{equation}
suggesting we should define also the variables
\begin{equation}
\mathcal{X}^{\A} = \frac{1}{2} (\rho^{\A}_2 + i \rho^{\A}_3) \mathrlap{.}
\end{equation}
These are all complex, with the exception of $\mathcal{X}^1$, which is real and positive. The reality condition is important for the coordinate count: in what follows we will coordinatize the quaternionic K\"ahler space by means of the variables $\psi_0$, $\psi_{\A}$, $\tilde{\psi}^{\A}$ and $\mathcal{X}^{\A}$, which would not be possible if the total number of their real components were not equal to $4n$. 

Note that since the function $F_{\A\B}$ is homogeneous of degree zero, one can argue that $\smash{ F_{\A\B}(\chi) = F_{\A\B}(\mathcal{X}) }$, which then allows us to pass from the variables $\chi^{\A}$ to the $\mathcal{X}^{\A}$ ones. With obvious notations, the hyperk\"ahler potential formula \eqref{U_c-inv} gives us in this case
\begin{equation} \label{R-exprs}
R = - 2 \hp (\bar{\mathcal{X}} N \mathcal{X}) = -2 \Im [\bar{\mathcal{X}}^{\A} F_{\A}(\mathcal{X})] \mathrlap{.}
\end{equation}
This satisfies, incidentally, the differential identity
\begin{equation} \label{whatevs}
(\theta_0 = \pt)\, \frac{dR}{2R} = \frac{\Im(\bar{\mathcal{X}} N d\mathcal{X})}{(\bar{\mathcal{X}} N \mathcal{X})}
\end{equation}
a consequence of the holomorphicity property of the prepotential.

\subsubsection{}

Knowledge of the reduced Higgs field and connection 1-forms allows us to finally compute the quaternionic K\"ahler metric explicitly. To this end, observe first that from the expression \eqref{theta_vec} for the $SO(3)$ connection 1-forms and with the above definitions we obtain in this case
{\allowdisplaybreaks
\begin{align}
\theta_1 & = - \frac{1}{2R} [ \pt \alpha + \Re (\bar{\mathcal{X}}^{\A}dF_{\A}(\mathcal{X}) - F_{\A}(\mathcal{X}) \hp d\bar{\mathcal{X}}^{\A}) ] \\
\frac{1}{2}(\theta_2+i\hp\theta_3)  & = - \frac{1}{2R} [ \pt \mathcal{X}^{\A}d\psi_{\A} + F_{\A}(\mathcal{X}) \hp d\tilde{\psi}^{\A} ] \nonumber
\intertext{where we have used the notation}
\alpha & = d\psi_{0} + \frac{1}{2} (\tilde{\psi}^{\A} d\psi_{\A} - \psi_{\A} d\tilde{\psi}^{\A}) \mathrlap{.} 
\end{align}
}%
The metric itself can be most conveniently worked out from the version \eqref{QK-quat-1} of the metric Ansatz, in which one can make use of the above formulas for the $SO(3)$ connection 1-forms. In the end, using among other things the identity \eqref{whatevs}, we arrive as predicted in \cite{Rocek:2005ij, Rocek:2006xb} at the Ferrara--Sabharwal metric \cite{Ferrara:1989ik}
\begin{equation}
2 s \pt g = g_{\scriptscriptstyle PSK} - \scalebox{0.8}[1]{\bigg(} \frac{dR}{2R} \pt\scalebox{0.8}[1]{\bigg)}^{\!\nhp 2} \! - \frac{g_{\scriptscriptstyle T}}{R\,} - \frac{\alpha^2}{R^2} 
\end{equation}
where
{\allowdisplaybreaks
\begin{align}
& g_{\scriptscriptstyle PSK} = \frac{(\bar{\mathcal{X}}N\mathcal{X})(d\bar{\mathcal{X}}Nd\mathcal{X}) - (d\bar{\mathcal{X}}N\mathcal{X})(\bar{\mathcal{X}}Nd\mathcal{X})}{(\bar{\mathcal{X}}N\mathcal{X})^2} \\
\intertext{is the so-called projective special K\"ahler metric and}
& g_{\scriptscriptstyle T} = \frac{1}{2} (\Im \tau )^{-1\A\B} (d\psi_{\A} + \tau_{\A\C} \hp d\tilde{\psi}^{\C}) (d\psi_{\B} + \bar{\tau}_{\B\D} \hp d\tilde{\psi}^{\D})
\end{align}
}%
is a complex $n$-torus metric with period matrix \cite{deWit:1984wbb}
{\allowdisplaybreaks
\begin{align}
& \tau_{\A\B} = F_{\A\B}(\mathcal{X}) - 2i \pt \frac{(N\bar{\mathcal{X}})_{\A}(N\bar{\mathcal{X}})_{\B}}{(\bar{\mathcal{X}}N\bar{\mathcal{X}})} \mathrlap{.} \\
\intertext{Under a duality transformation this undergoes the modular transformation $\smash{ \tau_{\A\B} \mapsto \! - (\tau^{-1})^{\A\B} }$ and, furthermore, we have}
& (\Im \tau )^{-1\A\B} = N^{\A\B} - \frac{\bar{\mathcal{X}}^{\A}\mathcal{X}^{\B} + \mathcal{X}^{\A}\bar{\mathcal{X}}^{\B}}{(\bar{\mathcal{X}}N\mathcal{X})} \mathrlap{.}
\end{align}
}%

\subsubsection{}

To make the junction with the usual formulation found in the literature we need to perform one final change of coordinates. Letting
\begin{equation}
Z^{\A} = \frac{\mathcal{X}^{\A}}{\mathcal{X}^1\pt}
\end{equation}
we replace the subset of coordinates $\smash{ \{ \mathcal{X}^{\A}\}_{\A = 1,\dots,n} }$ with the \textit{complex inhomogeneous coordinates} $\smash{ \{ Z^{\A^{\prime}} \nhp \}_{\A^{\prime} = 2,\dots,n} }$ plus $R$. The projective special K\"ahler metric $g_{\scriptscriptstyle PSK}$ can be understood then as the K\"ahler metric corresponding to the K\"ahler potential
\begin{equation}
\mathscr{K} = \ln \Im[\bar{Z}^{\A} F_{\A}(Z)] = \ln (\bar{Z}NZ)
\end{equation}
with $\smash{ Z^{\A^{\prime}} \!}$ as complex holomorphic coordinates. (A more geometric definition of projective special K\"ahler metrics can be found in \cite{MR1695113,Alekseevsky:1999ts}.) The formulas are straightforward to re-express in the new coordinates due to their homogeneity properties. We leave the details  to the reader as an exercise. 

The $4n$-dimensional Ferrara--Sabharwal metric has \mbox{$2n+2$} isometries descending from the corresponding tri-Hamiltonian isometries on the hyperk\"ahler cone and satisfying the same graded Heisenberg algebra \eqref{H-berg}. In the above inhomogeneous coordinate frame their generating vector fields take the form \cite{deWit:1990na} 
{\allowdisplaybreaks
\begin{gather}
\mathscr{Q}^{\A} = \frac{\partial}{\partial\psi_{\A}} + \frac{1}{2} \tilde{\psi}^{\A} \frac{\partial}{\partial \psi_0} 
\qquad\qquad
\mathscr{P}_{\A} = \frac{\partial}{\partial\tilde{\psi}^{\A}} - \frac{1}{2} \psi_{\A} \frac{\partial}{\partial \psi_0} \\[-2pt]
\mathscr{I} = \frac{\partial}{\partial \psi_0} \nonumber \\[0pt]
\mathscr{W} = - \, \psi_{\A} \frac{\partial}{\partial \psi_{\A}} - \tilde{\psi}^{\A} \frac{\partial}{\partial \tilde{\psi}^{\A}} - 2 \hp \psi_0 \frac{\partial}{\partial \psi_0} - 2 R \frac{\partial}{\partial R} \mathrlap{.} \nonumber
\end{gather}
}%

Finally, let us also mention for the sake of completeness the well-known fact that if the prepotential function $F$ is such that $\smash{ N_{\A\B} = \Im F_{\A\B}}$ has Lorentzian signature \mbox{$(1,n-1)$}, then on the open complex domain given by the condition $(\bar{Z}NZ) > 0$, where we have \mbox{$R<0$}, both $g_{\scriptscriptstyle PSK}$ and $g_{\scriptscriptstyle T}$ are negative definite (the first statement is straightforward, the second one was shown in \cite{Cremmer:1984hj}  by an ingenious argument based on the so-called inverse Cauchy-Schwarz inequality), and therefore \mbox{$s\pt g$} is negative definite. In other words, in this case one can choose the metric to be positive definite, and this metric will have a negative scalar curvature.

\subsubsection{}

The local c-map has been studied extensively in both the physics and mathematics literature. In particular, a different and very interesting geometric construction, starting from 
the \textit{rigid} c-map and based on the so-called \textit{hyperk\"ahler/quaternionic K\"ahler correspondence}, has been given in \cite{Alekseevsky:2013nua} (see also the preceding work of \cite{MR2494168} and \cite{Alexandrov:2011ac}). Global aspects of the Ferrara--Sabharwal metric have been considered in \cite{Cortes:2011aj,MR3324146}, and completeness issues have been discussed in \cite{MR3751963,Cortes:2011aj}. Further considerations such as quantum perturbative and instanton corrections have been explored in \cite{Anguelova:2004sj,RoblesLlana:2006is,Alexandrov:2011ac}.

\bigskip

\bibliographystyle{amsplain}
\bibliography{ToricQK}

\providecommand{\bysame}{\leavevmode\hbox to3em{\hrulefill}\thinspace}
\providecommand{\MR}{\relax\ifhmode\unskip\space\fi MR }
\providecommand{\MRhref}[2]{%
  \href{http://www.ams.org/mathscinet-getitem?mr=#1}{#2}
}
\providecommand{\href}[2]{#2}
\begin{thebibliography}{10}

\bibitem{MR0231313}
D.~V. Alekseevsky, \emph{Riemannian spaces with unusual holonomy groups},
  Funcktional. Anal. i Prilo\v{z}en \textbf{2} (1968), no.~2, 1--10, (English
  translation: Funct. Anal. Appl., \textbf{2}, 97--105, 1968). \MR{0231313}

\bibitem{MR1376149}
D.~V. Alekseevsky, E.~Bonan, and S.~Marchiafava, \emph{On some structure
  equations for almost quaternionic {H}ermitian manifolds}, Complex structures
  and vector fields ({P}ravetz, 1994), World Sci. Publ., River Edge, NJ, 1995,
  pp.~114--134. \MR{1376149}

\bibitem{Alekseevsky:1999ts}
D.~V. Alekseevsky, V.~Cort\'es, and C.~Devchand, \emph{{Special complex
  manifolds}}, J. Geom. Phys. \textbf{42} (2002), 85--105.

\bibitem{Alekseevsky:2013nua}
D.~V. Alekseevsky, V.~Cort\'es, M.~Dyckmanns, and T.~Mohaupt,
  \emph{{Quaternionic K\"ahler metrics associated with special K\"ahler
  manifolds}}, J. Geom. Phys. \textbf{92} (2015), 271--287.

\bibitem{MR1441871}
D.~V. Alekseevsky and S.~Marchiafava, \emph{Quaternionic structures on a
  manifold and subordinated structures}, Ann. Mat. Pura Appl. (4) \textbf{171}
  (1996), 205--273. \MR{1441871}

\bibitem{Alexandrov:2011ac}
S.~Alexandrov, D.~Persson, and B.~Pioline, \emph{{Wall-crossing, Rogers
  dilogarithm, and the QK/HK correspondence}}, JHEP \textbf{12} (2011), 027.

\bibitem{Anguelova:2004sj}
L.~Anguelova, M.~Ro\v{c}ek, and S.~Vandoren, \emph{{Quantum corrections to the
  universal hypermultiplet and superspace}}, Phys. Rev. \textbf{D70} (2004),
  066001.

\bibitem{Ballmann}
W.~Ballmann, \emph{Homogeneous structures},  (2007), Lecture notes, available
  online.

\bibitem{MR2200267}
E.~Bergshoeff, S.~Cucu, T.~de~Wit, J.~Gheerardyn, S.~Vandoren, and
  A.~Van~Proeyen, \emph{The map between conformal hypercomplex/hyper-{K}\"ahler
  and quaternionic(-{K}\"ahler) geometry}, Comm. Math. Phys. \textbf{262}
  (2006), no.~2, 411--457. \MR{2200267}

\bibitem{MR1704547}
R.~Bielawski, \emph{Complete hyper-{K}\"{a}hler {$4n$}-manifolds with a local
  tri-{H}amiltonian {$\mathbb{R}^n$}-action}, Math. Ann. \textbf{314} (1999),
  no.~3, 505--528. \MR{1704547}

\bibitem{MR660020}
C.~P. Boyer and J.~D. Finley, III, \emph{Killing vectors in self-dual,
  {E}uclidean {E}instein spaces}, J. Math. Phys. \textbf{23} (1982), no.~6,
  1126--1130. \MR{660020}

\bibitem{MR1083148}
R.~L. Bryant, S.~S. Chern, R.~B. Gardner, H.~L. Goldschmidt, and P.~A.
  Griffiths, \emph{Exterior differential systems}, Mathematical Sciences
  Research Institute Publications, vol.~18, Springer-Verlag, New York, 1991.
  \MR{1083148}

\bibitem{MR1950174}
D.~M.~J. Calderbank and H.~Pedersen, \emph{Selfdual {E}instein metrics with
  torus symmetry}, J. Diff. Geom. \textbf{60} (2002), no.~3, 485--521.
  \MR{1950174}

\bibitem{MR2225692}
D.~M.~J. Calderbank and M.~A. Singer, \emph{Toric self-dual {E}instein metrics
  on compact orbifolds}, Duke Math. J. \textbf{133} (2006), no.~2, 237--258.
  \MR{2225692}

\bibitem{Cecotti:1988qn}
S.~Cecotti, S.~Ferrara, and L.~Girardello, \emph{{Geometry of Type II
  Superstrings and the Moduli of Superconformal Field Theories}}, Int. J. Mod.
  Phys. \textbf{A4} (1989), 2475.

\bibitem{MR3751963}
V.~Cort\'{e}s, M.~Dyckmanns, and S.~Suhr, \emph{Completeness of projective
  special {K}\"{a}hler and quaternionic {K}\"{a}hler manifolds}, Special
  metrics and group actions in geometry, Springer INdAM Ser., vol.~23,
  Springer, Cham, 2017, pp.~81--106. \MR{3751963}

\bibitem{Cortes:2011aj}
V.~Cort\'es, T.~Mohaupt, and H.~Xu, \emph{{Completeness in supergravity
  constructions}}, Comm. Math. Phys. \textbf{311} (2012), 191--213.

\bibitem{Cremmer:1984hj}
E.~Cremmer, C.~Kounnas, A.~Van~Proeyen, J.~P. Derendinger, S.~Ferrara,
  B.~de~Wit, and L.~Girardello, \emph{{Vector Multiplets Coupled to
  \mbox{$N=2$} Supergravity: SuperHiggs Effect, Flat Potentials and Geometric
  Structure}}, Nucl. Phys. \textbf{B250} (1985), 385--426.

\bibitem{MR1744865}
B.~de~Wit, B.~Kleijn, and S.~Vandoren, \emph{Superconformal hypermultiplets},
  Nucl. Phys. B \textbf{568} (2000), no.~3, 475--502. \MR{1744865}

\bibitem{MR1825214}
B.~de~Wit, M.~Ro\v{c}ek, and S.~Vandoren, \emph{Hypermultiplets,
  hyper-{K}\"ahler cones and quaternion-{K}\"ahler geometry}, JHEP (2001),
  no.~2. \MR{1825214}

\bibitem{deWit:1984wbb}
B.~de~Wit and A.~Van~Proeyen, \emph{{Potentials and Symmetries of General
  Gauged \mbox{$N=2$} Supergravity: Yang-Mills Models}}, Nucl. Phys.
  \textbf{B245} (1984), 89--117.

\bibitem{deWit:1990na}
\bysame, \emph{{Symmetries of dual quaternionic manifolds}}, Phys. Lett.
  \textbf{B252} (1990), 221--229.

\bibitem{Ferrara:1989ik}
S.~Ferrara and S.~Sabharwal, \emph{{Quaternionic Manifolds for Type II
  Superstring Vacua of Calabi-Yau Spaces}}, Nucl. Phys. \textbf{B332} (1990),
  317--332.

\bibitem{MR1695113}
D.~S. Freed, \emph{Special {K}\"{a}hler manifolds}, Comm. Math. Phys.
  \textbf{203} (1999), no.~1, 31--52. \MR{1695113}

\bibitem{MR872143}
K.~Galicki, \emph{A generalization of the momentum mapping construction for
  quaternionic {K}\"ahler manifolds}, Comm. Math. Phys. \textbf{108} (1987),
  no.~1, 117--138. \MR{872143}

\bibitem{MR960830}
K.~Galicki and H.~B. Lawson, Jr., \emph{Quaternionic reduction and quaternionic
  orbifolds}, Math. Ann. \textbf{282} (1988), no.~1, 1--21. \MR{960830}

\bibitem{Gibbons:1979zt}
G.~W. Gibbons and S.~W. Hawking, \emph{{Gravitational Multi-Instantons}}, Phys.
  Lett. \textbf{78B} (1978), 430.

\bibitem{MR1663805}
G.~W. Gibbons and P.~Rychenkova, \emph{Cones, tri-{S}asakian structures and
  superconformal invariance}, Phys. Lett. B \textbf{443} (1998), no.~1-4,
  138--142. \MR{1663805}

\bibitem{MR0244913}
A.~Gray, \emph{A note on manifolds whose holonomy group is a subgroup of
  {$\textup{Sp}(n)\cdot {\rm Sp}(1)$}}, Michigan Math. J. \textbf{16} (1969),
  125--128. \MR{0244913}

\bibitem{MR887284}
N.~J. Hitchin, \emph{The self-duality equations on a {R}iemann surface}, Proc.
  London Math. Soc. (3) \textbf{55} (1987), no.~1, 59--126. \MR{887284}

\bibitem{MR2494168}
\bysame, \emph{Quaternionic {K}\"{a}hler moduli spaces}, Riemannian topology
  and geometric structures on manifolds, Progr. Math., vol. 271, Birkh\"{a}user
  Boston, Boston, MA, 2009, pp.~49--61. \MR{2494168}

\bibitem{MR877637}
N.~J. Hitchin, A.~Karlhede, U.~Lindstr\"om, and M.~Ro\v{c}ek,
  \emph{Hyper-{K}\"ahler metrics and supersymmetry}, Comm. Math. Phys.
  \textbf{108} (1987), no.~4, 535--589. \MR{877637}

\bibitem{Ionas:2008gh}
R.~A. Iona\c{s} and A.~Neitzke, \emph{{A Note on conformal symmetry in
  projective superspace}},  (2008).

\bibitem{MR1324633}
D.~Joyce, \emph{Explicit construction of self-dual {$4$}-manifolds}, Duke Math.
  J. \textbf{77} (1995), no.~3, 519--552. \MR{1324633}

\bibitem{MR0152974}
S.~Kobayashi and K.~Nomizu, \emph{Foundations of differential geometry. {V}ol
  {I}}, Interscience Publishers, a division of John Wiley \& Sons, New
  York-London, 1963. \MR{0152974}

\bibitem{MR0238225}
\bysame, \emph{Foundations of differential geometry. {V}ol. {II}}, Interscience
  Publishers John Wiley \& Sons, Inc., New York-London-Sydney, 1969.
  \MR{0238225}

\bibitem{MR0084825}
B.~Kostant, \emph{Holonomy and the {L}ie algebra of infinitesimal motions of a
  {R}iemannian manifold}, Trans. Amer. Math. Soc. \textbf{80} (1955), 528--542.
  \MR{0084825}

\bibitem{Lindstrom:1983rt}
U.~Lindstr{\"o}m and M.~Ro\v{c}ek, \emph{{Scalar Tensor Duality and
  \mbox{$N=1$}, \mbox{$N=2$} Nonlinear Sigma Models}}, Nucl. Phys.
  \textbf{B222} (1983), 285--308.

\bibitem{MR3324146}
O.~Macia and A.~Swann, \emph{Twist geometry of the c-map}, Comm. Math. Phys.
  \textbf{336} (2015), no.~3, 1329--1357. \MR{3324146}

\bibitem{Neitzke:2007ke}
A.~Neitzke, B.~Pioline, and S.~Vandoren, \emph{{Twistors and black holes}},
  JHEP \textbf{04} (2007), 038.

\bibitem{MR953820}
H.~Pedersen and Y.~S. Poon, \emph{Hyper-{K}\"ahler metrics and a generalization
  of the {B}ogomolny equations}, Comm. Math. Phys. \textbf{117} (1988), no.~4,
  569--580. \MR{953820}

\bibitem{RoblesLlana:2006is}
D.~Robles-Llana, M.~Ro\v{c}ek, F.~Saueressig, U.~Theis, and S.~Vandoren,
  \emph{{Nonperturbative corrections to 4D string theory effective actions from
  $SL(2,Z)$ duality and supersymmetry}}, Phys. Rev. Lett. \textbf{98} (2007),
  211602.

\bibitem{Rocek:2005ij}
M.~Ro\v{c}ek, C.~Vafa, and S.~Vandoren, \emph{{Hypermultiplets and topological
  strings}}, JHEP \textbf{02} (2006), 062.

\bibitem{Rocek:2006xb}
\bysame, \emph{{Quaternion-K\"ahler spaces, hyperk\"ahler cones, and the
  c-map}},  (2006).

\bibitem{MR664330}
S.~Salamon, \emph{Quaternionic {K}\"ahler manifolds}, Invent. Math. \textbf{67}
  (1982), no.~1, 143--171. \MR{664330}

\bibitem{MR1096180}
A.~Swann, \emph{Hyper-{K}\"ahler and quaternionic {K}\"ahler geometry}, Math.
  Ann. \textbf{289} (1991), no.~3, 421--450. \MR{1096180}

\bibitem{MR1423177}
K.~P. Tod, \emph{The {$\textup{SU}(\infty)$}-{T}oda field equation and special
  four-dimensional metrics}, Geometry and physics ({A}arhus, 1995), Lecture
  Notes in Pure and Appl. Math., vol. 184, Dekker, New York, 1997,
  pp.~307--312. \MR{1423177}

\bibitem{MR1045295}
R.~S. Ward, \emph{Einstein-{W}eyl spaces and {$\textup{SU}(\infty)$} {T}oda
  fields}, Class. Quant. Grav. \textbf{7} (1990), no.~4, L95--L98. \MR{1045295}

\end{thebibliography}


\Address

\end{document}